\documentclass[12pt]{article}
\usepackage[utf8]{inputenc}
\usepackage[margin=2.3cm]{geometry}
\usepackage{amsthm}
\usepackage{enumitem}
\usepackage{amsmath}
\usepackage{amsfonts}
\usepackage{ amssymb }
\usepackage{graphicx}
\usepackage{adjustbox}
\usepackage{color}
\usepackage[UKenglish]{babel}
\usepackage{hyperref}
\usepackage{tikz}
\usepackage[title]{appendix}
\usepackage{quiver}
\usepackage{thm-restate}
\usepackage{bm}
\usepackage{centernot}
\usepackage{mathtools}
\usepackage{stmaryrd}
\usepackage{titlefoot}
\usepackage{titlesec}
\titleformat*{\section}{\centering\MakeUppercase}
\titleformat*{\subsection}{\normalsize\bfseries}
\titlespacing\section{0pt}{6pt plus 4pt minus 2pt}{0pt plus 2pt minus 2pt}
\titlespacing{\subsection}{0pt}{\parskip}{-\parskip}
\titlespacing{\subsubsection}{0pt}{\parskip}{-\parskip}

\setlist[enumerate]{topsep=0pt}

\makeatletter
\def\Mapstofill@{\arrowfill@{\Mapstochar\Relbar}\Relbar\Rightarrow}
\makeatother
\usepackage{tikz}
\usetikzlibrary{arrows, automata, positioning, quotes}

\usepackage[capitalise]{cleveref}

\theoremstyle{plain}
\newtheorem{thm}{Theorem}[section]
\newtheorem{cor}[thm]{Corollary}
\newtheorem{lemma}[thm]{Lemma}

\newtheorem{prop}[thm]{Proposition}
\newtheorem{question}[thm]{Question}

\theoremstyle{definition}
\newtheorem{defn}[thm]{Definition}
\newtheorem{rmk}[thm]{Remark}
\newtheorem{exmp}[thm]{Example}

\newcommand{\comm}[1]{}

\newcommand{\Z}{\mathbb{Z}}

\newcommand{\myRed}[1]{\textbf{\textcolor{red}{#1}}}
\newcommand{\RAAG}{A_{\Gamma}}
\newcommand{\G}{A_{\phi}}

\title{\normalsize{\uppercase{\textbf{Conjugacy languages in virtual graph products}}}}
\date{}
\author{\small{\uppercase{Gemma Crowe}}}
\begin{document}
\sloppy
\maketitle

\begin{abstract}
In this paper we study the behaviour of conjugacy languages in virtual graph products, extending results by Ciobanu et. al. \cite{Ciobanu2016}.
We focus primarily on virtual graph products in the form of a semi-direct product. First, we
study the behaviour of twisted conjugacy representatives in right-angled Artin and Coxeter groups. We prove a criterion for when the language of conjugacy geodesics in group extensions is regular, which we then use to study conjugacy geodesic in virtual graph products. Finally, we prove a criterion for when the spherical conjugacy language is not unambiguous context-free for virtual graph products. We can extend this further in the case of virtual RAAGs, to show the spherical conjugacy language is not context-free.  
\end{abstract}
\unmarkedfntext{\emph{Keywords}: Right-angled Artin groups, conjugacy growth, formal languages, twisted conjugacy.}
\section{Introduction}
Graph products are a well studied class of groups in geometric group theory, which include right-angled Artin and Coxeter groups (RAAGs/RACGs). Given a finite simple graph $\Gamma$, where each vertex has a group assigned to it, we define the associated graph product $G_{\Gamma}$ as the group generated by the vertex groups, with the added relations that two elements commute if they lie in distinct vertex groups which are adjacent in the graph. 
Conjugacy in these groups has been studied in detail - for example the conjugacy problem in graph products was first solved by Green in the 90s \cite{GreenThesis}, and a linear-time solution was found for RAAGs by Crisp, Godelle and Wiest \cite{CGW}. In this paper, we focus on conjugacy in virtual graph products, which we define as groups which contain a finite index subgroup isomorphic to a graph product.
\par 
Virtual graph products are quasi-isometric to graph products by the Milnor-Schwarz Lemma \cite{OfficeHour}. While there exist many properties, such as hyperbolicity and standard growth, which are quasi-isometry invariant, there are also important properties which are not. A relevant example is given by Collins and Miller \cite{Collins1977}, who constructed groups $H \leq G$, where $H$ is of finite index in $G$, such that $H$ has solvable conjugacy problem but $G$ does not. We therefore cannot immediately assume that the conjugacy problem is solvable in virtual graph products. One area of research which has been used to look at conjugacy classes in finitely generated groups is formal language theory. The language theoretic properties of representatives of conjugacy classes have been studied for various types of groups (see \cite{Ciobanu2016}). These languages are interlinked with conjugacy growth series, and in some cases we can establish the nature of these series through a language theoretic approach. 
\par
More formally, if $G$ is a finitely generated group with generating set $X$, we define the \emph{conjugacy geodesic language}, denoted $\mathsf{ConjGeo}(G, X)$, to be the set of words $w \in X^{\ast}$ such that the word length of $w$ is smallest in its conjugacy class. The \emph{spherical conjugacy language}, denoted $\mathsf{ConjSL}(G,X)$, consists of a unique representative word for each conjugacy class in $G$ (see \cref{sec:languages in groups} for further details). By taking a formal power series over $\mathsf{ConjSL}(G,X)$, we obtain the conjugacy growth series of the group. It was shown by Ciobanu and Hermiller (Theorem 3.1, \cite{Ciobanu2013}) that for a graph product $G_{\Gamma}$, with respect to the standard generating set $X$, the language $\mathsf{ConjGeo}(G_{\Gamma}, X)$ is regular, provided each of the languages $\mathsf{ConjGeo}(G_{i}, X_{i})$ of the vertex groups of $G_{\Gamma}$ are regular. They also give an example of a RAAG, namely the free group on two generators, such that the language $\mathsf{ConjSL}(A_{\Gamma}, X)$ is not context-free, with respect to a free basis $X = \{x^{\pm 1}, y^{\pm 1}\}$. We note here that it remains open as to whether the algebraic type (i.e. rational/algebraic/transcendental) of the conjugacy growth series is a quasi-isometry invariant. In particular, the nature of formal languages has already been shown to not be quasi-isometry invariant (see \cref{language:qi}), and so to establish how conjugacy growth may behave in virtual graph products, we need to study these groups directly.
\par 
We focus primarily on virtual graph products of the form $G_{\phi} = G_{\Gamma} \rtimes_{\phi} \langle t \rangle$ (see \cref{eqn: extension}), where $\phi \in \mathrm{Aut}(G_{\Gamma})$ and $t$ are of finite order. Conjugacy in these types of group extensions is equivalent to twisted conjugacy in $G_{\Gamma}$. For a group $G$ and automorphism $\phi \in \mathrm{Aut}(G)$, we say two elements $u, v \in G$ are \emph{$\phi$-conjugate} if there exists an element $x \in G$ such that $u = \phi(x)^{-1}vx$. Twisted conjugacy is an active topic of research, and relates to proving the $R_{\infty}$ property in various types of groups, or solving the conjugacy problem in group extensions (see \cite{Thompson}, \cite{Cox2017}, \cite{Dekimpe2020}, \cite{Stein2015}). 
\par 
In a RAAG, which we denote $\RAAG$, it is well known that two cyclically reduced words are conjugate if and only if they are related by a finite sequence of cyclic permutations and commutation relations. To prove a twisted analogue of this result, we define $\psi$-cyclically reduced words in \cref{section:twisted conj}, where $\psi \in \mathrm{Aut}(\RAAG)$. Informally, these are words which do not decrease in length after performing any sequence of twisted cyclic permutations and commutation relations. These twisted cyclic permutations are similar to cyclic permutations, except we need to consider images of first and last letters of words under powers of $\psi$. 
\par 
Automorphisms of RAAGs can be split into two categories: length preserving, where the length of any geodesic word is preserved under the automorphism, and non-length preserving otherwise. Our first result gives an analogue for length preserving automorphisms as to how twisted conjugate elements are related in RAAGs. 

\begin{restatable*}{thm}{twist}\label{cor:sequence}
Let $\RAAG = \langle X \rangle$ be a RAAG, and let $\psi \in \mathrm{Aut}(\RAAG)$. Let $u,v \in X^{\ast}$ be two $\psi$-cyclically reduced words. Then $u$ and $v$ are $\psi$-conjugate if and only if $u$ and $v$ are related by a finite sequence of $\psi$-cyclic permutations, commutation relations and free reductions.
\end{restatable*}
\comm{
\vspace{-10pt}
We can extend \cref{cor:sequence} to all finite order automorphisms.
\begin{restatable*}{cor}{twistext}\label{cor:extended sequence}
Let $\RAAG = \langle X \rangle$ be a RAAG, and let $\psi \in \mathrm{Aut}(\RAAG)$ be of finite order. Let $u,v \in X^{\ast}$ be two $\psi$-cyclically reduced words. Then $u$ and $v$ are $\psi$-conjugate if and only if $u$ and $v$ are related by a finite sequence of $\psi$-cyclic permutations, commutation relations and free reduction.

\end{restatable*}
}
\vspace{-10pt}
We then consider the language of conjugacy geodesics for group extensions. The following result holds for any finite extension of a finitely generated group $N$. 
\begin{restatable*}{thm}{conjgeoextensiongeneral}\label{thm:conjgeo regular extension group}
    Let $G = N \rtimes_{\phi} \langle t \rangle$ be an extension of a group as in \cref{eqn:original split}, such that
    \begin{enumerate}
        \item[(i)] $\mathsf{Geo}(N,X)$ is regular,
        \item[(ii)]  $\phi \in \mathrm{Aut}(N)$ is length-preserving, and 
        \item[(iii)] $\mathsf{ConjGeo}_{\phi^{l}}(N,X) = \mathsf{CycGeo}_{\phi^{l}}(N,X)$ for $0 \leq l \leq m-1$.
    \end{enumerate}
    Then $\mathsf{ConjGeo}\left(G, \widehat{X}\right)$ is regular.
\end{restatable*}
\vspace{-10pt}
This allows us to prove regularity of the conjugacy geodesic language in certain types of virtual RAAGs and RACGs.

\begin{restatable*}{cor}{conjgeos}
\label{thm:ConjGeo Regular}
Let $A_{\phi}$ be a virtual RAAG of the form $A_{\phi} = A_{\Gamma} \rtimes_{\phi} \langle t \rangle$ (as in \cref{eqn: extension}), where $\widehat{X}$ is the standard generating set for $A_{\phi}$, and $\phi \in \mathrm{Aut}(\RAAG)$ is a composition of inversions. Then $\mathsf{ConjGeo}\left(A_{\phi}, \widehat{X}\right)$ is regular.
\end{restatable*}
\vspace{-10pt}
It still remains open whether \cref{thm:ConjGeo Regular} holds for all length-preserving automorphisms. We also study non-length preserving automorphisms, and find examples where regularity of conjugacy geodesics is once again preserved in certain virtual RAAGs.

\begin{restatable*}{cor}{nonlengthconjgeo}\label{cor:non length conj geo}
There exists virtual RAAGs $A_{\phi} = A_{\Gamma} \rtimes_{\phi} \langle t \rangle$ (as in \cref{eqn: extension}), where $\phi \in \mathrm{Aut}(\Gamma)$ is a finite order non-length preserving automorphism, such that the language $\mathsf{ConjGeo}\left(A_{\phi}, \widehat{X}\right)$ is regular.
\end{restatable*}
\vspace{-10pt}
It remains unclear if this statement holds for all finite order non-length preserving automorphisms. One instead can consider the language of \emph{twisted conjugacy geodesics}, which is a subset of the conjugacy geodesics in a group extension (see \cref{sec:twisted conjugacy languages}). Here regularity is not necessarily preserved. 

\begin{restatable*}{cor}{conjgeonotCF}
There exists virtual RAAGs of the form $A_{\phi} = A_{\Gamma} \rtimes_{\phi} \langle t \rangle$ (as in \cref{eqn: extension}), where $\phi \in \mathrm{Aut}(\RAAG)$ is a finite order non-length preserving automorphism, such that the language $\mathsf{ConjGeo}_{\phi}(\RAAG, X)$ is not context-free.
\end{restatable*}
\vspace{-10pt}
Next we study the spherical twisted conjugacy language for virtual graph products (see \cref{def:ConjSL}). This language, denoted $\mathsf{ConjSL}_{\phi}(G,X)$, consists of a unique representative word for each $\phi$-twisted conjugacy class in $G$. We provide an example of a RAAG with different orderings on the generating set, and different automorphisms, such that the language $\mathsf{ConjSL}_{\phi}(G,X)$ is not context-free. 

\begin{cor}(see \cref{exp: change})
    There exists RAAGs $A_{\Gamma} = \langle X \rangle$ and length-preserving automorphisms $\phi \in \mathrm{Aut}(\RAAG)$, such that $\mathsf{ConjSL}_{\phi}(\RAAG,X)$ is not context-free.
\end{cor}
\vspace{-10pt}
For the language $\mathsf{ConjSL}\left(G_{\phi}, \widehat{X}\right)$, we prove the following result which does not require any restrictions on the automorphism, unlike our results for the conjugacy geodesic language.
\begin{thm}\label{thm:conjslgraph}(see \cref{prop:new thm} and \cref{cor:graph product extension conjsl})\\
Let $G_{\phi} = G_{\Gamma} \rtimes_{\phi} \langle t \rangle$ be a virtual graph product as in \cref{eqn: extension}, where $\widehat{X}$ is the standard generating set for $G_{\phi}$, endowed with an order. For an induced subgraph $\Lambda \subseteq \Gamma$, let $G_{\Lambda}$ be the induced graph product with respect to $\Lambda$, with standard generating set $X_{\Lambda}$, with an induced ordering from $X$. 

If $\mathsf{ConjSL}(G_{\Lambda}, X_{\Lambda})$ is not regular, then the languages $
\mathsf{ConjSL}(G_{\Gamma}, X)$ and $\mathsf{ConjSL}\left(G_{\phi}, \widehat{X}\right)$ are not regular. This result also holds when we replace regular with either unambiguous context-free or context-free.
\end{thm}
\vspace{-10pt}
For RAAGs, we can extend this further. 
\begin{thm}(see \cref{thm: RAAG not CF} and \cref{cor: v RAAGs not CF})
    Let $A_{\phi}= \RAAG \rtimes_{\phi} \langle t \rangle$ be a virtual RAAG as in \cref{eqn: extension}, where $\widehat{X}$ is the standard generating set for $A_{\phi}$, endowed with an order. If $\RAAG$ is not free abelian, then the languages $\mathsf{ConjSL}(A_{\Gamma}, X)$ and $\mathsf{ConjSL}\left(A_{\phi},\widehat{X}\right)$ are not context-free.
\end{thm}
\vspace{-10pt}
As an aside, we collect results from the literature to prove another result for more general virtual RAAGs and RACGs, where the finite index subgroup is acylindrically hyperbolic.

\begin{restatable*}{cor}{acyhyp}\label{cor:acyhyp2}
Let $G$ be a group with a finite index normal subgroup $H \trianglelefteq G$, such that $H$ is an acylindrically hyperbolic RAAG (or RACG). Then $\mathsf{ConjSL}(G, X)$ is not unambiguous context-free, with respect to any generating set $X$. 
\end{restatable*}
\vspace{-10pt}
The organisation of this paper is as follows. After providing necessary definitions in \cref{Section: Def}, we define twisted conjugacy languages for finitely generated groups in \cref{sec:twisted conjugacy languages}. We then study the behaviour of twisted conjugate elements in RAAGs in \cref{section:twisted conj}, and apply these results to understand the geodesic conjugacy language in \cref{section: conj geos}. In \cref{section: conj sl}, we study the  spherical conjugacy language in virtual graph products.

\section{Preliminaries}\label{Section: Def}
All groups in this paper are finitely generated with inverse-closed generating sets. For a subset $S$ of a group, we let $S^{\pm} = S \cup S^{-1}$, where $S^{-1} = \{s^{-1} \mid s \in S \}$. In certain cases, group presentations will be stated using generating sets which are not inverse-closed, and when we pass to the relevant inverse-closed generating set, we won't explicitly mention the corresponding change in the group presentation. 
\subsection{Group extensions} 
As is standard, an extension of a group $H$ by $N$ is a group $G$ such that $N \trianglelefteq G$ and $G/N \cong H$. This can be encoded by a short exact sequence of groups
\[ 1 \rightarrow N \rightarrow G \rightarrow H \rightarrow 1.
\]
A group extension is split if and only if $G$ is a semi-direct product, that is $G \cong N \rtimes_{\alpha} H$ for some homomorphism $\alpha \colon H \rightarrow \mathrm{Aut}(N)$. For this paper we assume $H$ is a finite cyclic group, and so $N$ has finite index in $G$. It is sufficient to specify $\alpha(t)$ for a generator $t$ of $H$. We let $\phi = \alpha(t)$, and so $G$ has presentation
\begin{equation}\label{eqn:original split}
    G \cong N \rtimes_{\phi} H = \langle X \cup t \mid R(N), \; t^{m} = 1, \; t^{-1}xt = \phi(x) \; (x \in X) \; \rangle,
\end{equation}
where $N$ has presentation $\langle X \mid R(N) \rangle$, and $\phi \in \mathrm{Aut}(N)$ is of finite order such that $m$ divides the order of $\phi$. Any element of $G$ can be written in the form $g = t^{l}v$ where $v \in N$ and  $0 \leq l \leq m-1$.
\subsection{\normalsize{Twisted conjugacy}} 
Let $G$ be a group, let $u,v \in G$, and let $\phi \in \mathrm{Aut}(G)$. We write $u \sim v$ when two group elements $u,v \in G$ are conjugate. We say $u$ and $v$ are \emph{$\phi$-conjugate}, denoted $u \sim_{\phi} v$, if there exists an element $x \in G$ such that $u = \phi(x)^{-1}vx$. There is an immediate link between conjugacy in a split group extension $G \cong N \rtimes_{\phi} H$, and twisted conjugacy in the normal subgroup $N$.  

\begin{lemma}(\cite{Free-by-cyclic}, Page 4)\label{lem:Ventura}
Let $G$ be a group extension of the form \eqref{eqn:original split}. Let $t^{a}u, t^{b}v \in G$ be two elements in $G$, with $u,v \in N$, $0 \leq a,b \leq m-1$. Then $t^{a}u \sim t^{b}v$ if and only if $a=b$ and $v \sim_{\phi^{a}} \phi^{k}(u) $, for some integer $k$ where $0 \leq k \leq m-1$. 
\end{lemma}
\subsection{\normalsize{Graph products}}
Let $\Gamma$ be an undirected finite simple graph, that is, with no loops or multiple edges. We let $V(\Gamma)$ denote the vertex set of $\Gamma$, and let $E(\Gamma)$ denote the edge set of $\Gamma$. For a collection of groups $\mathcal{G} = \{G_{v} \; | \; v \in V(\Gamma) \} $ which label the vertices of $\Gamma$, the associated \emph{graph product}, denoted $G_{\Gamma}$, is the group defined as the quotient
\[ \left( \ast_{v \in V(\Gamma)} G_{v}\right) / \langle \langle st= ts, \; s \in G_{u}, \; t \in G_{v}, \; \{u,v\} \in E(\Gamma) \rangle \rangle.
\]
We let $X = \bigcup^{k}_{v=1} X_{v}$ denote the standard generating set for $G_{\Gamma}$. These groups include two well studied classes: right-angled Artin groups (RAAGs), denoted $A_{\Gamma}$, where each vertex group is isomorphic to $\Z$, and right-angled Coxeter groups (RACGs), denoted $W_{\Gamma}$, where each vertex group is isomorphic to $\Z/2\Z$. With this convention, elements commute if and only if an edge exists between the corresponding vertex groups. Some authors use the opposite convention, for example in \cite{CGW}.
\comm{
\begin{defn}
Let $\Gamma$ be a finite simple graph. The right-angled Artin group (RAAG) defined by $\Gamma$, denoted $A_{\Gamma}$, is the group with presentation
\[ \RAAG = \langle V(\Gamma) \; | \; \{s,t\} = 1 \; \text{for all} \; \{s,t\} \in E(\Gamma) \rangle
\]
Right-angled Coxeter groups (RACG) are defined similarly to RAAGs with the extra relations that every generator is of order 2. We denote $W_{\Gamma}$ to be the RACG defined by $\Gamma$.
\[ W_{\Gamma} = \langle V(\Gamma) \; | \;, s^{2} = 1 \; \text{for all} \; s \in V(\Gamma), \{s,t\} = 1 \; \text{for all} \; \{s,t\} \in E(\Gamma) \rangle
\]
\end{defn}
These groups are widely studied in geometric group theory, see \cite{Charney2007a} for further information.}

We can create a virtual graph product $G_{\phi}$ from the following split short exact sequence
\[ 1 \rightarrow G_{\Gamma} \rightarrow G_{\phi} \rightarrow H \rightarrow 1,
\]
where $G_{\Gamma} = \langle X \rangle$ is a graph product such that $G_{\Gamma}$ is a finite index subgroup in $G_{\phi}$, and $H$ is a finite cyclic group.
In this case we have
\begin{equation}\label{eqn: extension}
     G_{\phi} = G_{\Gamma} \rtimes_{\phi} \langle t \rangle = \langle s, t \mid R(G_{\Gamma}), \; t^{m} = 1, \;  t^{-1}xt = \phi(x) \; (x \in X) \; \rangle,
\end{equation}
where $\phi \in \mathrm{Aut}(G_{\Gamma})$ is of finite order. Of course not all virtual graph products arise from such a short exact sequence, but in this paper we will focus primarily on virtual graph products of the form \eqref{eqn: extension}. The exception will be in \cref{section: conj sl}, where we look at acylindrically hyperbolic cases.
\par 
For notation, if $X$ is the standard generating set for $G_{\Gamma}$, then we define $\widehat{X} = \{X,t^{\pm 1}\}$ to be the standard generating set for the virtual graph product $G_{\phi} = G_{\Gamma} \rtimes_{\phi} \langle t \rangle$. In particular, any RAAG $\RAAG$ has standard generating set $X = V(\Gamma)^{\pm 1}$. 

\comm{
\begin{defn}
A \emph{short exact sequence} of groups $N, G, H$ is a sequence of group homomorphisms
\[ 1 \rightarrow N \xrightarrow{i} G \xrightarrow{\pi} H \rightarrow 1
\]
such that
\begin{enumerate}
    \item $im(i) = ker(\pi)$
    \item $i$ is injective, i.e. $N \trianglelefteq G$
    \item $\pi$ is surjective, i.e. $H \cong G/N$
\end{enumerate}
A short exact sequence \emph{splits} if there exists a homomorphism $\gamma: H \rightarrow G$ such that $\pi \circ \gamma$ is the identity on $H$. 
\end{defn}

\begin{thm}\label{thm:SES}
A group $G$ is the semi-direct product of two groups $N$ and $H$ if and only if there exists a short exact sequence
\[ 1 \rightarrow N \xrightarrow{i} G \xrightarrow{\pi} H \rightarrow 1
\]
which splits. 
\end{thm}}
\subsection{\normalsize{Automorphisms of graph products}}\label{sec:autos RAAGs}
Let $X$ be a finite set. For any word $w \in X^{\ast}$, we let $l(w)$ denote the word length of $w$ over $X$. For a group element $g \in G$, we define the \emph{length} of $g$, denoted $| g |_{X} $, to be the length of a shortest representative word for the element $g$ over $X$ (if $X$ is fixed or clear from the context, we write $|g|$). A word $w \in X^{\ast}$ is called a \emph{geodesic} if $l(w) = |\pi(w)|$, where $\pi \colon X^{\ast} \rightarrow G$ is the natural projection (recall $l(w)$ denotes the word length of $w$ over $X$). 

\begin{defn}\label{defn: length pres}
Let $G = \langle X \rangle$. We say $\phi \in \mathrm{Aut}(G)$ is:
\begin{enumerate}
    \item[(i)] \emph{length preserving} if $|\pi(w)| = |\phi(\pi(w))|$ for any word $w \in X^{\ast}$, and
    \item[(i)] \emph{non-length preserving} otherwise.
\end{enumerate}
\end{defn}
\vspace{-5pt}
For RAAGs and RACGs, generators of the automorphism groups were classified by Servatius and Laurence \cite{Servatius1989} \cite{Laurence}, and consist of four types:
\begin{enumerate}
    \item \textbf{Inversions}: For some $x \in V(\Gamma)$, send $x \mapsto x^{-1}$, and fix all other vertices.
    \item \textbf{Graph automorphisms}: Induced by the defining graph $\Gamma$.
    \item \textbf{Partial conjugations}: Let $x \in V(\Gamma)$, let $Q \subseteq \Gamma$ be the subgraph obtained by deleting $x$, all vertices adjacent to $x$ and all incident edges. Let $P \subseteq Q$ be a union of connected components of $Q$. A partial conjugation maps $p \mapsto xpx^{-1}$ for all $p \in P$, and fixes all other vertices.
    \item \textbf{Dominating transvections}: Let $x,y \in V(\Gamma)$, $x \neq y$, and assume $y$ is adjacent to all vertices in $\Gamma$ which are adjacent to $x$. A dominating transvection maps $x \mapsto xy^{\pm 1}$ or $x \mapsto y^{\pm 1}x$, and fixes all other vertices.
\end{enumerate}
\comm{
\subsection*{Automorphisms of RAAGs}
For any vertex $v \in V(\Gamma)$, we define the link of a vertex $Lk(v)$ as
\[ Lk(v) = \{ x \in V(\Gamma) \; | \; \{v,x\} \in E(\Gamma) \}
\]
Similarly we define the star of a vertex $St(v)$ as 
\[ St(v) = Lk(v) \cup \{v\}
\]
The automorphism group of a RAAG is generated by 4 types of automorphism:
\begin{enumerate}
    \item \textbf{Inversions:} $\phi: x \mapsto x^{-1}$
    \item \textbf{Graph automorphism:} Induced by any graph automorphism of $\Gamma$.
    \item \textbf{Dominated Transvections:} $x, y \in \Gamma$, $x \neq y$, $Lk(x) \subseteq St(y)$:
    \[\phi: x \mapsto xy^{\pm 1} \quad \text{or} \quad x \mapsto y^{\pm 1}x
    \]
    \item \textbf{Locally inner automorphism:} $x \in \Gamma$, $P \subseteq \Gamma \setminus St(x)$ is a union of connected components:
    \[ \phi: y \mapsto xyx^{-1} \quad \text{for all} \; y \in P
    \]
\end{enumerate}
This was conjectured by Servatius \cite{Servatius1989} in 1989, and proven for all possible RAAGs by Laurence six years later \cite{Laurence}. We say $\phi \in \mathrm{Aut}(\RAAG)$ is:
\begin{itemize}
    \item \emph{length preserving} if $l(w) = l(\phi(w))$ for any word $w \in X^{\ast}$
    \item \emph{non-length preserving} otherwise.
\end{itemize}
\vspace{-5pt}
For RAAGs, it is clear that inversions and graph automorphisms are length preserving, while dominated transvections and locally inner automorphisms are non-length preserving.} The following proposition is straightforward to show by the classification given above. 
\begin{prop}\label{prop:combo len p}
Let $\psi \in \mathrm{Aut}(\RAAG)$ be a length preserving automorphism. Then $\psi$ is a finite composition of inversions and graph automorphisms, and hence has finite order.
\end{prop}

\comm{
\begin{proof}
Let $\psi = \varphi_{1}\varphi_{2}\dots \varphi_{n}$ where each $\varphi_{i}$ is a generating automorphism. We want to show that each $\varphi_{i}$ has to necessarily be either an inversion or a graph automorphism. We proceed by induction on $n$. \\
Base case: $n=1$. This follows immediately since the only generating automorphisms which are length preserving are the inversions and graph automorphisms.\\
Assume statement holds for all $n$. Then if 
\[ \psi = \varphi_{1}\varphi_{2}\dots \varphi_{n}
\]
and $X$ is any letter, then $\psi(X)$ maps to a single letter. Now consider
\[ \widehat{\psi} = \psi\varphi_{n+1} = \varphi_{1}\varphi_{2}\dots \varphi_{n}\varphi_{n+1}
\]
Then
\[ \widehat{\psi}(X) = \varphi_{n+1}(\psi(x))
\]
By the inductive hypothesis, in order for $\widehat{\psi}$ to be length preserving, $\varphi_{n+1}$ must be length preserving. Since it is a generating automorphism, this implies $\varphi_{n+1}$ must be either an inversion or graph automorphism. This completes the proof.

\end{proof}}

\subsection{\normalsize{Languages and growth series}}\label{sec:languages}
We recall some basic notions from formal language theory, and refer the reader to \cite{Epstein1992} and \cite{Hopcroft} for further details.

\begin{defn}
For a finite set $X$, we define a \emph{language} $L$ to be any subset of $X^{\ast}$, where $X^{\ast}$ is the set of all words over $X$. A \emph{finite state automaton} over $X$ is a tuple $M = (Q, X, \mu, A, q_{0})$, where $Q$ is the finite set of states, $\mu \colon Q \times X \rightarrow Q$ the transition function, $A \subseteq Q$ the set of accept states, and $q_{0} \in Q$ the start state. A word $w = x_{1}\dots x_{n} \in X^{\ast}$, where $x_{i} \in X$ for all $1 \leq i \leq n$, is \emph{recognised} by a finite state automaton if $\mu(\dots (\mu(\mu(q_{0}, x_{1}), x_{2}) \dots, x_{n}) \in A$. A language $L$ is \emph{regular} if and only if $L$ is recognised by some finite state automaton.
\end{defn}

\comm{
\begin{defn}
A \emph{finite state automaton} is a tuple $M = (Q, X, \mu, A, s_{0})$ where:
\vspace{-10pt}
\begin{enumerate}[(1)]
    \item $Q$ is a finite set, called the \emph{state set},
    \item $X$ is a finite set, called the \emph{alphabet},
    \item $\mu: Q \times X \rightarrow Q$ is a function, called the \emph{transition function},
    \item $A$ is a (possibly empty) subset of $Q$, called the \emph{accept states}, and
    \item $q_{0} \in Q$ is called the \emph{start state} 
\end{enumerate}
\end{defn}
\vspace{-8pt}
For a finite set $X$, we define a \emph{language} $L$ to be any subset of $X^{\ast}$, where $X^{\ast}$ is the set of all possible words over $X$. We say a word $w = w_{1}\dots w_{n} \in X^{\ast}$ is \emph{recognised} by a finite state automaton if there exists a set of states $q_{1}, \dots, q_{n} \in Q$, and a sequence of transitions in an automaton such that
\[\begin{tikzcd}
	{q_{0}} & {q_{1}} & {q_{2}} & \dots & {q_{n},}
	\arrow["{w_{1}}", from=1-1, to=1-2]
	\arrow["{w_{2}}", from=1-2, to=1-3]
	\arrow["{w_{3}}", from=1-3, to=1-4]
	\arrow["{w_{n}}", from=1-4, to=1-5]
\end{tikzcd}\]
where $q_{n} \in A$ is an accept state. A language $L$ is \emph{regular} if and only if $L$ is recognised by some finite state automaton.
\par}
\vspace{-5pt}
Informally, a pushdown automaton is a finite state automaton combined with a `stack': this stores a string of symbols, and acts as a way of remembering information. A language is $\emph{context-free}$ if and only if it is accepted by a pushdown automaton. Alternatively, a language is context-free if it is generated by a context-free grammar. If this grammar produces each word in a unique way, we say the language is \emph{unambiguous context-free}.
\par 
We recall the following properties of regular and context-free languages. 
\begin{lemma}\label{lem: closure}
Let $L, L^{'}$ be regular languages over a finite alphabet $X$. Then the languages $L^{*}$ (Kleene closure), $X^{*} \setminus L$, $LL^{'}$, $L \cap L^{'}$ and $L \cup L^{'}$ are also regular, i.e. regularity is closed under Kleene closure, complement, concatenation, intersection and union. Similarly, context-free languages are closed under Kleene closure, concatenation and union. 
\end{lemma}
\vspace{-10pt}
Note that context-free languages are not closed under complement.
\begin{lemma}\label{lem:homs reverse}(\cite{Hopcroft}, Theorem 7.24 and 7.25)\\
Let $h \colon X^{\ast} \rightarrow X^{\ast}$ be a monoid homomorphism. Then if $P$ is a context-free language over $X$, so is $h(P)$, i.e. context-free languages are closed under monoid homomorphisms.
\comm{
Let $L^{R}$ denote the language consisting of all words in $L$ in reverse. Then if $L$ is context-free, so is $L^{R}$.}

\end{lemma}
\vspace{-5pt}
The following result can be shown using the fact that if $L$ is context-free, then the language $L^{R}$, consisting of all words in $L$ in reverse, is context-free.
\begin{lemma}\label{prop:suffixes}
If $L$ is a context-free language, then so is 
\[ \mathsf{suff}(L) = \{ u \in X^{*} \; | \; vu \in L \; \text{for some} \; v \in X^{\ast} \}.
\]
\end{lemma}
\comm{
\begin{proof}
First note that the set of prefixes $\mathsf{pref}(L) = \mathsf{suff}(L)^{R}$, so it remains to show that the set of prefixes is context-free by \cref{lem:homs reverse}. First we take a copy of our alphabet and denote as:
\[ \overline{X} = \{\overline{a} \; | \; a \in X^{*} \}
\]
Define the following morphisms
\[ g,h: (X \cup \overline{X}^{*}) &\rightarrow X^{*}
\]
\[ g(a) = a, \quad g(\overline{a}) = \epsilon \quad h(a) = a \quad h(\overline{a}) = a
\]
Then 
\[ \mathsf{pref}(L) = g\left(h^{-1}\left(L\right) \cap X^{*}\overline{X}^{*}\right)
\]
The result then follows by \cref{lem: closure} and \cref{lem:homs reverse}. 

\end{proof}}
\begin{lemma}\label{lem:intersection}(\cite{Hopcroft}, Theorem 7.27)\\
Let $L$ be a regular language over a finite alphabet $X$.
\begin{enumerate}
    \item If $N$ is an unambiguous context-free language over $X$, then $N \cap L$ is unambiguous \\ context-free.
    \item If $P$ is a context-free language over $X$, then $P \cap L$ is context-free.
\end{enumerate}
\end{lemma}
\begin{defn}
A formal power series $f(z) \in \Z[[z]]$ is \emph{rational} if there exist non-zero polynomials $p(z), q(z) \in \Z[[z]]$ such that $f(z) = \frac{p(z)}{q(z)}$. We say $f(z)$ is \emph{algebraic} over $\mathbb{Q}(z)$ if there exists a non-trivial polynomial $p(z,u) \in \mathbb{Q}(z,u)$ such that $p(z, f(z)) = 0$. If $f(z)$ is not algebraic, then $f(z)$ is \emph{transcendental}.
\end{defn}
\vspace{-10pt}
Any language $L \subseteq X^{\ast}$ gives rise to a \emph{strict growth function}, defined as 
\[ \phi_{L}(n) := | \{ w \in L \; | \; l(w) = n  \}|.
\]
The generating series associated to this function, defined as 
\[ f_{L}(z) := \sum^{\infty}_{i=0} \phi_{L}(i)z^{i},
\]
is known as the \emph{strict growth series}. It is well known that if $L$ is a regular language, then the series $f_{L}(z)$ is rational \cite{Meier2008}. Another key link between formal languages and growth series comes from the following result:
\begin{thm}(Chomsky-Sch\"{u}tzenberger, \cite{Chomsky1963})\\ 
Let $L \subset X^{\ast}$ be a language. If $L$ is unambiguous context-free, then the series $f_{L}(z)$ is algebraic.
 
\end{thm}
\vspace{-5pt}
By taking the contrapositive, we immediately see that if the series $f_{L}(z)$ is transcendental, then the language $L$ is not unambiguous context-free. 
\subsection{Languages in groups}\label{sec:languages in groups}
For a group $G = \langle X \rangle$ and words $u, v \in X^{\ast}$, we use $u = v$ to denote equality of words, and $u =_{G} v$ to denote equality of the group elements represented by $u$ and $v$.

\begin{defn}\label{defn:cyclically geodesic}
    Let $G = \langle X \rangle$ and let $v \in X^{\ast}$ be a geodesic. We say $v$ is \emph{cyclically geodesic} if there does not exist a sequence $v=v_{0}, v_{1}, \dots, v_{n}$ such that $l(v_{n}) < l(v)$ and $\pi(v_{i+1}) = \pi(v'_{i})$, where $v'_{i}$ is a cyclic permutation of $v_{i}$.
\end{defn}

\begin{defn}\cite[Section 2.2.1]{CGW}\label{defn:CR RAAGS}
Let $\RAAG = \langle X \rangle$, and let $v \in X^{\ast}$ be a geodesic. We say $v$ is \emph{cyclically reduced} if there does not exist a sequence of cyclic permutations, commutation relations and free reductions to a geodesic word $w \in X^{\ast}$, such that $l(v) > l(w)$.
\end{defn}

\begin{rmk}
    Note that Definitions \ref{defn:cyclically geodesic} and \ref{defn:CR RAAGS} coincide when our group $G$ is a RAAG. To avoid confusion with the literature, we will refer to `cyclically geodesic' words for general groups, and `cyclically reduced' words for RAAGs. 
\end{rmk}

\begin{defn}\label{def:cyclic shift}
Let $X$ be a finite set, and let $w = x_{1}\dots x_{n} \in X^{\ast}$, where $x_{i} \in X$ for all $1 \leq i \leq n$. We say $w$ is \emph{freely reduced} if there does not exist any $i$, for $1 \leq i < n$, such that $x_{i} = x^{-1}_{i+1}$. We define a \emph{cyclic permutation} or $\emph{cyclic shift}$ of $w$ to be any word of the form
\[ w' = x_{i+1}\dots x_{n}x_{1}\dots x_{i},
\]
for some $i \in \{1, \dots, n-1 \}$.  
\end{defn}
\vspace{-10pt}

\begin{defn}\label{def: languages}
We define the language of geodesic words of $G$, with respect to $X$, as
\[ \mathsf{Geo}(G,X) := \{ w \in X^{\ast} \; | \; l(w) = |\pi(w)| \}.
\]
We let $[g]_{c}$ denote the conjugacy class of $g \in G$. We define the \emph{length up to conjugacy} of an element $g \in G$ by
\[ |g|_{c} := \text{min}\{ |h| \; | \; h \in [g]_{c} \}.
\]
We now consider languages associated to conjugacy classes. We define the conjugacy geodesic and cyclic geodesic languages of $G$, with respect to $X$, as follows:
\begin{align*}
   \mathsf{ConjGeo}(G, X) &:= \{ w \in X^{\ast} \; | \; l(w) = |\pi(w)|_{c} \}, \\
    \mathsf{CycGeo}(G, X) &:= \{ w \in X^{\ast} \; | \; w \; \text{is a cyclic geodesic} \},
\end{align*}
where a cyclic geodesic is a word such that all cyclic permutations are geodesic. Given an order on $X$, let $\leq_{sl}$ be the induced shortlex ordering of $X^{\ast}$. For each conjugacy class $c$, we define the \emph{shortlex conjugacy normal form} of $g$ to be the shortlex least word $z_{c}$ over $X$ representing an element of $c$. We define the shortlex conjugacy language, for $G$ over $X$, as
\[ \mathsf{ConjSL}(G,X) := \{z_{c} \; | \; c \in G/\sim \}.
\]
\end{defn}
\vspace{-10pt}
These languages,  with respect to a generating set $X$ of $G$, satisfy
\[ \mathsf{ConjSL}(G, X) \subseteq \mathsf{ConjGeo}(G,X) \subseteq \mathsf{CycGeo}(G,X) \subseteq  \mathsf{Geo}(G,X).
\] 
We collect a few results from the literature.
\begin{prop}\label{language:qi}
The property of being a regular language is not a quasi-isometry invariant for
\begin{enumerate}
    \item[(i)] (\cite{Neumann1995}, Page 268) $\mathsf{Geo}(G,X)$,
    \item[(ii)] (\cite{Ciobanu2016}, Propositions 5.3 and 5.4) $\mathsf{ConjGeo}(G,X)$, and
    \item[(iii)] (\cite{Ciobanu2016}, Propositions 5.1 and 5.2) $\mathsf{ConjSL}(G,X)$.
\end{enumerate}
\end{prop}
Cannon \cite{Neumann1995} provided an example of a virtually abelian group $G = \Z^{2} \rtimes \Z / 2\Z$, where the $\Z / 2\Z$ action swaps the generators of $\Z^{2}$, such that the language $\mathsf{Geo}(G,X)$ could be either regular or not depending on the generating set. In Section 5 of \cite{Ciobanu2016}, further examples of virtually abelian groups were studied, and in certain cases the languages $\mathsf{ConjGeo}(G,X)$ and $\mathsf{ConjSL}(G,X)$ could be either regular or not depending on the generating set.
\par 
Notice that $\mathsf{ConjSL}(G,X)$ counts the number of conjugacy classes of $G$ whose smallest element, with respect to the shortlex ordering, is of length $n$, for any $n \geq 0$. We denote $\tilde{\sigma} = \tilde{\sigma}(G,X)$ to be the strict growth series of $\mathsf{ConjSL}(G,X)$. This is also known as the \emph{spherical conjugacy growth series.} The nature of this series has been studied for many types of groups. \cref{prop:all gen} collects the only known results so far which hold for any generating set.
\begin{prop}\label{prop:all gen}
Let $G $ be a finitely generated group.
\begin{enumerate}
    \item[(1)] \cite{Evetts2019} If $G$ is virtually abelian, then $\tilde{\sigma}$ is rational with respect to any generating set.
    \item[(2)] \cite{Antolin2016a} If $G$ is hyperbolic and not virtually cyclic, then $\tilde{\sigma}$ is transcendental with respect to any generating set. 
\end{enumerate}
\end{prop}
It has been conjectured that the only groups with rational $\tilde{\sigma}$ are virtually abelian groups \cite[Conjecture 7.2]{CiobanuE2020}.

\section{Twisted conjugacy languages}\label{sec:twisted conjugacy languages}
We now define twisted languages analogous to \cref{def: languages}, with respect to any automorphism $\psi \in \mathrm{Aut}(G)$. 

\begin{defn}\label{defn:cyc perm}
Let $G = \langle X \rangle$, and let $w = x_{1}\dots x_{n} \in X^{\ast}$ be a geodesic, where $x_{i} \in X$ for all $1 \leq i \leq n$. Let $\psi \in \mathrm{Aut}(G)$ be of finite order $m$. We define a \emph{$\psi$-cyclic permutation} of $w$ to be any word of the form
\[ w' = \psi^{k}(x_{i+1})\dots \psi^{k}(x_{n})\psi^{k-1}(x_{1})\dots \psi^{k-1}(x_{i}),
\]
for some $0 \leq k \leq m-1$.  

\end{defn}
\vspace{-10pt}
We note that if $\psi$ is the trivial automorphism, a $\psi$-cyclic permutation is equivalent to a cyclic permutation (see \cref{def:cyclic shift}). 

\begin{defn}
Let $G = \langle X \rangle$, let $\psi \in \mathrm{Aut}(G)$ and let $w \in X^{\ast}$ be a geodesic. We say $w$ is a \emph{$\psi$-cyclic geodesic} if all $\psi$-cyclic permutations of $w$ are geodesic. 
\end{defn}

\begin{defn}\label{defn:twisted conjugacy languages}
Let $G = \langle X \rangle$ be a group. Let $[g]_{\psi}$ denote the twisted conjugacy class of a group element $g \in G$. We define the \emph{length up to twisted conjugacy} of an element $g \in G$ as
\[ |g|_{\psi} := \text{min}\{|h| \mid h \in [g]_{\psi} \}.
\]
We define the following languages for a group $G$, with respect to a generating set $X$:
\begin{align*}
    \mathsf{ConjGeo}_{\psi}(G, X) &:= \{ w \in X^{*} \; | \: l(w) = |\pi(w)|_{\psi}\}, \\
    \mathsf{CycGeo}_{\psi}(G, X) &:= \{ w \in X^{*} \mid w \; \text{is a} \; \psi-\text{cyclic geodesic} \}.
\end{align*}
\end{defn}
\begin{lemma}\label{lemma:subset twisted conj languages}
    Let $G = \langle X \rangle$, and let $\psi \in \mathrm{Aut}(G)$ be length-preserving. Then 
    \[ \mathsf{ConjGeo}_{\psi}(G,X) \subseteq \mathsf{CycGeo}_{\psi}(G, X).
    \]
\end{lemma}

\begin{proof}
If $w$ is a $\psi$-conjugacy geodesic, then so is every $\psi$-cyclic permutation of $w$, since $\psi$ is length-preserving.
\end{proof}

\begin{defn}\label{def:ConjSL}
Let $G = \langle X \rangle$ and let $\psi \in \mathrm{Aut}(G)$. Given an order on $X$, let $\leq_{SL}$ be the induced shortlex ordering of $X^{\ast}$. For each twisted conjugacy class $c \in G/\sim_{\psi}$, we define the \emph{shortlex twisted conjugacy normal form} of $c$ to be the shortlex least word $z_{c}$ over $X$ representing an element of $c$, i.e. $z_{c}$ is shortlex minimal length in $\pi^{-1}(c)$. We define the \emph{shortlex twisted conjugacy language}, for $G$ with respect to $X$, as
\[ \mathsf{ConjSL}_{\psi}(G,X) := \{z_{c} \mid c \in G/\sim_{\psi} \}.
\]
\end{defn}
Note for any extension $G = N \rtimes_{\phi} \langle t \rangle$ of the form in \cref{eqn:original split}, we can write
\[ \mathsf{ConjSL}\left(G, \widehat{X}\right) = \bigcup_{l=a}^{b}\{Ut^{l} \mid U \in \mathsf{ConjSL}_{\phi^{l}}(N, X) \},
\]
where $x < t < t^{-1}$ for all $x \in X$, and
\[ (a,b) = \begin{cases}
    (-(k-1), k), & m = 2k, \; k \in \Z_{>0}, \\
    (-k, k), & m = 2k+1, \; k \in \Z_{\geq 0}.
\end{cases}
\]
\section{Twisted conjugacy in RAAGs}\label{section:twisted conj}
In this section we prove \cref{cor:sequence}, which is a twisted analogue of the well known fact that two cyclically reduced elements in a RAAG are conjugate if and only if they are related by a finite sequence of cyclic permutations and commutation relations \cite{CGW}. We use similar techniques to \cite{Bardakov2005}, where the authors studied twisted conjugacy in free groups. 
\par 
Let $\RAAG = \langle X \rangle$, and let $v \in X^{\ast}$ be a geodesic, where $l(v) > 1$. Then we can write $v$ in the form $v = xy$ for some non-empty geodesic words $x,y \in X^{\ast}$. We say $x$ is a \emph{proper prefix} of $v$, and $y$ is a \emph{proper suffix} of $v$. We let $\mathcal{P}(v)$ and $\mathcal{S}(v)$ be the set of all possible proper prefixes and suffixes of $v$ respectively. 

\begin{defn}\label{def:shifts moves}
Let $G = \langle X \rangle$, and let $v \in X^{\ast}$ be a geodesic, where $l(v) > 1$. Then we can write $v$ in the form $v = xy$ for some non-empty geodesic words $x,y \in X^{\ast}$. We say $x$ is a \emph{proper prefix} of $v$, and $y$ is a \emph{proper suffix} of $v$. We let $\mathcal{P}(v)$ and $\mathcal{S}(v)$ be the set of all possible proper prefixes and suffixes of $v$ respectively. 

Let $\psi \in \mathrm{Aut}(G)$. We define a \emph{$\psi$-cyclic shift of a prefix} of $v$ to be the following operation on the word $v$:
\[ v = xy \xmapsto{\psi-\text{cyclic shift}}   y\cdot \psi^{-1}\left(x\right).
\]
Similarly we define a \emph{$\psi$-cyclic shift of a suffix} of $v$ as
\[ v = xy \xmapsto{\psi-\text{cyclic shift}}   \psi\left(y\right)\cdot x.
\]

\end{defn}
\comm{
\begin{rmk}
We note that the Geod notation is unnecessary for length-preserving automorphisms, since $\text{Geod}\left(\psi\left(v\right)\right) = \psi\left(v\right)$ for any geodesic $v \in X^{\ast}$. From now on, we will assume $\psi \in \mathrm{Aut}(\RAAG)$ is length-preserving, and remove the Geod notation. This notation will be introduced again in \cref{sec:non length p}, where we will generalise \cref{cor:sequence} to all finite order automorphisms.
\end{rmk}}
\vspace{-10pt}
For brevity, we will use $\xleftrightarrow{\psi}$ to denote a $\psi$-cyclic shift of either a prefix or suffix. The double arrow is necessary here since any $\psi$-cyclic shift can be reversed, i.e.
\[ xy \xmapsto{\psi-\text{cyclic shift}}   y\cdot \psi^{-1}\left(x\right) \xmapsto{\psi-\text{cyclic shift}} \psi\left(\psi^{-1}\left(x\right)\right)\cdot y = xy.
\]
Here is an example of how this operation works in practice.
\begin{exmp}\label{exmp: RAAG F2}
Let $A_{\Gamma} = F_{2} \times F_{2}$, and label the vertices of the defining graph as follows: 
\[\begin{tikzcd}
	a & b \\
	d & c
	\arrow[no head, from=1-1, to=1-2]
	\arrow[no head, from=1-2, to=2-2]
	\arrow[no head, from=2-2, to=2-1]
	\arrow[no head, from=1-1, to=2-1]
\end{tikzcd}\]
Let $\psi \colon a \rightarrow b \rightarrow c \rightarrow d$ be a graph automorphism. Consider the geodesic word $w = ac^{-1}d$. We can compute the $\psi$-cyclic shift of the first and last letters as follows:
\[ w = ac^{-1}d \xleftrightarrow{\psi} \psi(d)ac^{-1} = a^{2}c^{-1}, \quad w = ac^{-1}d \xleftrightarrow{\psi} c^{-1}d\psi^{-1}(a) = c^{-1}d^{2}.
\]
Note that $w =_{\RAAG} adc^{-1}$, so we could also compute the following $\psi$-cyclic shift:
\[  w =_{\RAAG} adc^{-1} \xleftrightarrow{\psi} \psi(c^{-1})ad = d^{-1}ad.
\]
\end{exmp}
\begin{defn}
Let $G = \langle X \rangle$, let $v \in X^{\ast}$ be a geodesic and $\psi \in \mathrm{Aut}(G)$. We define the set of all $\psi$-cyclic shifts of $v$ to be the following set:
\[ \psi[v] = \{ y\psi^{-1}(x), \; \psi(y)x \mid v = xy \; \text{for all possible} \; x \in \mathcal{P}(v), y \in \mathcal{S}(v)\}.
\]
\end{defn}
\vspace{-10pt}
Recall that two words $u,v \in X^{\ast}$ representing groups elements of $G$ are $\psi$-conjugate if there exists $w \in X^{\ast}$ such that $u =_{G} \psi(w)^{-1}vw$.

\begin{lemma}\label{lem:twisted}
Let $G = \langle X \rangle$, and let $v \in X^{\ast}$ be a geodesic. Then $v$ is $\psi$-conjugate to all elements $w \in \psi[v]$. 
\end{lemma}

\begin{proof}
Write $v = xy$ for some $x \in \mathcal{P}(v), \; y \in \mathcal{S}(v)$. Consider the words $y\psi^{-1}\left(x\right)$, $\psi\left(y\right)x \in \psi[v]$. The result follows by the following relations:
\[ \psi\left(\psi^{-1}\left(x\right)\right)\cdot y\psi^{-1}\left(x\right)\cdot \psi^{-1}\left(x\right)^{-1} =_{G} x\cdot y\psi^{-1}\left(x\right)\cdot \psi^{-1}\left(x\right)^{-1} =_{G} xy = v.
\]
\[ \psi\left(y\right)^{-1}\cdot \psi\left(y\right)x\cdot y  =_{G} xy = v. \]
\end{proof}
\vspace{-10pt}
Recall \cref{defn:cyc perm}. A $\psi$-cyclic permutation is equivalent to taking successive $\psi$-cyclic shifts of letters in $w$, including taking higher powers of $\psi$. By \cref{lem:twisted}, we immediately have the following result.

\begin{cor}\label{cor:twisted perms}
Let $G = \langle X \rangle$, and let $w \in X^{\ast}$ be a geodesic. Let $w' \in X^{\ast}$ be a $\psi$-cyclic permutation of $w$. Then $w$ is $\psi$-conjugate to $w'$. 
\end{cor}
\vspace{-10pt}
We now have the tools necessary to define a twisted version of cyclic reduction in RAAGs.
\begin{defn}\label{defn:twisted CR general}
    Let $G = \langle X \rangle$, let $v \in X^{\ast}$ be a geodesic, and let $\psi \in \mathrm{Aut}(G)$. We say $v$ is \emph{$\psi$-cyclically geodesic} if there does not exist a sequence $v=v_{0}, v_{1}, \dots, v_{n}$ such that $l(v_{n}) < l(v)$ and $\pi(v_{i+1}) = \pi(v'_{i})$, where $v'_{i}$ is a $\psi$-cyclic permutation of $v_{i}$.
\end{defn}

\begin{defn}\label{def: CR}
Let $\RAAG = \langle X \rangle$, and let $v \in X^{\ast}$ be a geodesic. We say $v$ is \emph{$\psi$-cyclically reduced} ($\psi$-CR) if there does not exist a sequence of $\psi$-cyclic permutations, commutation relations and free reductions to a geodesic $w \in X^{\ast}$, such that $l(v) > l(w)$. 
\end{defn}

\begin{rmk}
    Similar to before, Definitions \ref{defn:twisted CR general} and \ref{def: CR} coincide when our group $G$ is a RAAG. Again we will refer to `$\psi$-cyclically geodesic' words for general groups, and `$\psi$-cyclic reduction' for RAAGs.
\end{rmk}
\vspace{-10pt}
It is well known that for a RAAG $\RAAG = \langle X \rangle$, a word $v \in X^{\ast}$ is cyclically reduced if and only if $v$ cannot be written in the form $v =_{\RAAG} u^{-1}wu$, where $l(w) < l(v)$ and $u^{-1}wu \in X^{\ast}$ is geodesic. We prove one direction of this idea with respect to twisted cyclically geodesic words. 
\begin{prop}\label{prop:cycle red}
Let $G = \langle X \rangle$, and let $v \in X^{\ast}$ be a geodesic. Let $\psi \in \mathrm{Aut}(G)$. Suppose $v$ can be written in the form 
\[ v =_{\RAAG} \psi(u)^{-1}wu,
\]
where $l(w) < l(v)$ and $\psi(u)^{-1}wu \in X^{\ast}$ is geodesic. Then $v$ is not $\psi$-CR. 
\end{prop}

\begin{proof}
We can perform a $\psi$-cyclic shift from $v$ to $w$ as follows:
\[ v =_{G} \psi(u)^{-1}wu \xleftrightarrow{\psi} \psi(u)\cdot \psi(u)^{-1}w =_{G} w.
\]
By assumption $l(w) < l(v)$, so $v$ cannot be $\psi$-cyclically geodesic. 
\end{proof}
\comm{
\begin{rmk}
This doesn't hold for non-length preserving cases. Consider the RAAG $A_{\Gamma} = F_{2} \times F_{2}$ (see \cref{exmp: RAAG F2}), and let $\psi$ be the following order 2 automorphism:
\[ \psi: a \mapsto ac, \quad c \mapsto c^{-1}
\]
(fixing all other vertices). Let $V = aca \in \RAAG$. Then we can $\psi$-cyclic shift the first two letters to get a shorter word:
\[ V = aca \xmapsto{\psi} a\psi^{-1}(ac) = aa
\]
So by definition, $V$ is not $\psi$-CR. However it is clear that $V$ cannot be written in the form $V =_{\RAAG} \psi^{-1}(U)WU$.

\end{rmk}}

\comm{
\begin{lemma}\label{lem:conj and twist}
Let $u = t^{l}U, v=t^{l}V \in A_{\phi}$, and set $\psi = \phi^{l}$. Then
\[ U \sim_{\psi} V \; \Leftrightarrow \; t^{l}U \sim t^{l}V \quad \text{by an element} \; X \in \RAAG
\]

\end{lemma}

\begin{proof}
\comm{
Suppose 
\[ X^{-1}\cdot t^{l}U \cdot X = t^{l}V
\]
holds for some $X \in \RAAG$. This is true if and only if:
\begin{align*}
    t^{-l}\cdot X^{-1}t^{l}UX &= V \\
    \Leftrightarrow \psi(X)^{-1}UX &= V \\
    \Leftrightarrow U \sim_{\psi} V
\end{align*}
}
Let $X \in \RAAG$ be the conjugating element. Then:
\begin{align*}
    u &= X^{-1}vX \\
    \Leftrightarrow t^{l}U &= X^{-1}t^{l}VX \\
    \Leftrightarrow U &= (t^{-l}X^{-1}t^{l})VX \\
    \Leftrightarrow U &= \psi(X)^{-1}VX 
\end{align*}
\end{proof}}

\comm{
\begin{lemma}\label{lem:cycles equal}
Let $\psi$ be a length preserving automorphism. If $U \in \RAAG$ is reduced and $\psi$-CR, any $\psi$-cyclic permutation of $U$ is also $\psi$-CR. 
\end{lemma}

\begin{proof}
Let $U = U_{1}\dots U_{n}$ be reduced and $\psi$-CR word, and consider an arbitrary $\psi$-cyclic permutation:
\[ P = \psi^{k}(U_{i+1})\dots \psi^{k}(U_{n})\psi^{k-1}(U_{1})\dots \psi^{k-1}(U_{i})
\]
\textbf{Claim}: Since $U$ is $\psi$-CR, $|P| = |U| = n$. \\
\myRed{Do we need the length to be same?}\\
\textbf{Proof of Claim}: Since $\psi$ is length preserving, $|P| \leq n$, so we want to show that $|P| \not < n$. Suppose, for a contradiction, that $|P| < n$, so there exists at least one trivial cancellation in $P$. This cancellation must then be one of the following three cases:
\begin{enumerate}
    \item $\psi^{k}(U_{s}) = \psi^{k}(U_{t})^{-1}$ which cancel after shuffling.
    \item $\psi^{k-1}(U_{s}) = \psi^{k-1}(U_{t})^{-1}$ which cancel after shuffling.
    \item $\psi^{k}(U_{s}) = \psi^{k-1}(U_{t})^{-1}$ which cancel after shuffling.
\end{enumerate}
(for some $s,t$). For both Cases 1 and 2, this implies that
\[ U_{s} = U^{-1}_{t}
\]
But this implies that we can write $U$ in the form
\[ U =_{\RAAG} U_{1}\dots U_{s}U_{t}\dots U_{n}
\]
after shuffling, and so $U$ contains a trivial cancellation. This cannot happen since $U$ is reduced. \\
For Case 3, this implies that 
\[ \psi(U_{s}) = U^{-1}_{t}
\]
This means that we can write $U$ as 
\[U =_{\RAAG} U_{t}U_{1} \dots U_{n}U_{s} = \psi(U_{s})^{-1}U_{1} \dots U_{n}U_{s}
\]
In particular, $U$ is not $\psi$-CR, which again gives a contradiction. Since this is all possible cases, this proves the Claim.\\
Now suppose $P$ is not $\psi$-CR. If this were true, we could rewrite $P$ as one of the following 4 forms:
\begin{align*}
    1:& \quad \psi^{k}(U_{s})\cdot V\cdot \psi^{k-1}(U_{t}) \\
    2:& \quad \psi^{k}(U_{s})\cdot V\cdot \psi^{k}(U_{t}) \\
    3:& \quad \psi^{k}(U_{s})\cdot V\cdot \psi^{k+1}(U_{t}) \\
    4:& \quad \psi^{k-1}(U_{s})\cdot V\cdot \psi^{k-1}(U_{t}) 
\end{align*}
for some reduced word $V$ where $|V| < |P|$. For Case 1 (and similarly 3), we have that
\[ \psi^{k}(U_{s}) = \psi^{k}(U_{t})^{-1} \Rightarrow U_{s} = U^{-1}_{t}
\]
In this situation, we can write $P$ as follows:
\[ P = \psi^{k}(U_{i+1}) \dots \mathbf{\psi^{k}(U_{s})}\dots \psi^{k}(U_{n})\psi^{k-1}(U_{1})\dots \mathbf{\psi^{k-1}(U_{t})}\dots \psi^{k-1}(U_{i})
\]
where the letters in bold can shuffle to the ends of the word. In particular,
\[ [U_{s}, U_{j}] = 1 \; \text{for all} \; i+1 \leq j \leq s-1
\]
and 
\[ [U_{t}, U_{l}] = 1 \; \text{for all} \; t+1 \leq l \leq i
\]
This means that we can write $U$ as
\[ U = U_{1} \dots \mathbold{U_{t}}\dots U_{i}U_{i+1}\dots \mathbold{U_{s}}\dots U_{n} =_{\RAAG} U_{1}\dots \mathbold{U_{t}U_{s}}\dots U_{n}
\]
These two letters then cancel since $U_{s} = U^{-1}_{t}$, which contradicts the assumption that $U$ is reduced.\\
For Case 2 (and similarly 4), we have that
\[ \psi^{k}(U_{s}) = \psi^{k+1}(U_{t})^{-1} \Rightarrow U_{s} = \psi(U_{t})^{-1}
\]
In this situation, we can write $P$ as
\[ P = \psi^{k}(U_{i+1})\dots \mathbold{\psi^{k}(U_{s})}\dots  \mathbold{\psi^{k-1}(U_{t})}\dots  \psi^{k}(U_{n})\psi^{k-1}(U_{1})\dots \psi^{k-1}(U_{i})
\]
where the letters in bold can shuffle to the ends of the word. In particular:
\[ [U_{s}, U_{j}] = 1 \; \text{for all} \; i+1 \leq j \leq s-1
\]
and 
\[ [U_{t}, U_{l}] = 1 \; \text{for all} \; t+1 \leq l \leq n
\]
We also have
\[ [\psi(U_{t}), U_{r}] = 1 \; \text{for all} \; 1 \leq r \leq i \Rightarrow [U_{s}, U_{r}] = 1 \; \text{for all} \; 1 \leq r \leq i
\]
since $U_{s} = \psi(U_{t})^{-1}$. Combining this information, we can write $U$ as 
\[ U = U_{1}\dots U_{i}U_{i+1}\dots \mathbold{U_{s}} \dots \mathbold{U_{t}}\dots U_{n} \equiv \mathbold{U_{s}}U_{1}\dots U_{n}\mathbold{U_{t}}
\]
But this implies that $U$ is not $\psi$-CR, which is again a contradiction.\\
Hence in all cases, we get a contradiction, and so $P$ must be $\psi$-CR. 

\end{proof}}
\vspace{-10pt}
We are now able to prove the main theorem of this section.
\twist
\begin{proof}
We start with the reverse implication. Suppose there exists a finite sequence 
\[ u = u_{0} \leftrightarrow u_{1} \leftrightarrow \dots \leftrightarrow u_{n} = v
\]
of $\psi$-cyclic permutations and commutation relations. Then by \cref{cor:twisted perms}, $u_{i} \sim_{\psi} u_{i+1}$ for all $0 \leq i < n$, and so $u \sim_{\psi} v$.
\par
For the forward direction, we can write $u$ in the form $u =_{\RAAG} \psi(x)^{-1}vx$ for some geodesic $x \in X^{\ast}$. Since both $u$ and $v$ are $\psi$-CR, $l(u) = l(v)$, and so $\psi(x)^{-1}vx$ must contain cancellations by \cref{prop:cycle red}. First suppose there exists cancellations between $\psi(x)^{-1}$ and $x$ after commutation relations. Let $x =_{\RAAG} x_{1}x_{2}$ and $\psi(x)^{-1} =_{\RAAG} x^{-1}_{3}x^{-1}_{1}$, and so
\[ u =_{\RAAG} \psi(x)^{-1}vx =_{\RAAG} x^{-1}_{3}x^{-1}_{1}vx_{1}x_{2} =_{\RAAG} x^{-1}_{3}vx_{2}.
\]
In particular, all letter in $x_{1}$ commute with all letters in $v$. Note that
\[ \psi(x)^{-1} =_{\RAAG} \psi(x_{2})^{-1}\psi(x_{1})^{-1} =_{\RAAG} x^{-1}_{3}x^{-1}_{1},
\]
and so $\psi(x_{2})x^{-1}_{3} =_{\RAAG} \psi(x_{1})^{-1}x_{1}$. In this scenario, $u$ and $v$ are related by a sequence of $\psi$-cyclic shifts and commutation relations as follows:
\[ u =_{\RAAG} x^{-1}_{3}vx_{2} \xleftrightarrow{\psi} \psi(x_{2})x^{-1}_{3}v =_{\RAAG} \psi(x_{1})^{-1}x_{1}v =_{\RAAG} \psi(x_{1})vx_{1} \xleftrightarrow{\psi} \psi(x_{1})\psi(x_{1})^{-1}v =_{\RAAG} v.
\]
Otherwise, we assume $x$ is of minimal length, that is, if cancellations can occur between $\psi(x)^{-1}$ and $x$ after commutation relations are applied, then we apply these cancellations. This leaves us to consider whether $v$ cancels fully or not in $\psi(x)^{-1}vx$.

\textbf{Case 1:} $v$ is not fully cancelling in  $\psi(x)^{-1}vx$.\\
Suppose there are cancellations in $\psi(x)^{-1}v$ only. Since $u$ is $\psi$-CR, we must have that $\psi(x)^{-1}$ is fully cancelled. In this case, write $v =_{\RAAG} \psi(x)v_{1}$. Then
\[ u =_{\RAAG} \psi(x)^{-1}vx =_{\RAAG} \psi(x)^{-1}\psi(x)v_{1}x =_{\RAAG} v_{1}x.
\]
Here $u =_{\RAAG} v_{1}x$ and $v =_{\RAAG} \psi(x)v_{1}$, so these words are related to each other by a $\psi$-cyclic shift of $x$. \\
Similarly suppose there are cancellations in $vx$ only. Similar to before, we must have that $x$ is fully cancelled. Let $v =_{\RAAG} v_{2}x^{-1}$. Then
\[ u =_{\RAAG} \psi(x)^{-1}vx =_{\RAAG} \psi(x)^{-1}v_{2}.
\]
Here $u =_{\RAAG} \psi(x)^{-1}v_{2}$ and $v =_{\RAAG} v_{2}x^{-1}$, so these words are related to each other by a $\psi$-cyclic shift of $x^{-1}$. 
\par 
Finally, suppose there exists cancellations in both $\psi(x)^{-1}v$ and $vx$. We note that since $v$ is $\psi$-CR, we can assume $v$ is not of the form $v =_{\RAAG} \psi(x_{s})^{-1}\tilde{v}x_{s}$ by \cref{prop:cycle red}. Hence there exists $\alpha, \beta \in X^{\ast}$, where $\alpha \neq_{\RAAG} \beta$, such that $v$ can be written in the form  
$v =_{\RAAG} \psi(\alpha)^{-1}\tilde{v}\beta$, where $\tilde{v}$ is geodesic and non-empty. By assumption we can write $\psi(x)^{-1} =_{\RAAG} \psi(x_{1})^{-1}\psi(\alpha)$ (cancellation in $\psi(x)^{-1}v)$) and $x =_{\RAAG} \beta^{-1}x_{2}$ (cancellation in $vx$). Now we have 
\[ u =_{\RAAG} \psi(x)^{-1}vx =_{\RAAG} \psi(x_{1})^{-1}\psi(\alpha)\psi(\alpha)^{-1}\tilde{v}\beta \beta^{-1}x_{2}.
\]
Note that $x =_{\RAAG} \alpha^{-1}x_{1} =_{\RAAG} \beta^{-1}x_{2}$. We claim that $\alpha^{-1}$ must be a subword of $x_{2}$, and similarly $\beta^{-1}$ must be a subword of $x_{1}$. Indeed suppose there exists $z \in X$ such that $\alpha^{-1} =_{\RAAG} z\alpha_{z}$. Then either $z$ is a letter from $x_{2}$ or $\beta^{-1}$, which can commute to the left of $\beta^{-1}x_{2}$. If the latter case is true, then $\beta^{-1} =_{\RAAG} z\beta_{z}$, which contradicts our assumption that $v$ is not of the form $v =_{\RAAG} \psi(x_{s})^{-1}\tilde{v}x_{s}$. Therefore any $z$ of this form must come from $x_{2}$, and can commute with all of $\beta^{-1}$, and so $\alpha^{-1}$ is a subword of $x_{2}$. A symmetric argument shows that $\beta^{-1}$ must be a subword of $x_{1}$. 

Since $\beta^{-1}$ is a subword of $x$ which can be moved via commutation relations to the left in $x$, $\psi(\beta)$ must be a subword of $\psi(x)^{-1}$ which can be moved via commutation relations to the right in $\psi(x)^{-1}$. Similarly $\alpha^{-1}$ must be a subword of $x$ which can be moved via commutation relations to the left in $x$. Therefore, we can also write $u$ in the form
\[ u =_{\RAAG} \psi(x_{r})^{-1}\psi(\beta)\psi(\alpha)\psi(\alpha)^{-1}\tilde{v}\beta \beta^{-1}\alpha^{-1}x_{r}.
\]
Since $u$ is $\psi$-CR, we can assume $x_{r}$ is the empty word, and so $u =_{\RAAG} \psi(\beta)\tilde{v}\alpha^{-1}$. We can see $u$ and $v$ are related by the following sequence of $\psi$-cyclic shifts and commutation relations:
\[ u =_{\RAAG} \psi(\beta)\tilde{v}\alpha^{-1}  \xleftrightarrow{\psi} \psi(\alpha)^{-1}\psi(\beta)\tilde{v} =_{\RAAG} \psi(\beta)\psi(\alpha)^{-1}\tilde{v} \xleftrightarrow{\psi} \psi(\alpha)^{-1}\tilde{v}\beta =_{\RAAG} v.
\]
We note this scenario also holds when $\tilde{v}$ is the empty word, which lies in the second case we're about to prove, when $v$ is fully cancelling. 
\par 
\textbf{Case 2:} $v$ is fully cancelling in $\psi(x)^{-1}vx$.\\
First suppose $v$ fully cancels in $vx$. We can assume there are cancellations of the remaining part of $x$ with $\psi(x)^{-1}$, otherwise $u$ would not be $\psi$-CR. We set
\[ x =_{\RAAG} v^{-1}x_{1} \quad \Rightarrow \psi(x)^{-1} =_{\RAAG} \psi(x_{1})^{-1}\psi(v).
\]
Then
\[ u =_{\RAAG} \psi(x)^{-1}vx =_{\RAAG} \psi(x_{1})^{-1}\psi(v)x_{1}.    
\]
We now consider the cancellations remaining in $x$. We first note that if no letters from $x_{1}$ cancel with letters in $\psi(v)$, then $x_{1}$ and $\psi(x_{1})^{-1}$ must be inverses of each other and cancel after commutation relations, since $u$ is $\psi$-CR. We therefore assume $x_{1}$ is of minimal length up to any cancellation with terms in $\psi(x_{1})^{-1}$. Suppose $x_{1}$ fully cancels with $\psi(v)$. Then let $\psi(v) =_{\RAAG} v_{2}x^{-1}_{1}$, and so $u =_{\RAAG} \psi(x_{1})^{-1}v_{2}$. Now
\[ \psi(v) =_{\RAAG} v_{2}x^{-1}_{1} \quad \Rightarrow v =_{\RAAG} \psi^{-1}(v_{2})\psi^{-1}(x^{-1}_{1}),
\]
and so $u$ and $v$ are related by a sequence of $\psi$-cyclic shifts and commutation relations:
\[v =_{\RAAG} \psi^{-1}(v_{2})\psi^{-1}(x^{-1}_{1}) \xleftrightarrow{\psi} x^{-1}_{1}\psi^{-1}(v_{2}) \xleftrightarrow{\psi} v_{2}x^{-1}_{1} \xleftrightarrow{\psi} \psi(x^{-1}_{1})v_{2} =_{\RAAG} u.
\]
If instead, $\psi(v)$ fully cancels in $\psi(v)x_{1}$, then we have
\[ x_{1} =_{\RAAG} \psi(v)^{-1}x_{11} \Rightarrow \psi(x_{1})^{-1} =_{\RAAG} \psi(x_{11})^{-1}\psi^{2}(v),
\]
and hence $u =_{\RAAG} \psi(x_{11})^{-1}\psi^{2}(v)x_{11}$. The result then follows by reverse induction on the length of $x$.
\par 
The case where $v$ fully cancels in $\psi(x^{-1})v$, follows a similar proof. We already gave a proof in Case 1 for when $v$ cancels in both $\psi(x)^{-1}v$ and $vx$, and so in all cases, $u$ and $v$ are related by a sequence of $\psi$-cyclic permutations and commutation relations.
\end{proof}
\comm{
\subsection{Non-length preserving cases}\label{sec:non length p}
\vspace{5pt}
We mention briefly how the results from above can be extended to all finite order automorphisms. For $\RAAG = \langle X \rangle$, let $v \in X^{\ast}$ be $\psi$-cyclically reduced for some finite order $\psi \in \mathrm{Aut}(\RAAG)$.
\begin{enumerate}
    \item If $\psi$ is length preserving, then any word $w$ which can be obtained from $v$ via a sequence of $\psi$-cyclic permutations and commutation relations must also be $\psi$-cyclically reduced.
    \item If $\psi$ is non-length preserving, we can find words $w$ from $v$ via $\psi$-cyclic shifts and commutation relations such that $l(\mathrm{Geod}(w)) > l(v)$. In particular, $w$ is not $\psi$-cyclically reduced.
\end{enumerate}
To extend \cref{cor:sequence} for all finite order automorphisms, we need to also include free reduction in any sequence between twisted conjugate words, since any word in the sequence could increase in length after $\psi$-cyclic shifts. The definitions and results from Section 3 follow similarly as before, by adding $\mathrm{Geod}$ notation when required. 
\twistext}
\subsection{Twisted cyclic reduction: inversions}
\vspace{5pt}
When considering standard cyclic reduction in a RAAG $\RAAG = \langle X \rangle$, we recall that any word $w \in X^{\ast}$ which is cyclically reduced cannot have the form $w = x_{1}a_{i}^{\pm 1}x_{2}a_{i}^{\mp 1}x_{3}$, where all letters of subwords $x_{1}$ and $x_{3}$ commute with $a_{i}$. For twisted cyclic reduction, however, we have a different situation. In particular, the reverse direction of \cref{prop:cycle red} does not necessarily hold.
\begin{exmp}\label{exmp: line order 2 auto}
Let $\RAAG$ be the RAAG with defining graph 
\[\begin{tikzcd}
	a & b & c & d.
	\arrow[no head, from=1-1, to=1-2]
	\arrow[no head, from=1-2, to=1-3]
	\arrow[no head, from=1-3, to=1-4]
\end{tikzcd}\]
Let $\psi \colon a \leftrightarrow d, \; b \leftrightarrow c$ be a reflection, and consider the geodesic word $w = c^{-1}aac$. Then $w \not =_{\RAAG} \psi(u)^{-1}vu$ for any shorter word $v \in X^{\ast}$, where $\psi(u)^{-1}vu$ is geodesic, but $w$ is not $\psi$-CR. This can be seen by the following sequence of $\psi$-cyclic shifts and commutation relations of $w$:
\[ w = c^{-1}aac \xleftrightarrow{\psi} aac\cdot \psi^{-1}(c^{-1}) = aacb^{-1} =_{\RAAG} b^{-1}aac \xleftrightarrow{\psi} aac\cdot \psi^{-1}(b^{-1}) = aacc^{-1} =_{\RAAG} aa.
\]
However, $w$ is a $\psi$-cyclic geodesic, since all $\psi$-cyclic permutations of $w$ are geodesic:
\[ c^{-1}aac \xleftrightarrow{\psi} bc^{-1}aa \xleftrightarrow{\psi} dbc^{-1}a \xleftrightarrow{\psi} ddbc^{-1} \xleftrightarrow{\psi} b^{-1}ddb \xleftrightarrow{\psi} cb^{-1}dd \xleftrightarrow{\psi} acb^{-1}d \xleftrightarrow{\psi} aacb^{-1} \xleftrightarrow{\psi} c^{-1}aac.
\]

\end{exmp}

\comm{
\begin{exmp}
Recall \cref{exmp: RAAG F2}, and consider the word $w = b^{-1}dbd$. Then clearly $w \not =_{\RAAG} \psi(u)^{-1}vu$, but $w$ is not $\psi$-CR. Indeed, we can compute the following $\psi$-cyclic permutations and commutation relations to get a shorter word:
\[ w = b^{-1}dbd \xleftrightarrow{\psi} dbda^{-1} =_{\RAAG} a^{-1}dbd \xleftrightarrow{\psi} aa^{-1}db = db
\]
\end{exmp}}
\vspace{-10pt}
We can however show the reverse direction of \cref{prop:cycle red} holds for inversions. We now show that if $\psi$ is an inversion, then all $\psi$-cyclic geodesics are necessarily $\psi$-CR.
\begin{lemma}\label{twist cyc geos}
    Let $\RAAG = \langle X \rangle$ and let $\psi \in \mathrm{Aut}(\RAAG)$ be a composition of inversions. Let $g \in \RAAG$ be arbitrary and let $w \in X^{\ast}$ be a geodesic representing $g$, such that $w$ is a $\psi$-cyclic geodesic. Then any geodesic words which represent $g$ are also $\psi$-cyclic geodesics. 
\end{lemma}
\vspace{-10pt}
The proof is similar to Lemma 3.8 of \cite{Ferov2016}, which considers the case where $\psi \in \mathrm{Aut}(\RAAG)$ is trivial.
\vspace{-10pt}
\begin{proof}
    Let $w = w_{1}\dots w_{n}$, where each $w_{i} \in X$, and suppose $[w_{i}, w_{i+1}] = 1$ for some $i \in \{1, \dots n-1\}$, where $w_{i} \neq w_{i+1}$. Consider the word $w' = w_{1} \dots w_{i+1}w_{i}w_{i+2}\dots w_{n} =_{\RAAG} w$. Let $w''$ be a $\psi$-cyclic permutation of $w'$. We claim that $w''$ is geodesic. Since $\psi$ has order two, we only have three cases to consider:
    \begin{enumerate}
        \item $w'' = \psi(w_{j+1})\dots \psi(w_{i+1})\psi(w_{i})\psi(w_{i+2})\dots \psi(w_{n})w_{1}\dots w_{j}$, where $j < i$.
        \item $w'' = \psi(w_{j+1})\dots \psi(w_{n})w_{1}\dots w_{i+1}w_{i}w_{i+2}\dots w_{j}$, where $j > i$.
        \item $w'' = \psi(w_{i})\psi(w_{i+2}) \dots \psi(w_{n})w_{1}\dots w_{i-1}w_{i+1}$.
    \end{enumerate}
For the first case, we see that $w'' =_{\RAAG} \psi(w_{j+1})\dots \psi(w_{i})\psi(w_{i+1})\psi(w_{i+2})\dots \psi(w_{n})w_{1}\dots w_{j}$, which is a $\psi$-cyclic permutation of $w$. This must be geodesic since $w$ is a $\psi$-cyclic geodesic. The second case follows similarly. For the third case, we note that the subword $s = \psi(w_{i+2}) \dots \psi(w_{n})w_{1}\dots w_{i-1}$ must be geodesic, since it is a subword of $p = \psi(w_{i+1})\psi(w_{i+2}) \dots \psi(w_{n})w_{1}\dots w_{i-1}w_{i}$, which is a $\psi$-cyclic permutation of $w$. If $w''$ is not geodesic, then either $\psi(w_{i})$ or $w_{i+1}$ cancel with a letter in the subword $s$ after applying commutation relations, or $\psi(w_{i}) = w^{-1}_{i+1}$, which cancel after applying commutation relations.
\par 
First suppose $\psi(w_{i})$ cancels with $\psi(w_{k})$, where $k \in \{i+2, \dots, n\}$, after applying commutation relations. Then $w_{i}$ can cancel with $w_{k}$ in $w$, and so $w$ is not geodesic, giving a contradiction. Now suppose $\psi(w_{i})$ cancels with $w_{l}$, where $l \in \{1, \dots, i-1\}$, after applying commutation relations. If we consider the $\psi$-cyclic permutation $q = \psi(w_{i})\psi(w_{i+1}) \dots \psi(w_{n})w_{1}\dots w_{i-1}$ of $w$, then $q$ is not geodesic, again giving a contradiction. The case for $w_{i+1}$ cancelling with letters in $s$ follows similarly.
\par 
Finally, suppose $\psi(w_{i}) = w^{-1}_{i+1}$ which cancel after applying commutation relations. Since $\psi$ is a composition of inversions, we can assume that $\psi(w_{i}) = w^{\pm 1}_{i}$. If $\psi$ acts as the identity on $w_{i}$, then our original word $w$ would not be geodesic. Otherwise $\psi(w_{i}) = w^{-1}_{i} = w^{-1}_{i+1}$, which contradicts our original assumption that $w_{i} \neq w_{i+1}$. Indeed if this was the case, $w = w'$. This completes the proof.
\end{proof}
We note that the final case of this proof does not hold if $\psi$ is a graph automorphism. This can be seen in \cref{exmp: line order 2 auto}, 
where the word $w = c^{-1}aac$ is a $\psi$-cyclic geodesic, but is not $\psi$-CR. 
\begin{cor}\label{form:inversions}
    Let $\RAAG = \langle X \rangle$, and let $\psi \in \mathrm{Aut}(\RAAG)$ be a composition of inversions. Then any geodesic $v \in X^{\ast}$ is $\psi$-CR if and only if $v$ cannot be written in the form 
\[ v =_{\RAAG} \psi(u)^{-1}wu,
\]
where $l(w) < l(v)$ and $\psi(u)^{-1}wu \in X^{\ast}$ is geodesic. 
\end{cor}
\vspace{-10pt}
\begin{proof}
    This follows immediately by \cref{prop:cycle red} and \cref{twist cyc geos}.
\end{proof}
\comm{
\begin{prop}\label{form:inversions}
Let $\RAAG = \langle X \rangle$, and let $\psi \in \mathrm{Aut}(\RAAG)$ be a composition of inversions. Then any geodesic $v \in X^{\ast}$ is $\psi$-CR if and only if $v$ cannot be written in the form 
\[ v =_{\RAAG} \psi(u)^{-1}wu,
\]
for some geodesics $u,w \in X^{\ast}$, where $l(w) < l(v)$. 
\end{prop}
\begin{proof}
Let $v = v_{1}\dots v_{n}$, and suppose $v$ is not $\psi$-CR. Then by definition there exists a sequence of $\psi$-cyclic permutations and commutation relations to a non-geodesic word. We proceed by induction on the length of this sequence.
\par 
Base case: The result is trivial if $v$ and $w$ are related by commutation relations only, so suppose $v \xleftrightarrow{\psi} w$, where $w$ is not geodesic. Then by \cref{defn:cyc perm}, we can assume $w$ is of the form
\[ w = \psi(v_{i})\dots \psi(v_{n})v_{1}\dots v_{i-1},
\]
for some $1 \leq i < n$. Note that we do not need to consider higher powers of $\psi$, since $\psi$ is of order 2.
\par 
If there exists free reduction after shuffling between letters $\{\psi(v_{i}), \dots, \psi(v_{n})\}$ or $\{v_{1}, \dots,  v_{i-1}\}$, then $v$ would not be geodesic. Hence we can assume there exists free cancellation after shuffling between at least two letters $\psi(v_{s})$ and $v_{t}$, where
\[ w = \psi(v_{i})\dots \psi(v_{s}) \dots \psi(v_{n})v_{1}\dots v_{t} \dots v_{i-1},
\]
for some $i \leq s \leq n, 1 \leq t \leq i-1$. In this case, we can write $v =_{\RAAG} v_{t}x v_{s}$ for some geodesic $x \in X^{\ast}$. Since $\psi(v_{s}) = v_{t}^{-1}$, this is equivalent to $v =_{\RAAG} \psi(v_{s})^{-1}x v_{s}$, and hence $v$ is of the correct form. 
\par 
Now consider a sequence 
\[ v = v_{0} \leftrightarrow v_{1} \leftrightarrow \dots \leftrightarrow v_{s} = w
\]
of $\psi$-cycles and commutation relations to a shorter word $w \in X^{\ast}$. Since $v_{0}$ is not $\psi$-CR, neither is $v_{1}$, so by inductive hypothesis we can assume that
\[ v_{1} = x_{1}\psi\left(a_{i}^{\mp 1}\right)x_{2}a_{i}^{\pm 1}x_{3},
\]
where all letters of $x_{1}$ commute with $\psi(a_{i})$, and all letters of $x_{3}$ commute with $a_{i}$. We note here that if all the letters of $x_{2}$ commute with $a_{i}$ as well, then $a_{i}^{\pm 1}$ commutes with all letters in $v_{1}$. In this case, any $\psi$-cyclic permutation will give either a word of the same form or a non-geodesic word, since the $a_{i}$ terms can shuffle with all letters in the word to cancel. Therefore in general, we can assume this is not the case and $v_{1}$ is the only possible arrangement.
\par 
If $v_{0}, v_{1}$ are related by commutation relations only we're done, so assume $v_{0} \xleftrightarrow{\psi} v_{1}$. Then $v_{0}$ must be one of the following:
\begin{enumerate}
    \item $v_{0} = \psi\left(x_{32}\right)x_{1}\psi\left(a_{i}^{\mp 1}\right)x_{2}a_{i}^{\pm1}x_{31}$ where $x_{3} = x_{31}x_{32}$.
    \par 
    Since $[a_{i}, x_{32}] = 1$ (by assumption $[a_{i}, x_{3}] = 1$), then $[\psi\left(a_{i}\right), \psi\left(x_{32}\right)] = 1$, and so 
    \[ v_{0} =_{\RAAG} \psi\left(a_{i}^{\mp1}\right)\cdot \psi\left(x_{32}\right)x_{1}x_{2}x_{3}\cdot a_{i}^{\pm 1}.
    \]
    \item $v_{0} = \psi\left(x_{22}\right)\psi\left(a_{i}^{\pm1}\right)\psi\left(x_{3}\right)x_{1}\psi\left(a_{i}^{\mp1}\right)x_{21}$ where $x_{2} = x_{21}x_{22}$.
    \par 
    Here the two $\psi(a_{i})$-terms can shuffle and cancel to give a shorter word, which is a contradiction.
    \item $v_{0} = \psi\left(x_{12}\right)a_{i}^{\mp1} \psi\left(x_{2}\right)\psi\left(a_{i}^{\pm1}\right)\psi\left(x_{3}\right)x_{11}$ where $x_{1} = x_{11}x_{12}$. 
    \par
    Similar to Case 1, we can rewrite $v_{0}$ in the form
    \[ v_{0} =_{\RAAG} a_{i}^{\mp1}\cdot \psi\left(x_{12}\right)\psi\left(x_{2}\right)\psi\left(x_{3}\right)x_{11}\cdot\psi\left(a_{i}^{\pm1}\right).
    \]
\end{enumerate}
The remaining possibilities follow symmetric arguments. In particular, for all cases which are geodesic, $v_{0}$ is of the required form. 
\end{proof}}

\comm{
Finally we have the following decision problems for a group $G$.
\begin{defn}
\textbf{The conjugacy problem}, $\mathrm{CP}(G)$: Given two elements $u,v \in G$, decide whether
\[ u \sim v 
\]
i.e. decide if there exists an element $x \in G$ such that $v = x^{-1}ux$
\end{defn}

\begin{defn}
\textbf{The twisted conjugacy problem}, $\mathrm{TCP}_{\phi}(G)$: 
Given two elements $u,v \in G$ and $\phi \in \mathrm{Aut}(G)$, decide whether
\[ u \sim_{\phi} v 
\]
i.e. decide if there exists an element $x \in G$ such that $v = \phi(x)^{-1}ux$. 

\end{defn}
Note that $\mathrm{TCP}_{Id}(G)$ is equivalent to $\mathrm{CP}(G)$. The twisted conjugacy problem has fewer known solutions than the more well-known conjugacy problem. Here is a few examples of groups where the twisted conjugacy problem is solvable:
\begin{itemize}
    \item Free groups (\cite{Free-by-cyclic}, Theorem 1.5)
    \item Houghton's groups (\cite{Cox2017}, Theorem 1)
    \item Thompson's group F (\cite{Thompson}, Theorem 1.2)
\end{itemize}
}

\section{Conjugacy geodesics in virtual graph products}\label{section: conj geos}


In this section we prove the following, which allows us to consider the language of conjugacy geodesics in group extensions.

\conjgeoextensiongeneral
\vspace{-10pt}
We will then use this result in the context of RAAGs to prove \cref{thm:ConjGeo Regular}. Our method follows a similar technique to \cite{Ciobanu2016}, where the authors prove that $\mathsf{ConjGeo}(\RAAG, X)$ is regular for any RAAG (or RACG), with respect to the standard generators.

\begin{defn}
Let $G = \langle X \rangle$, and let $\psi \in \mathrm{Aut}(G)$. For any language $L \subset X^{*}$, let $\mathsf{Cyc}_{\psi}(L)$ denote the \emph{$\psi$-cyclic closure} of $L$, which is the set of all $\psi$-cyclic permutations of words in $L$ (recall \cref{defn:cyc perm}). 
\end{defn}
\vspace{-10pt}
The following result is well known in the case where $\psi$
is the identity map, for regular and other languages (see Lemma 2.1, \cite{Ciobanu2016}).
\begin{prop}\label{lem:automata}
Let $\psi \in \mathrm{Aut}(\RAAG)$ be length preserving. If $L$ is regular, then $\mathsf{Cyc}_{\psi}(L)$ is regular. 
\end{prop}
\vspace{-10pt}
We acknowledge here a recent paper \cite{Mahalingam2022}, where a similar statement is shown for twisting by antimorphic involutions in semigroups.
\begin{proof} 
The idea of this proof is to take copies of the automaton accepting $L$, and adjust edge labels with respect to $\psi$, to account for $\psi$-cyclic shifts of letters.
\par 
Let $M$ be a finite state automaton accepting $L$, with state set $Q$ and initial state $q_{0}$. A word $w=x_{1}x_{2}\dots x_{n}$, where $x_{i}$ are letters, is in $\mathsf{Cyc}_{\psi}(L)$ if and only if there exists a $\psi$-cyclic permutation 
\[ v = \psi^{k}(x_{i})\dots \psi^{k}(x_{n})\psi^{k-1}(x_{1})\dots \psi^{k-1}(x_{i-1}) \in L
\]
of $w$, for some $1 \leq i \leq n, \; 0 \leq k \leq m-1$, where $m$ is the order of $\psi$. In other words, $w$ is in $\mathsf{Cyc}_{\psi}(L)$ if and only if there exist states $q, q' \in Q$, with $q'$ accepting, such that $M$ contains a path from $q_{0}$ to $q$ labelled $\psi^{k}(x_{i})\dots \psi^{k}(x_{n})$, and a path from $q$ to $q'$ labelled $\psi^{k-1}(x_{1})\dots \psi^{k-1}(x_{i-1})$. We first define an automaton for each value of $i$ and $k$ as follows.
\par 
For each $i,k$, we take two copies of the states and transitions of $M$, which we denote by $\psi^{-k}(M)$ and $\psi^{-(k-1)}(M)$. For all edge labels $e \in \psi^{-k}(M)$ and $f \in \psi^{-(k-1)}(M)$, we apply a homomorphism
\[ e \mapsto \psi^{-k}(e), \quad f \mapsto \psi^{-(k-1)}(f).
\]
We define $M_{i,k}(q,q')$ to be the union of $\psi^{-k}(M)$ and $\psi^{-(k-1)}(M)$, where we add an $\epsilon$-transition from the state $q'$ in $\psi^{-(k-1)}(M)$ to the state $q_{0}$ in $\psi^{-k}(M)$. The start state of $M_{i,k}(q,q')$ is the state $q \in \psi^{-(k-1)}(M)$, and the single accept state is the state $q \in \psi^{-k}(M)$. With this construction, $M_{i,k}(q,q')$ will accept $w$ (see Figure \ref{fig:my_label}). 

\begin{figure}[h]
    \centering
\begin{tikzpicture}
    [scale=.8, auto=left,every node/.style={circle}]]
    \node[draw, initial] (qstart) at (0, 0) {\(q\)};
    \node[draw] (p1) at (1, 1) {};
    \node[draw] (p2) at (2, 2) {};
    \node[scale=0.5] (invisul1) at (3.3, 3.3) {};
    \node[scale=0.5] (invisul2) at (3.7, 3.7) {};
    \node[draw] (q'l) at (5, 5) {\(q'\)};
    \node[draw] (q0r) at (7, 5) {\(q_0\)};
    \node[draw] (p3) at (8, 4) {};
    \node[draw] (p4) at (9, 3) {};
    \node[scale=0.5] (invisur1) at (10.3, 1.7) {};
    \node[scale=0.5] (invisur2) at (10.7, 1.3) {};
    \node[draw] (p5) at (11, -1) {};
    \node[draw] (p6) at (10, -2) {};
    \node[scale=0.5] (invislr1) at (8.7, -3.3) {};
    \node[scale=0.5] (invislr2) at (8.3, -3.7) {};
    \node[draw] (q'r) at (7, -5) {\(q'\)};
    \node[draw] (q0l) at (5, -5) {\(q_0\)};
    \node[draw] (p7) at (4, -4) {};
    \node[draw] (p8) at (3, -3) {};
    \node[scale=0.5] (invisll1) at (1.7, -1.7) {};
    \node[scale=0.5] (invisll2) at (1.3, -1.3) {};
    \node[draw, double] (qend) at (12, 0)  {\(q\)};
    \draw[->, >=stealth] (qstart) edge["\(x_1\)"] (p1);
    \draw[->, >=stealth] (p1) edge["\(x_2\)"] (p2);
    \draw[->, >=stealth] (p2) to (invisul1);
    \draw[dotted] (invisul1) to (invisul2);
    \draw[->, >=stealth] (invisul2) edge["\(x_{i-1}\)"] (q'l);
    \draw[->, >=stealth, out=45, in=135] (q'l) edge["\(\epsilon\)"] (q0r);
    \draw[->, >=stealth] (q0r) edge["\(x_{i}\)"] (p3);
    \draw[->, >=stealth] (p3) edge["\(x_{i+1}\)"] (p4);
    \draw[->, >=stealth] (p4) to (invisur1);
    \draw[dotted] (invisur1) to (invisur2);
    \draw[->, >=stealth] (invisur2) edge["\(x_{n}\)"] (qend);
    \draw[->, >=stealth] (qend) edge[near start, "\(\psi(x_{1})\)"] (p5);
    \draw[->, >=stealth] (p5) edge[near start, "\(\psi(x_{2})\)"] (p6);
    \draw[->, >=stealth] (p6) to (invislr1);
    \draw[dotted] (invislr1) to (invislr2);
    \draw[->, >=stealth] (invislr2) edge[near start, "\(\psi(x_{i-1})\)"] (q'r);
    \draw[->, >=stealth] (q0l) edge[near end, "\(\psi(x_{i})\)"]  (p7);
    \draw[->, >=stealth] (p7) edge[near end, "\(\psi(x_{i+1})\)"] (p8);
    \draw[->, >=stealth] (p8) to (invisll1);
    \draw[dotted] (invisll1) to (invisll2);
    \draw[->, >=stealth] (invisll2) edge[near end, "\(\psi(x_{n})\)"] (qstart);
\end{tikzpicture}

\comm{
\begin{tikzpicture}
        \node[state] at (2,0) (a) {$q_{0}$} ;
        \node[state, initial] at (0,2) (b) {$q$} ;
        \node[state] at (2,4) (c) {$q^{'}$};
        \node[state] at (6,4) (d) {$q_{0}$};
        \node[state, accepting] at (8,2) (e) {$q$};
        \node[state] at (6,0) (f) {$q^{'}$};
        \draw(a) edge[left] node{$\psi(w_{i})\dots \psi(w_{n})$} (b);
        \draw (b) edge[left] node{$w_{1}\dots w_{i-1}$} (c);
        \draw (c) edge[bend left, above] node{$\epsilon$} (d);
        \draw (d) edge[right] node{$w_{i}\dots w_{n}$} (e);
        \draw (e) edge[right] node{$\psi(w_{1})\dots \psi(w_{i-1})$} (f);
\end{tikzpicture}}
    \caption{Automata $M_{i,k}(q,q')$}
    \label{fig:my_label}
\end{figure}
\par 
We now construct an automaton $\widehat{M}$ as follows. First construct all possible $M_{i,k}(q,q')$ for $1 \leq i \leq n$, $0 \leq k \leq m-1$. Add two new states $S$ and $E$, which will become our start and end states for $\widehat{M}$. We add $\epsilon$-transitions from $S$ to each start state in $M_{i,k}(q,q')$, for all possible $i,k$, and similarly add  $\epsilon$-transitions from $E$ to each end state in $M_{i,k}(q,q')$. With this construction, $\widehat{M}$ will accept $w = x_{1}\dots x_{n}$ for all possible $i,k$, and so will accept $w$ for all $w \in \mathsf{Cyc}_{\psi}(L)$. $\widehat{M}$ is a finite state automaton, since $\psi$ is of finite order, hence $\mathsf{Cyc}_{\psi}(L)$ is regular. 
\end{proof}

We now consider regularity for twisted conjugacy languages (recall \cref{defn:twisted conjugacy languages}).
\begin{prop}\label{prop:cycgeo}
Let $\psi \in \mathrm{Aut}(G)$ be length preserving. If $\mathsf{Geo}(G,X)$ is regular, then the language $\mathsf{CycGeo}_{\psi}(G, X)$ is regular.
\end{prop}
\vspace{-10pt}
\begin{proof}
We prove that
\[ \mathsf{CycGeo}_{\psi} = X^{*} \setminus \mathsf{Cyc}_{\psi}(X^{*}\setminus \mathsf{Geo}).
\]
($\supseteq$): Let $w \in X^{*} \setminus \mathsf{Cyc}_{\psi}(X^{*}\setminus \mathsf{Geo})$, and suppose $w \not \in \mathsf{CycGeo}_{\psi}$. Then there exists a $\psi$-cyclic permutation $w'$ of $w$ such that $w'$ is not geodesic. By definition this means $w \in \mathsf{Cyc}_{\psi}(X^{*}\setminus \mathsf{Geo})$, which is a contradiction. 
\par 
($\subseteq$): Let $w \in \mathsf{CycGeo}_{\psi}$, and suppose $w \in \mathsf{Cyc}_{\psi}(X^{*}\setminus \mathsf{Geo})$. Then there exists a $\psi$-cyclic permutation $w'$ of $w$ such that $w'$ is not geodesic. But since $w \in \mathsf{CycGeo}_{\psi}$, all $\psi$-cyclic permutations of $w$ are geodesic, again giving a contradiction.
\par 
Since $\mathsf{Geo}(\RAAG, X)$ is regular, the result follows by \cref{lem: closure} and \cref{lem:automata}.
\end{proof}

\begin{proof}[Proof of \cref{thm:conjgeo regular extension group}]
    First consider the set 
    \[ S =  \bigcup^{m-1}_{l=0} \{Ut^{l} \mid U \in \mathsf{ConjGeo}_{\phi^{l}}(N, X) \} \subset \mathsf{ConjGeo}\left(G, \widehat{X}\right)
    \]
    For fixed $l$, the languages $\{t^{l}\}$ (singleton language) and $\mathsf{ConjGeo}_{\phi^{l}}(N, X)$ are both regular (by \cref{prop:cycgeo}), and so the concatenation of these languages is regular by \cref{lem: closure}. Since $t$ has finite order, the union of these regular languages across all possible $l$ must also be regular. Let $M$ be the finite state automaton which accepts the language $S$, with state set $Q$ and initial state $q_{0}$. 

    Let $w = w_{1}w_{2} \in \mathsf{ConjGeo}_{\phi^{l}}(N, X)$, for some $w_{1} \in \mathcal{P}(w)$, $w_{2} \in \mathcal{S}(w)$ (see \cref{def:shifts moves}). Then $w_{1}w_{2}t^{l} \in S$, and so there exists states $q_{1}, q_{2}, q_{3} \in Q$, with $q_{3}$ accepting, such that $M$ contains a path from $q_{0}$ to $q_{1}$ labelled $w_{1}$, a path from $q_{1}$ to $q_{2}$ labelled $w_{2}$, and a path from $q_{2}$ to $q_{3}$ labelled $t^{l}$. To account for all words in $\mathsf{ConjGeo}\left(G, \widehat{X}\right)$, we need to adapt our automaton $M$ as follows.

    Note that $w_{1}w_{2}t^{l} =_{G} w_{1}t^{l}\phi^{l}(w_{2}) \in \mathsf{ConjGeo}\left(G, \widehat{X}\right)$ for all possible prefixes and suffixes of a word $w \in \mathsf{ConjGeo}_{\phi^{l}}(N, X)$. To include all possible words of this form $w_{1}t^{l}\phi^{l}(w_{2}) \in \mathsf{ConjGeo}\left(G, \widehat{X}\right)$, then for every path in $M$ from states $q_{0}$ to $q_{3}$ as described above, we add a new state $q_{4}$, a transition from $q_{1}$ to $q_{4}$ labelled $t^{l}$, and a transition from $q_{4}$ to $q_{3}$ labelled $\phi^{l}(w_{2})$ (see \cref{fig:extension conj geo}). 
    \begin{figure}[h]
    \centering
\begin{tikzpicture}
    [scale=.8, auto=left,every node/.style={circle}]]
    \node[draw, initial] (qstart) at (0, 0) {\(q_{0}\)};
    \node[draw] (q1) at (4,0) {$q_{1}$};
    \node[draw] (q2) at (8,0) {$q_{2}$};
    \node[draw, double] (q3) at (12,0) {$q_{3}$};
    \node[draw] (q4) at (8, -3) {$q_{4}$};
    \draw[->, >=stealth] (qstart) edge["\(w_{1}\)"] (q1);
    \draw[->, >=stealth] (q1) edge["\(w_{2}\)"] (q2);
    \draw[->, >=stealth] (q2) edge["\(t^{l}\)"] (q3);
    \draw[->, >=stealth] (q1) edge[swap, "\(t^{l}\)"] (q4);
    \draw[->, >=stealth] (q4) edge[swap, near end, "\(\phi^{l}(w_{2})\)"] (q3);
    \end{tikzpicture}
    \caption{Partial picture of automaton $M'$ which accepts $\mathsf{ConjGeo}\left(G, \widehat{X}\right)$}
    \label{fig:extension conj geo}
    \end{figure}\\
    Our new automaton $M'$, obtained from adding all possible states and transitions as above to $M$, is also a finite state automaton, since $\phi$ is length-preserving, which accepts the language $\mathsf{ConjGeo}\left(G, \widehat{X}\right)$.
\end{proof}
\vspace{-10pt}
We now apply this result in the context of RAAGs.

\begin{cor}\label{lem:conjgeos appearance}
Let $\psi \in \mathrm{Aut}(\RAAG)$ be length preserving. Then $\mathsf{ConjGeo}_{\psi}(\RAAG, X)$ consists precisely of all geodesic $\psi$-CR words.
\end{cor}

\begin{proof}
This follows immediately by \cref{cor:sequence}. 
\end{proof}

\begin{prop}\label{prop:conj and cyc}
Let $\psi \in \mathrm{Aut}(\RAAG)$ be a composition of inversions. Then
\[ \mathsf{ConjGeo}_{\psi}(\RAAG, X) = \mathsf{CycGeo}_{\psi}(\RAAG, X).
\]
\end{prop}
\begin{proof}
The $\subseteq$ direction is clear from the definitions, since $\psi$ is length preserving. 
\par 
Now suppose $v \in \mathsf{CycGeo}_{\psi}(\RAAG, X)$, but $v \not \in \mathsf{ConjGeo}_{\psi}(\RAAG, X)$. By \cref{lem:conjgeos appearance}, $v$ is not $\psi$-CR, and hence $v =_{\RAAG} \psi(u)^{-1}wu$ by \cref{form:inversions}. In particular, we can write
\[ v =  x_{1}\psi\left(a_{i}^{\mp 1}\right)x_{2}a_{i}^{\pm 1}x_{3},
\]
where $x_{1}, x_{2}, x_{3} \in X^{\ast}$ are geodesic, all letters of $x_{1}$ commute with $\psi\left(a_{i}^{\mp 1}\right)$, and all letters of $x_{3}$ commute with $a_{i}^{\pm 1}$. Then there exists a $\psi$-cyclic permutation
\[ v = x_{1}\psi\left(a_{i}^{\mp 1}\right)x_{2}a_{i}^{\pm 1}x_{3} \xleftrightarrow{\psi} \psi\left(a_{i}^{\pm 1}\right)\psi\left(x_{3}\right)x_{1}\psi\left(a_{i}^{\mp 1}\right)x_{2} =_{\RAAG} \psi(x_{3})x_{1}x_{2}.
\]
This contradicts the fact that $v$ is a $\psi$-cyclic geodesic, since the length of $\psi(x_{3})x_{1}x_{2}$ is less than the length of $v$. 
\end{proof}
\vspace{-10pt}
The following is immediate by \cref{thm:conjgeo regular extension group}, Propositions \ref{prop:cycgeo}, \ref{prop:conj and cyc} and \cite{LOEFFLER2002}.

\conjgeos

\begin{rmk}
\cref{thm:ConjGeo Regular} also holds for RACGs with analogous proof.
\end{rmk}
\vspace{-10pt}
One observation we make here is that the crucial step in the proof of \cref{prop:conj and cyc} comes from \cref{form:inversions}. One might ask if \cref{prop:conj and cyc} holds for other types of length preserving automorphisms, such as graph automorphisms. The next example illustrates how this is not the case. 
\begin{exmp}\label{exmp:clash languages}
Recall \cref{exmp: line order 2 auto} and consider the geodesic word $w = c^{-1}aac$. Then $w \in \mathsf{CycGeo}_{\psi}(\RAAG, X)$ since all $\psi$-cyclic permutations of $w$ are geodesic. However $w \not \in \mathsf{ConjGeo}_{\psi}(\RAAG, X)$, since there exists a sequence of $\psi$-cyclic shifts, commutation relations and free reduction to $w' = aa$, which is shorter than $w$.
\end{exmp}

\comm{
\begin{lemma}\label{lemma:conjgeo options}
Let $\psi \in \mathrm{Aut}(G_{\Gamma})$ be length preserving and fix elements within vertex groups. Then
\[ w \in \psi-\mathsf{ConjGeo}(G_{\Gamma},X) \Leftrightarrow \; \forall \; i: 
\]
\[
\rho_{i}(w) \in \psi-\mathsf{ConjGeo}(G_{i}, X_{i}) \cup \{ u_{0}\$ u_{1} \dots \$ u_{n} \; | \; \psi(u_{n})u_{0}, u_{1} \dots, u_{n-1} \in \mathsf{Geo}_{i} \}
\]
\end{lemma}

\begin{proof}
($\Rightarrow$): $w \in \psi-\mathsf{ConjGeo}(G_{\Gamma},X)$ implies $w \in \mathsf{Geo}(G_{\Gamma},X)$ and so $\rho_{i}(w) \in \mathsf{Geo}_{i}(\$\mathsf{Geo}_{i})^{*}$. In particular
\[ \rho_{i}(w) = u_{0}\$ u_{1} \dots \$ u_{n}
\]
where each $u_{j} \in \mathsf{Geo}_{i}$.\\
\textbf{Case 1: $n=0$}: Here $\rho_{i}(w) = u_{0} \in \mathsf{Geo}_{i}$. Suppose $\rho_{i}(w)$ is not $\psi$-CR, i.e.
\[ u_{0} = \psi(g)^{-1}ug
\]
Then
\[ w =_{\RAAG} \psi(g)^{-1}\cdot vu \cdot g
\]
where $v,g$ get mapped to $\epsilon$ under $\rho_{i}$. But this implies $w \not \in \psi-\mathsf{ConjGeo}(G_{\Gamma},X)$, which is a contradiction. Hence $\rho_{i}(w) \in \psi-\mathsf{ConjGeo}(G_{i}, X_{i})$.\\
\textbf{Case 2: $n>0$}: Note that if $w^{'}$ is a $\psi$-cyclic permutation of $w$, then $\rho_{i}(w^{'})$ is a $\psi$-cyclic permutation of $\rho_{i}(w)$. In particular there exists $w^{'}$ such that
\[ \rho_{i}(w^{'}) = \psi(u_{n})u_{0}\$ u_{1} \dots \$ u_{n-1}\$
\]
Since $w^{'}$ is geodesic, $\rho_{i}(w^{'}) \in \mathsf{Geo}_{i}(\$\mathsf{Geo}_{i})^{*}$ and so $\psi(u_{n})u_{0} \in \mathsf{Geo}_{i}$ as required.\\
($\Leftarrow$): By assumption $\rho_{i}(w) \in \mathsf{Geo}_{i}(\$\mathsf{Geo}_{i})^{*}$, and so $w \in \mathsf{Geo}$. Suppose $w \not \in \psi-\mathsf{ConjGeo}(G_{\Gamma}, X)$. Then there exists $y \in \mathsf{Geo}$ such that
\[ w = \psi(y)^{-1}xy
\]
and $l(x) < l(w)$. Here we choose $y$ to be of minimal length, and note $l(y) = l(\psi(y))$. \\
Since $w \in \mathsf{Geo}$, we can write
\[ \rho_{i}(w) = w_{i}u_{i}w^{'}_{i}
\]
where either:
\begin{itemize}[(a)]
    \item $w_{i} \in \mathsf{Geo}_{i}$, $u_{i} = w^{'}_{i} = \epsilon$, or
    \item $u_{i} \in \$ (\mathsf{Geo}_{i}\$)^{*}$, $w_{i}, w^{'}_{i} \in \mathsf{Geo}_{i}$
\end{itemize}
Since $y \in \mathsf{Geo}$ we can write
\[ \rho_{i}(y) = v_{i}y_{i}
\]
where $v_{i} \in (\mathsf{Geo}_{i}\$)^{*}$ and $y_{i} \in \mathsf{Geo}_{i}$. Define a map
\begin{align*}
    \widehat{\psi}: (X_{i} \cup \$)^{*} &\rightarrow (X_{i} \cup \$)^{*} \\
    a &\mapsto \psi(a) \quad a \in X_{i} \\
    \$ &\mapsto \$
\end{align*}
Then since $\psi$ fixes elements within vertex groups, we can assume that
\[ \rho_{i}(\psi(y)) = \widehat{\psi}(v_{i})\psi(y_{i})
\]
Therefore
\[ \rho_{i}(x) = \rho_{i}(\psi(y)wy^{-1}) = \widehat{\psi}(v_{i})\psi(y_{i})\cdot w_{i}u_{i}w^{'}_{i} \cdot y^{-1}_{i}v^{-1}_{i}
\]
Let $l_{i}(u)$ be the number of $X_{i}$ letters in a word $u \in (X_{i} \cup \$)^{*}$. Then
\[ l(u) = \sum_{i=1}^{n} l_{i}(\rho_{i}(u))
\]
In particular
\begin{equation}\label{eqn:2}
    l(x) = \sum_{i=1}^{n} l_{i}(\widehat{\psi}(v_{i})\psi(y_{i})\cdot w_{i}u_{i}w^{'}_{i} \cdot y^{-1}_{i}v^{-1}_{i}) \geq \sum_{i=1}^{n}(l(\psi(y_{i})w_{i}) + l_{i}(u_{i}) + l(w^{'}_{i}y^{-1}_{i})
\end{equation}
We need the following claim:\\
\textbf{Claim:}
\[ l(w_{i}) + l(w^{'}_{i}) \leq l(\psi(y_{i})w_{i}) + l(w^{'}_{i}y^{-1}_{i})
\]
\textbf{Proof of Claim:} \\
\textbf{Case 1:} $\rho_{i}(w) = w_{i} \in \mathsf{Geo}_{i}$. This implies
\[ \rho_{i}(\psi(y)^{-1}xy) \in \mathsf{Geo}_{i}
\]
so in particular we can assume $v_{i} = \epsilon$, i.e
\[ \rho_{i}(y) = y_{i}, \quad \rho_{i}(\psi(y)) = \psi(y_{i}) \quad \Rightarrow w_{i} = \psi(y_{i})^{-1}\rho_{i}(x)y_{i}
\]
Then
\begin{align*}
    l(\psi(y_{i})w_{i}) + l(w^{'}y^{-1}_{i}) &= l(\rho_{i}(x)y_{i}) + l(y^{-1}_{i}) \\
    &= l(\rho_{i}(x)) + l(y_{i}) + l(y^{-1}_{i}) \\
    &= l(\rho_{i}(x)) + l(\psi(y_{i})) + l(y^{-1}_{i}) \\
    &= l(w_{i})
\end{align*}
\textbf{Case 2}: $\rho_{i}(w) = w_{i}u_{i}w^{'}_{i}$. By assumption $\psi(w^{'}_{i})w_{i} \in \mathsf{Geo}_{i}$. Then
\begin{align*}
    l(\psi(y_{i})w_{i}) + l(w^{'}_{i}y^{-1}_{i}) &= l(\psi(y_{i})w_{i}) + l(\psi(w^{'}_{i})\psi(y^{-1}_{i}) \\
    &\geq l(\psi(w^{'}_{i})\psi(y^{-1}_{i})\psi(y_{i})w_{i}) \\
    &= l(\psi(w^{'}_{i})w_{i}) \\
    &= l(\psi(w^{'}_{i})) + l(w_{i}) \\
    &= l(w^{'}_{i}) + l(w_{i})
\end{align*}
We need to check whether the LHS of this inequality can equal zero. Suppose $\psi(y_{i})w_{i} = \epsilon$, i.e. $\psi(y_{i}) = w^{-1}_{i}$. Since
\begin{align*}
    \psi(w^{'}_{i})w_{i} \in \mathsf{Geo}_{i} &\Rightarrow \psi(w^{'}_{i})\psi(y^{-1}_{i}) \in \mathsf{Geo}_{i} \\
    &\Rightarrow \psi(w^{'}_{i}y^{-1}_{i}) \in \mathsf{Geo}_{i} \\
    &\Rightarrow w^{'}_{i}y^{-1}_{i} \in \mathsf{Geo}_{i}
\end{align*}
Let $y = \tilde{y}y_{i}$ and $w \in w_{i}\tilde{w}w^{'}_{i}$. Then
\[ \psi(y)wy^{-1} = \psi(\tilde{y})\cdot \psi(y_{i})w_{i}\tilde{w}w^{'}_{i}y^{-1}_{i}\cdot \tilde{y}^{-1} = \psi(\tilde{y})\cdot \tilde{w}w^{'}_{i}y^{-1}_{i}\cdot \tilde{y}^{-1} 
\]
Since $\tilde{w}w^{'}w_{i}$ and $w^{'}_{i}y^{-1}_{i}$ are both geodesic, $\tilde{w}w^{'}_{i}y^{-1}_{i} \in \mathsf{Geo}_{i}$, and since $l_{X}(\tilde{y}) < l_{X}(y)$, the equation above contradicts our assumption of minimality in choice of $y$. \\
Therefore by Equation \ref{eqn:2} and the Claim above, we have that
\begin{align*}
    l(w) &\geq \sum_{i=1}^{n}(l(w_{i}) + l_{i}(u_{i}) + l(w^{'}_{i})  \\
    &= \sum_{i=1}^{n}l_{i}(\rho_{i}(w)) \\
    &= l(w)
\end{align*}
But this contradicts our assumption that $l(x) < l(w)$, and so $w \in \psi-\mathsf{ConjGeo}(G_{\Gamma}, X)$, which completes the proof.
\end{proof}

\begin{thm}
Let $\psi \in \mathrm{Aut}(G_{\Gamma})$ be length preserving and fix elements within vertex groups. Suppose $\psi-\mathsf{CycGeo}(G_{i}, X_{i}) = \psi-\mathsf{ConjGeo}(G_{i}, X_{i})$ for each $i$. Then 
\[ \psi-\mathsf{CycGeo}(G_{\Gamma}, X) = \psi-\mathsf{ConjGeo}(G_{\Gamma}, X)
\]
\end{thm}

\begin{proof}
First note $\psi-\mathsf{ConjGeo}(G_{\Gamma},X) \subseteq \psi-\mathsf{CycGeo}(G_{\Gamma}, X)$ since if $w$ is a $\psi$-conjugacy geodesic, then so is every $\psi$-cyclic permutation of $w$ (since $\psi$ is length preserving). Now suppose $w \in \psi-\mathsf{CycGeo}(G_{\Gamma}, X)$. Since $w$ is geodesic, for each $i$ we have
\[ \rho_{i}(w) = u_{0}\$u_{1}\$\dots \$ u_{n}
\]
where each $u_{j} \in \mathsf{Geo}_{i}$. If $w^{'}$ is a $\psi$-cyclic permutation of $w$, then $\rho_{i}(w^{'})$ is a $\psi$-cyclic permutation of $\rho_{i}(w)$. In particular, there exists $w^{'}$ such that
\[ \rho_{i}(w^{'}) = \psi(u_{n})u_{0}\$u_{1}\dots u_{n-1}\$
\]
Since $w^{'} \in \mathsf{Geo}(G_{\Gamma},X)$ by assumption, we must have that 
\[ \psi(u_{n})u_{0}, u_{1}, \dots, u_{n-1} \in \mathsf{Geo}_{i}
\]
when $n \geq 1$. \\
Now suppose $n=0$, i.e. $\rho_{i}(w) = u_{0} \in \mathsf{Geo}_{i}$. For every $\psi$-cyclic permutation $u^{'}$ of $u_{0}$, there exists a $\psi$-cyclic permutation $w^{'}$ of $w$ such that $\rho_{i}(w^{'}) = u^{'} \in \mathsf{Geo}_{i}$. Therefore $u_{0} \in \mathsf{CycGeo}(G_{i}, X_{i})$. By assumption this means $u_{0} \in \mathsf{ConjGeo}(G_{i}, X_{i})$. \\
Therefore for all $n \geq 0$, we can apply Lemma \ref{lemma:conjgeo options} to conclude $w \in \psi-\mathsf{ConjGeo}(G_{\Gamma},X)$ as required. 

\end{proof}}
\subsection{Non-length preserving automorphisms and conjugacy geodesics}
\vspace{5pt}
We now focus on non-length preserving automorphisms of RAAGs, in particular partial conjugations and dominating transvections (see \cref{sec:autos RAAGs}). While these are infinite order automorphisms, they can be composed with inversions to give us order two automorphisms. We can then construct a virtual RAAG $A_{\phi}$, with respect to this type of finite order non-length preserving automorphism, and study the language $\mathsf{ConjGeo}\left(A_{\phi}, \widehat{X}\right)$.
\comm{
We start by highlighting an example which shows that \cref{prop:conj and cyc} doesn't hold in general for non-length preserving automorphisms.
\begin{exmp}
Let $\RAAG$ be the RAAG with defining graph
\[\begin{tikzcd}
	a & b \\
	d & c
	\arrow[no head, from=1-1, to=1-2]
	\arrow[no head, from=1-2, to=2-2]
	\arrow[no head, from=2-2, to=2-1]
	\arrow[no head, from=1-1, to=2-1]
\end{tikzcd}\]
Let $\psi: a \mapsto cac^{-1}, c \mapsto c^{-1}$ and consider $w = ca$. Then $w \in \mathsf{ConjGeo}_{\psi}(\RAAG, X)$, however the $\psi$-cyclic permutation
\[ w = ca \xleftrightarrow{\psi} \psi(a)c = cac^{-1}c
\]
is clearly not geodesic.
\end{exmp}}

\begin{defn}
For any vertex $v \in V(\Gamma)$, we define the \emph{link} of a vertex $\mathrm{Lk}(v)$ as
\[ \mathrm{Lk}(v) = \{ x \in V(\Gamma) \; | \; \{v,x\} \in E(\Gamma) \}.
\]
Similarly we define the \emph{star} of a vertex $\mathrm{St}(v)$ as 
\[ \mathrm{St}(v) = \mathrm{Lk}(v) \cup \{v\}.
\]
\end{defn}
\vspace{-10pt}
Our first result focuses on a composition of a partial conjugation and inversion. With this composition, we find that the the finite extension $A_{\phi}$ is precisely a graph product, and hence $\mathsf{ConjGeo}\left(A_{\phi}, \widehat{X}\right)$ is regular by \cite[Theorem 3.1]{Ciobanu2013}.

\begin{thm}\label{thm:non len geos}
Let $A_{\phi}$ be a virtual RAAG of the form $A_{\phi} = A_{\Gamma} \rtimes_{\phi} \langle t \rangle$ (as defined in \eqref{eqn: extension}) where $X$ is the standard generating set. For some $x \in V(\Gamma)$ and a connected component $D \subseteq \Gamma \setminus \mathrm{St}(x)$, let $\phi \in \mathrm{Aut}(\RAAG)$ be the following order two automorphism:
\begin{align*}
    \phi \colon X &\rightarrow X \\
    y &\mapsto xyx^{-1} \quad \forall \; y \in D, \\
    x &\mapsto x^{-1}, \\
    s &\mapsto s \quad \forall \; s \in X \setminus D \cup \{x\}.
\end{align*}
Then $A_{\phi}$ is a graph product, with vertex groups $\Z$ and $\Z/2\Z$.
\end{thm}

\begin{proof}
We start with the standard presentation for $A_{\phi}$, and use Tietze transformations to rewrite the extension in the form of a graph product. We have
\[ A_{\phi} = \langle X \cup \{t\} \; | \; R_{1}, R_{2}, \; t^{2} = 1, (tx)^{2} = 1, tyt = xyx^{-1} \; \text{for all} \; y \in D \rangle,
\]
where
\[ R_{1} = \{ [a,b] = 1 : \{a,b\} \in E(\Gamma) \}, \quad R_{2} = \{ [t,s] = 1 : s \in X \setminus D \cup \{x\} \}.
\]
Define a new generator $u = tx$, and remove $x$ using this relation. Our presentation becomes
\[ A_{\phi} = \langle X \setminus \{x\} \cup \{t, u\} \; | \; R'_{1}, R_{2}, t^{2} = 1, u^{2} = 1, tyt = tuyut \; \text{for all} \; y \in D \rangle,
\]
where the only other relations that change are of the form $[x,a]$ in $R_{1}$. These change to $[tu, a]$ for all $a \in St(x), a \neq x$. For each of these relations we have
\[ 1 = [tu,a] = uta^{-1}tua = ua^{-1}ua = [u,a],
\]
using the relation $[t,a] = 1$ from $R_{2}$. So each relation of the form $[tu,a]$ can be replaced by $[u,a]$. Also note
\[ tyt = tuyut \Leftrightarrow y = uyu \Leftrightarrow uy^{-1}uy = 1  \Leftrightarrow [u,y] = 1,
\]
so we can replace the relation $tyt = tuyut$ by $[u,y] = 1$. This leaves us with
\[ A_{\phi} = \langle X \setminus \{x\} \cup \{t, u\} \; | \; R_{1}, R_{2}, t^{2} = 1, u^{2} = 1, [u,y] = 1 \; \text{for all} \; y \in D \rangle,
\]
where
\[ R'_{1} = 
\begin{cases}
[a,b] = 1, & \{a,b\} \in E(\Gamma), \; a,b \neq x, \\
[u,a] = 1, & \; a \in St(x), \; a \neq x,
\end{cases}
\]
and 
\[ R_{2} = \{ [t,s] = 1, \; s \in X \setminus D \cup \{x\} \}.
\]
Our presentation is now in the form of a graph product with vertex groups $\mathbb{Z}$ and $\mathbb{Z}/2\Z$. 
\end{proof}
The following is immediate by \cite[Theorem 3.1]{Ciobanu2013}, since $\mathbb{Z}$ and $\mathbb{Z}/2\Z$ have regular conjugacy geodesics.

\nonlengthconjgeo
\comm{
\myRed{Remove this example? Proof fail CR step, need further justification.}\\
We also have regularity when we consider a composition of a specific type of dominating transvection with an inversion. We denote $Z(G)$ to be the \emph{centre} of a group.
\begin{prop}
Let $\psi \in \mathrm{Aut}(\RAAG)$ where
\[ \psi: x \mapsto xy, \quad y \mapsto y^{-1}
\]
for some $x,y \in \Gamma$ such that $y \in Z(\RAAG)$. Let $X$ be the standard generating set for $\RAAG$. Then $\psi-\mathsf{ConjGeo}(\RAAG, X)$ is regular.\\ Moreover, if we let $A_{\phi} = A_{\Gamma} \rtimes \langle t \rangle$ (as defined in (\ref{eqn: extension})), and $\widehat{X} = \{X, t\}$, then $\mathsf{ConjGeo}\left(A_{\phi}, \widehat{X}\right)$ is regular.
\end{prop}

\begin{proof}
We first show that if a word $w \in \psi-\mathsf{ConjGeo}(\RAAG, X)$, then it cannot contain both $x$ and $y$ letters. Let $n(x)$ and $n(y)$ denote the number of $x$ and $y$ letters respectively which appear in $w$. \\
\textbf{Case 1: $n(x) \leq n(y)$}: First shuffle $y$ letters next to all the $x$ letters, and collect together all remaining $y$ letters on the right, i.e.
\[ w =_{\RAAG} uv = \dots xy \dots xy \dots y^{k}
\]
where all $x$ letters are contained in $u$, with a $y$ letter directly right of each $x$ letter in the word. If we look at $\psi(w)$, i.e. a $\psi$-cyclic shift of $w$, we get
\[ \psi(w) = \dots x \dots x \dots y^{-k}
\]
Since all other letters remain fixed, we have that $|\psi(w)| < |w|$, which is a contradiction. \\
\textbf{Case 2: $n(x) > n(y)$}: Write $w =_{\RAAG} uv$ where $u$ contains all $y$ letters and each $y$ letter is next to an $x$ letter, and $v$ contains the remaining $x$ letters, i.e.
\[ w = \dots xy \dots xy \dots x \dots x
\]
If we compute the $\psi$-cyclic shift $v\psi(u)$ we get
\[ v\psi(u) = \dots x \dots x \dots x \dots x
\]
In particular, this word is shorter and so again gives a contradiction.
\par 
Next we consider all words $w \in (X \setminus \{x,y\})^{*}$. We immediately see that
\[ \psi-\mathsf{ConjGeo}(\RAAG, X \setminus \{x,y\}) = \mathsf{ConjGeo}(\RAAG, X \setminus \{x,y\})
\]
since all letters in $w$ are fixed under $\psi$.
\par 
We now need to check separately words which don't contain $x$, and words which don't contain $y$. We claim that
\begin{align*}
     \psi-\mathsf{ConjGeo}(\RAAG, X \setminus \{x\}) &= \mathsf{ConjGeo}(\RAAG, X \setminus \{x\}) \setminus \{ywy, y^{-1}wy^{-1} \; | \; w \in X \setminus \{x\} \} \\
     &\bigcup \{ywy^{-1}, y^{-1}wy \; | \; w \in X \setminus \{x\} \} 
\end{align*}
The $\supseteq$ inclusion is clear. For the $\subseteq$ direction, we note that if $w \in \psi-\mathsf{ConjGeo}(\RAAG, X \setminus \{x\})$, then any $\psi$-cyclic shift preserves the length of the word. Hence we can assume that $W \not =_{\RAAG} \psi^{-1}(U)VU$ for some $U,V$. This means we need to include all words in $\mathsf{ConjGeo}(\RAAG, X \setminus \{x\})$ except words of the form $ywy$ or $y^{-1}wy^{-1}$, since these are not $\psi$-CR. If $W$ is not part of this set, then $W$ must be of the form $ywy^{-1}$ or $y^{-1}wy$. These are precisely words which are not cyclically reduced, but are $\psi$-cyclically reduced. 
\par 
For any word over $X \setminus \{y\}$, we can use a similar argument to show
\[ \psi-\mathsf{ConjGeo}(\RAAG, X \setminus \{y\}) = \mathsf{ConjGeo}(\RAAG, X \setminus \{y\}) \bigcup \{xwx^{-1}, x^{-1}wx \; | \; w \in X \setminus \{y\} \} 
\]
noting that we don't need to discard any non-$\psi$-cyclically reduced words, as these would include both $x$ and $y$ letters. Hence
\[ \psi-\mathsf{ConjGeo}(\RAAG, X) = \psi-\mathsf{ConjGeo}(\RAAG, X \setminus \{x, y\}) \bigcup \psi-\mathsf{ConjGeo}(\RAAG, X \setminus \{x\}) \bigcup \psi-\mathsf{ConjGeo}(\RAAG, X \setminus \{y\})
\]
This is regular by \cref{lem: closure}.

\end{proof}}
\vspace{-10pt}
Regularity of the language $\mathsf{ConjGeo}\left(A_{\phi}, \widehat{X}\right)$ is surprisingly not universal across all non-length preserving cases. This can be seen by the following example.
\begin{prop}\label{prop:conjgeo not reg}
Let $A_{\Gamma} = \langle X \rangle$. For some $x,y \in V(\Gamma)$ such that $[x,y] \neq 1$, and $y$ is adjacent to all vertices in $\Gamma$ which are adjacent to $x$, define the following composition of a dominating transvection and inversion:
\begin{align*}
    \phi \colon X &\rightarrow X \\
    x &\mapsto xy \\
    y &\mapsto y^{-1} \\
    s &\mapsto s \quad \forall \; s \in X \setminus \{x,y\}^{\pm}.
\end{align*}
Then the language $\mathsf{ConjGeo}_{\phi}(\RAAG, X)$ is not context-free. 

\end{prop}
\comm{
\myRed{Following lemma not necessary?}
\begin{lemma}\label{lem:count xy}
Let $w \in \psi-\mathsf{ConjGeo}(\RAAG)$ as defined in \cref{prop:conjgeo not reg}. Define
\begin{enumerate}[(i)]
    \item $n(xy)$ := number of subwords $u$ of $w$ of the form $u = xy$, and
    \item $n(x)$ := number of $x$ letters in $w$ which don't occur as a subword $u$ of the form $u = xy$
\end{enumerate}
i.e. if $n(w)$ denotes the number of $x$ letters which occur in a word $w$, then $n(w) = n(xy) + n(x)$.\\
Then $n(xy) \leq n(x)$. 
\end{lemma}

\begin{proof}
Since $\RAAG$ is a direct product, we can write $w$ in the form $w = ab$ where $a \in \{x,y\}^{*}$, $b \in \{z\}^{*}$. Since any $\psi$-cyclic shifts will fix $b$, we only need to consider the prefix $a$.
\par
Suppose $n(xy) > n(x)$, and compute the $\psi$-shift
\[ w \xmapsto{\psi} \psi(w) = \psi(a)\psi(b) 
\]
Then any subwords $xy$ in $a$ get mapped to $x$ under $\psi$, which decreases the length of the word by $1$. Similarly any remaining $x$ letters get mapped to $xy$ under $\psi$, which increases the length of the word by $1$. Since $n(xy) > n(x)$, overall we decrease the length of the word, i.e. 
\[ |\psi(w) | < |w|
\]
This is a contradiction since $w \in \psi-\mathsf{ConjGeo}(\RAAG)$ and $w \sim_{\psi} \psi(w)$. 
\end{proof}
We now prove \cref{prop:conjgeo not reg}. }
\vspace{-10pt}
The proof of this proposition uses similar techniques to \cite{Ciobanu2013}, where the authors show that the spherical conjugacy language on the free group on two generators is not context-free.
\begin{proof} 
Suppose $\mathsf{ConjGeo}_{\phi}(\RAAG, X)$ is context-free, and consider the intersection $I = \mathsf{ConjGeo}_{\phi}(\RAAG, X) \cap L$, where $L = x^{+}(xy)^{+}x^{+}$. Since $L$ is regular, $I$ must be context-free by \cref{lem:intersection}. We claim that $ I = \{  x^{p}(xy)^{q}x^{r} \; | \; p \geq q, r \geq q \}$.
By \cref{cor:sequence} we need to check all possible sequences of $\phi$-cyclic permutations, commutation relations and free cancellations of $w = x^{p}(xy)^{q}x^{r} \in I$. 
\par 
Any word $v \in [w]_{\phi}$ must contain only the letters $x$ and $y$, and so we can assume no commutation relations exist in $v$. Also, $l(v) \geq l(w) = p+r+2q$, since $w$ is $\phi$-cyclically reduced. This leaves us with the following possible forms for $v$, by taking $\phi$-cyclic permutations of $w$:
\begin{enumerate}
    \item $v = (xy)^{r_{1}}x^{p}(xy)^{q}x^{r_{2}}$, where $r = r_{1} + r_{2}$. Here $l(v)>l(w)$.
    \item $v = y^{-1}(xy)^{r}x^{p}(xy)^{q-1}x$. Again $l(v) > l(w)$.
    \item $v = x(xy)^{r}x^{p}(xy)^{q-1}$. Here we need $p+2q+2r-1 \geq p+r+2q$, and so $r \geq 1$. If we keep repeating Cases 2 and 3, by moving letters from $(xy)^q$ to the front of the word, we find that $r \geq q$. This gives our first condition on $I$. 
    \item $v = (xy)^{p_{1}}x^{q}(xy)^{r}x^{p_{2}}$, where $p = p_{1} + p_{2}$. Here we require $q \leq r + p_{1}$, which always holds since $q \leq r$ and $p_{1} \geq 0$. 
    \item $v = x^{s}(xy)^{p}x^{q}(xy)^{r-s}$ where $1 \leq s \leq r$. Similar to Case 3, we can repeat this pattern of moving letters from $(xy)^{r}$ to find that $q \leq p$. This gives our second condition on $I$. 
    \item $v = (xy)^{q_{1}}x^{r}(xy)^{p}x^{q_{2}}$, where $q = q_{1} + q_{2}$. Here we need $q \leq p+q_{1}$, which holds since $q \leq p$ and $q_{1} \geq 0$. 
    \item $v = x^{s}(xy)^{q}x^{r}(xy)^{p-s}$ where $1 \leq s \leq p$. This case always works, since if $p-s>0$, then $l(v)>l(w)$. 
\end{enumerate}
\par 

\comm{
When considering all possible $\phi$-shifts of $w$, we notice that any word which could be shorter than $w$ must not start with an $(xy)^{i}$ block, since $w$ starts with a string of $x$ letters and $x < y$. Also, none of the letters can shuffle after any $\phi$-shifts, and so we only need to consider words of the form $x^{i}(xy)^{j}x^{l}$. Hence the only words which could be shorter are of the form
\[ x^{r}\phi\left(x^{p}\right)\phi\left(\left(xy\right)^{q}\right) = x^{r}\left(xy\right)^{p}x^{q} \quad \text{or} \quad \phi\left(\left(xy\right)^{q}\right)\phi\left(x^{r}\right)x^{p} = x^{q}\left(xy\right)^{r}x^{p}.
\]
In the first case, we need
\begin{align*}
    l(x^{p}(xy)^{q}x^{r}) &\leq l(x^{r}(xy)^{p}x^{q}) \\
    \Rightarrow p + r + 2q &\leq r + q + 2p \\
    \Rightarrow q &\leq p.
\end{align*}
Similarly for the second case we require $r \geq q$. These are the only cases we need to consider. 
\par }
This proves the claim, and it remains to prove that the language $I$ is not context-free. Let $k$ be the constant given by the Pumping Lemma for context-free languages (see \cite{Hopcroft}, Theorem 7.18, Page 281). Consider the word $W = x^{n}(xy)^{n}x^{n} \in I$ where $n > k$. Here $W$ contains 3 blocks, namely $x^{n}, (xy)^{n}$ and $x^{n}$. Hence by the Pumping Lemma, $W$ can be written as $W = stuvw$ where $l(tv) \geq 1, l(tuv) \leq k$ and $st^{i}uv^{i}w \in L$ for all $i \geq 0$. Since $l(tuv) \leq k < n$, $tuv$ cannot be part of more than 2 consecutive blocks. 
\par 
\textbf{Case 1: $tuv$ part of one block only}.\\
a) $tuv$ lies in first block. Then
\[ s \cdot tuv \cdot w = x^{n-j_{1}-j_{2}}\cdot x^{j_{1}} \cdot x^{j_{2}}(xy)^{n}x^{n}.
\]
When $i = 0$,
\[ suw = x^{n-j_{1}-j_{2}} \cdot x^{k} \cdot x^{j_{2}}(xy)^{n}x^{n} = x^{n+k-j_{1}}(xy)^{n}x^{n},
\]
for some $k < j_{1}$. Here $p < n$ and $q = n$, and so $suw \not \in I$, which is a contradiction.\\
b) $tuv$ lies in 2nd block. Then
\[ s \cdot tuv \cdot w = x^{n}(xy)^{i_{1}}\cdot (xy)^{i_{2}} \cdot (xy)^{i_{3}}x^{n},
\]
where $i_{1} + i_{2} + i_{3} = n$. If we let $i \geq 2$, then for $st^{i}uv^{i}w$ we have $q > n$ and $p = r = n$, and so $st^{i}uv^{i}w \not \in I$. \\
c) $tuv$ lies in third block. Similar to a), we have that for $i = 0$, $r < q$.
\par 
\textbf{Case 2: $tuv$ is part of more than one block.}\\
If either $t$ or $v$ contains both $x$ and $xy$ letters, then for $i \geq 2$, the word $st^{i}uv^{i}w$ contains at least 4 blocks alternating between powers of $x$ and powers of $xy$, and so can't lie in $I$. Hence either 
\[ t = x^{a}, \; v = (xy)^{b} \quad \text{or} \quad t = (xy)^{a}, \; v = x^{b}.
\]
i) $t$ in first block, $v$ in second block. Then
\[ stuvw = x^{i_{1}}\cdot x^{i_{2}}\cdot x^{i_{3}}(xy)^{j_{1}} \cdot (xy)^{j_{2}} \cdot (xy)^{j_{3}}x^{n},
\]
where $s = x^{i_{1}}, t = x^{i_{2}}, u = x^{i_{3}}(xy)^{j_{1}}, v = (xy)^{j_{2}}, w = (xy)^{j_{3}}x^{n}$, and $i_{1} + i_{2} + i_{3} = j_{1} + j_{2} + j_{3} = n$. When $i \geq 2$, then for $st^{i}uv^{i}w$ we have $q > n$ and $r = n$, so $st^{i}uv^{i}w \not \in I$.\\
ii) $t$ in second block, $v$ in third block. Again for $i \geq 2$, $q > p$.
\par
All cases give a contradiction, and so $I$ is not context-free. Hence the language $\mathsf{ConjGeo}_{\phi}(\RAAG, X)$ is not context-free. 
\end{proof}
\conjgeonotCF
\vspace{-10pt}
This result highlights interesting behaviour in the language of conjugacy geodesics in group extensions. For example, \cref{prop:conjgeo not reg} can be used to show that the language $\mathsf{ConjGeo}_{\phi}\left(F_{2}, \{x,y\}^{\pm}\right)$ is not context-free, with respect to a finite order non-length preserving automorphism. Consider the extension $F_{\phi} = F_{2} \rtimes_{\phi} \Z / 2\Z$ as in \cref{eqn:original split}. The language $\mathsf{ConjGeo}(F_{\phi}, \widehat{X})$ contains the following subsets:
\vspace{-10pt}
    \begin{itemize}
        \item[(i)] $\mathsf{ConjGeo}\left(F_{2}, \{x,y\}^{\pm}\right)$, which is regular \cite[Proposition 2.2]{Ciobanu2013}, and
        \item[(ii)] $\mathsf{ConjGeo}_{\phi}\left(F_{2}, \{x,y\}^{\pm}\right)$, which is not context-free.
    \end{itemize}
\vspace{-10pt}
Moreover, since $F_{\phi}$ is virtually free, then $\mathsf{ConjGeo}\left(F_{\phi}, \widehat{X}\right)$ must be regular \cite[Theorem 3.1]{Ciobanu2016}. This demonstrates that twisted conjugacy geodesics can exhibit different behaviour within a group extension. For RAAGs in general, the following question remains unclear.

\begin{question}
    Is $\mathsf{ConjGeo}\left(A_{\phi}, \widehat{X}\right)$ regular, where $A_{\phi} = A_{\Gamma} \rtimes_{\phi} \langle t \rangle$ (as in \cref{eqn: extension}), and $\phi \in \mathrm{Aut}(\RAAG)$ is a finite order transvection?
\end{question}

\subsection{Graph products}
\vspace{5pt}
We mention a more general result for graph products, with some restrictions on vertex groups, as a twisted adaptation of Theorem 2.4 in \cite{Ciobanu2016}. Further details can be found in \cite{Crowe2022}.
\par 
Recall $G_{\Gamma} = \langle X \rangle$ denotes a graph product with standard generating set $X = \bigcup^{k}_{i=1} X_{i}$, where each vertex group $G_{i}$ is generated by $X_{i}$.

\begin{defn}
For $\psi \in \mathrm{Aut}(G_{\Gamma})$, we say $\psi$ \emph{fixes the vertex groups} if for all $x_{i} \in X_{i}$, $\psi(x_{i}) \in G_{i}$ for every vertex group $G_{i} = \langle X_{i} \rangle$. 
\end{defn}

\begin{cor}\label{cor:graphprod geo regular}
Let $G_{\Gamma}$ be a graph product with vertex groups $\langle X_{i} \rangle$. Let $G_{\phi}$ be a virtual graph product of the form $G_{\phi} = G_{\Gamma} \rtimes_{\phi} \langle t \rangle$ (as defined in \eqref{eqn: extension}) with standard generating set $\widehat{X}$ where:
\begin{enumerate}
    \item[(i)] $\mathsf{Geo}(G_{i}, X_{i})$ is regular for each vertex group,
    \item[(ii)] $\phi \in \mathrm{Aut}(G_{\Gamma})$ is length preserving and fixes the vertex groups, and
    \item[(iii)] $\mathsf{CycGeo}_{\phi^{l}}(G_{i}, X_{i}) = \mathsf{ConjGeo}_{\phi^{l}}(G_{i}, X_{i})$ for each vertex group, for all $0 \leq l \leq m-1$.
\end{enumerate}
Then $\mathsf{ConjGeo}\left(G_{\phi}, \widehat{X}\right)$ is regular.
\end{cor}

\section{The Conjugacy shortlex language in virtual graph products}\label{section: conj sl}
We now turn our attention to the language $\mathsf{ConjSL}$ in virtual graph products. As a slight detour, we first study acylindrically hyperbolic RAAGs and RACGs (see \cite{Osin2015} for background on acylindrically hyperbolic groups).
\vspace{-10pt}

\subsection{$\mathcal{AH}$-accessability}
We begin with the following result by Antol\'in and Ciobanu.

\begin{thm}\label{thm:acy hyp languages}(\cite{Antolin2016a}, Theorem 1.5)\\
If $G$ is non-elementary acylindrically hyperbolic, then $\mathsf{ConjSL}(G, X)$ is not unambiguous context-free with respect to any generating set $X$.

\end{thm}
\vspace{-10pt}
We note that the algebraic type of the conjugacy growth series, for acylindrically hyperbolic groups, is still open. 
\par
\cref{thm:acy hyp languages} will allow us to prove the following.

\acyhyp
\vspace{-10pt}
We will show in this case that $G$ must necessarily be acylindrically hyperbolic, using a property of groups known as \emph{acylindrically hyperbolic accessibility} (see \cite{Abbott2019} for further details). In a correction note of \cite{Minasyan2015}, Minasyan and Osin proved a criterion for when extensions of acylindrically hyperbolic groups are acylindrically hyperbolic. We note that it is still an open question in general as to whether acylindrical hyperbolicity is a quasi-isometry invariant.
\begin{thm}\label{thm:acy hyp extensions}(\cite{Minasyan2015}, Lemma 6). If $H \trianglelefteq G$ of finite index, where $H$ is 
\begin{enumerate}
    \item[(i)] acylindrically hyperbolic, and
    \item[(ii)] $\mathcal{AH}$-accessible,
\end{enumerate}
then $G$ is acylindrically hyperbolic.
\end{thm}
\comm{
\myRed{TO DO:} AH-accessible and Genevois combined arguments. Definition of AH-accessible.\\}
\comm{
First let's consider RAAGs. 
Also in \cite{Abbott2019} is the following result:

\begin{lemma}\label{lemma: RAAG AH-acc}(\cite{Abbott2019}, Lemma 7.16)\\
If $\Gamma$ is connected, then $A_{\Gamma}$ is strongly $\mathcal{AH}$-accessible, and hence $\mathcal{AH}$-accessible. 

\end{lemma}
So immediately we can say that if $\RAAG$ is acylindrically hyperbolic on a connected graph, then the extension $A_{\phi}$ will also be acylindrically hyperbolic. 
\par}
\vspace{-10pt}
RAAGs are acylindrically hyperbolic if they are not cyclic or a direct product (see \cite{Osin2015}). To study the $\mathcal{AH}$-accessible property in RAAGs and RACGs, we need to look at \emph{hierarchically hyperbolic groups} (HHGs). These were first defined by Behrstock, Hagen and Sisto in 2017 (see \cite{Behrstock2019}, Definition 1.21 for the definition of a HHG, and \cite{Sisto2017} for a survey on these groups). By Proposition B and Theorem G from \cite{Behrstock2017}, RAAGs and RACGs are HHGs. 

\begin{prop}\label{prop:HHGS accessible}
Every HHG is $\mathcal{AH}$-accessible.

\end{prop}

\begin{proof}
By definition (see \cite{Abbott2019}), a group is $\mathcal{AH}$-accessible if the poset $\mathcal{AH}(G)$, consisting of hyperbolic structures corresponding to acylindrical actions, contains the largest element. Theorem A from \cite{Abbott2021} states that every HHG admits a largest acylindrical action, and so is $\mathcal{AH}$-accessible.
\end{proof}

\begin{proof}[Proof of \cref{cor:acyhyp2}] Let $H = \RAAG$ be acylindrically hyperbolic, where $\RAAG$ is a RAAG. By \cref{prop:HHGS accessible}, $H$ is $\mathcal{AH}$-accessible, and so $G$ is acylindrically hyperbolic by \cref{thm:acy hyp extensions}. Hence by \cref{thm:acy hyp languages}, $\mathsf{ConjSL}(G, X)$ is not unambiguous context-free for any generating set $X$. The proof is analogous for RACGs.
\end{proof}

\comm{
\begin{cor}\label{cor:acy hyper}
Let $A_{\Gamma}$ be directly indecomposable. Then for the extension $A_{\Gamma} \trianglelefteq A_{\phi}$ (of finite index), we have that $\mathsf{ConjSL}(A_{\phi})$ is not unambiguous context-free. 

\end{cor}

\begin{proof}
By \cref{lemma: RAAG acy hyp} and \cref{cor:ah access both}, $A_{\Gamma}$ is acylindrically hyperbolic and $\mathcal{AH}$-accessible. So \cref{thm:acy hyp extensions} implies that $A_{\phi}$ is acylindrically hyperbolic. Hence by \cref{thm:acy hyp languages}, $\mathsf{ConjSL}(A_{\phi})$ is not unambiguous context-free. 
\end{proof}

\begin{cor}
Let $W_{\Gamma}$ be a RACG which is acylindrically hyperbolic. Then for the extension $W_{\Gamma} \trianglelefteq W_{\phi}$ (of finite index), we have that $\mathsf{ConjSL}(W_{\phi})$ is not unambiguous context-free. 
\end{cor}

\begin{proof}
Similar to proof of \cref{cor:acy hyper}. 
\end{proof}}
\comm{
These results have been proven independently by Genevois without the requirement of $\mathcal{AH}$-accessibility. 

\begin{thm}(\cite{Genevois2018}, Corollary 1.8)\\
The semidirect product $\RAAG \rtimes_{\phi} H$ is acylindrically hyperbolic if and only if:
\begin{itemize}
    \item $\RAAG$ is acylindrically hyperbolic
    \item the kernel of $H \xrightarrow{\phi}  \mathrm{Aut}(\RAAG) \rightarrow \mathrm{Out}(\RAAG) $ is finite.
\end{itemize}
Similarly the semi-direct product $W_{\Gamma} \rtimes_{\phi} H$ is acylindrically hyperbolic if and only if:
\begin{itemize}
    \item $W_{\Gamma}$ is acylindrically hyperbolic
    \item the kernel of $H \xrightarrow{\phi}  \mathrm{Aut}(\RAAG) \rightarrow \mathrm{Out}(\RAAG) $ is finite.
\end{itemize}
\end{thm}
For our purposes, since we're only considering finite extensions, $H$ is finite and so the second requirement is always true. }
\subsection{Virtual graph products}
\vspace{5pt}
We now return our attention to extensions of graph products $G_{\phi} = G_{\Gamma} \rtimes_{\phi} \langle t \rangle$ as defined in \eqref{eqn: extension}. To study $\mathsf{ConjSL}\left(G_{\phi}, \widehat{X}\right)$, we again focus on twisted conjugacy in the base group (recall \cref{def:ConjSL}). The following result will be useful later, and motivates the study of $\mathsf{ConjSL}\left(G_{\phi}, \widehat{X}\right)$.

\begin{prop}\label{prop:2 cases regular not CF}
Let $G = N \rtimes_{\phi} \langle t \rangle$ be a group extension as defined in \eqref{eqn:original split}. 
\begin{enumerate}
    \item If $\mathsf{ConjSL}_{\phi^{l}}(N, X)$ is regular \textbf{for all $l$}, then $\mathsf{ConjSL}\left(G, \widehat{X}\right)$ is regular.
    \item If $\mathsf{ConjSL}_{\phi^{l}}(N, X)$ is not context-free (or unambiguous context-free) \textbf{for some $l$}, then $\mathsf{ConjSL}\left(G, \widehat{X}\right)$ is not context-free (or unambiguous context-free).
\end{enumerate}
\end{prop}
\begin{proof}
    Recall by \cref{def:ConjSL} that
\[ \mathsf{ConjSL}\left(G, \widehat{X}\right) = \bigcup_{l=a}^{b}\{Ut^{l} \mid U \in \mathsf{ConjSL}_{\phi^{l}}(N, X) \}.
\]
First suppose  $\mathsf{ConjSL}_{\phi^{l}}(N, X)$ is regular for all $l$. For fixed $l$, the languages $\{t^{l}\}$ (singleton language) and $\mathsf{ConjSL}_{\phi^{l}}(N, X)$ are both regular, and so the concatenation of these languages is regular by \cref{lem: closure}. Since $t$ has finite order, the union of these regular languages across all possible $l$ must also be regular, which proves $(i)$.

For $(ii)$, choose $l$ such that $\mathsf{ConjSL}_{\phi^{l}}(N, X)$ is not context-free (this exists by assumption). Suppose the set $\{Ut^{l} \mid U \in \mathsf{ConjSL}_{\phi^{l}}(N, X) \} $ is context-free. Then by \cref{prop:suffixes}, $\mathsf{ConjSL}_{\phi^{l}}(N, X)$ is context-free. This contradicts our assumption and so $\{Ut^{l} \mid U \in \mathsf{ConjSL}_{\phi^{l}}(N, X) \}$ is not context-free. 
\par 
Now suppose $\mathsf{ConjSL}\left(G, \widehat{X}\right)$ is context-free. Then consider
\[ \mathsf{ConjSL}\left(G, \widehat{X}\right) \cap X^{*}\cdot t^{l} = \mathsf{ConjSL}_{\phi^{l}}(N, X)\cdot t^{l},
\]
for $l$ such that $\mathsf{ConjSL}_{\phi^{l}}(N, X)$ is not context-free. By \cref{lem:intersection}, $\mathsf{ConjSL}_{\phi^{l}}(N, X)\cdot t^{l}$ is context-free. But this contradicts our previous claim, and so $\mathsf{ConjSL}\left(G, \widehat{X}\right)$ is not context-free. The proof is analogous when we replace context-free with unambiguous context-free.
\end{proof}

\comm{
\begin{lemma}\label{prop:suffixes}
If $L$ is a context-free language, then so is 
\[ \mathsf{suff}(L) = \{ u \in \Sigma^{*} \; | \; vu \in L \}
\]
\end{lemma}

\begin{proof}
First note that the set of prefixes $\mathsf{pref}(L) = \mathsf{suff}(L)^{R}$, so it remains to show that the set of prefixes is context-free by \cref{lem:homs reverse}. First we take a copy of our alphabet and denote as:
\[ \overline{\Sigma} = \{\overline{a} \; | \; a \in \Sigma^{*} \}
\]
Define the following morphisms
\begin{align*}
    g,h: (\Sigma \cup \overline{\Sigma}^{*}) &\rightarrow \Sigma^{*} \\
    g(a) &= a \\
    g(\overline{a}) &= \epsilon \\
    h(a) &= a \\
    h(\overline{a}) &= a
\end{align*}
Then 
\[ \mathsf{pref}(L) = g(h^{-1}(L) \cap \Sigma^{*}\overline{\Sigma}^{*})
\]
The result then follows by \cref{lem: closure} and \cref{lem:homs reverse}. 

\end{proof}
We now prove the second statement of \cref{prop:2 cases regular not CF}.
\begin{proof}
First note
\[ \mathsf{ConjSL}(G_{\phi}) = \bigcup_{l=1}^{m}\{t^{l}U \; | \; U \in \phi^{l}-\mathsf{ConjSL}(G_{\Gamma}, X) \} 
\]
Choose $l$ such that $\phi^{l}-\mathsf{ConjSL}(G_{\Gamma})$ is not context-free (this exists by assumption). Suppose the set $\{t^{l}U \; | \; U \in \phi^{l}-\mathsf{ConjSL}(G_{\Gamma}, X) \} $ is context-free. Then by Proposition \ref{prop:suffixes}, $\phi-\mathsf{ConjSL}(G_{\Gamma}, X)$ is context-free. But our assumption contradicts this claim.\\
Now suppose $\mathsf{ConjSL}(G_{\phi})$ is context-free. Then consider
\[ \mathsf{ConjSL}(G{\phi}) \cap t^{l}S^{*} = t^{l}\phi^{l}-\mathsf{ConjSL}(G_{\Gamma}, X)
\]
for $l$ such that $\phi^{l}-\mathsf{ConjSL}(G_{\Gamma})$ is not context-free. By \cref{lem:intersection}, it must be that $t^{l}\phi^{l}-\mathsf{ConjSL}(G_{\Gamma}, X)$ is context-free. But this contradicts our claim, and so $\mathsf{ConjSL}(G_{\phi})$ is not context-free.

\end{proof}}
\vspace{-10pt}
We now consider more general properties of $\mathsf{ConjSL}\left(G_{\phi}, \widehat{X}\right)$. We first consider irreducible graph products, that is, graph products which cannot be written as a direct product, using the following dichotomy by Minasyan and Osin. 
\comm{
To do so we use the following dichotomy from Osin \cite{Osin2015}:

\begin{lemma}\label{lemma: RAAG acy hyp}
Every RAAG is either:
\begin{itemize}
    \item infinite cyclic 
    \item directly decomposable, i.e. $A_{\Gamma} = A_{\Gamma_{1}} \times A_{\Gamma_{2}}$
    \item acylindrically hyperbolic
\end{itemize}

\end{lemma}
Note there does not exist a RAAG which is both a direct product and acylindrically hyperbolic due to the following result and the fact that RAAGs are torsion free.
\begin{thm}(\cite{Osin2015}, Corollary 7.2)\\
If $G$ is acylindrically hyperbolic and decomposes as a product $G = G_{1} \times G_{2}$, then one of the product groups must be finite.

\end{thm}}
\begin{cor}\label{cor:dicotomy}(\cite{Minasyan2015}, Corollary 2.13) Let $G_{\Gamma}$ be an irreducible graph product, where $V(\Gamma) \geq 2$. Then $G_{\Gamma}$ is either virtually cyclic or acylindrically hyperbolic. 

\end{cor}
\vspace{-10pt}
Recall the \emph{spherical conjugacy growth series}, denoted $\tilde{\sigma}(G,X)$, to be the strict growth series of $\mathsf{ConjSL}(G,X)$, for a group $G$ with respect to a generating set $X$.
\comm{
\vspace{-10pt}
Note some of these properties are not disjoint. For example if we take the RACG on the following defining graph
\[\begin{tikzcd}
	x & y & z
	\arrow[no head, from=1-1, to=1-2]
	\arrow[no head, from=1-2, to=1-3]
\end{tikzcd}\]
then $W_{\Gamma} \cong D_{\infty} \times \Z_{2}$, which is both a direct product and virtually cyclic. }
\comm{
\subsection{Direct products}
In this section we focus on RAAGs only. We start with the following motivation:
\begin{prop}
Let $A_{\Gamma} = F_{n} \times H$ for some RAAG $H$. Then $\mathsf{ConjSL}(A_{\phi})$ is not context-free for any finite extension of the form (\ref{eqn: extension}). 
\end{prop}

\begin{proof}
Let $\phi \in \mathrm{Aut}(\RAAG)$ be of finite order $t$. Then 
\[ \phi^{t}-\mathsf{ConjSL}(F_{n} \times H) \cong \mathsf{ConjSL}(F_{n} \times H)
\]
$\mathsf{ConjSL}(F_{n} \times H)$ is not context-free by Corollary 2.3 of \cite{Ciobanu2013}. Hence for the extension $A_{\phi}$, $\mathsf{ConjSL}(A_{\phi})$ is not context-free by \cref{prop:2 cases regular not CF}. 
\end{proof}

\begin{thm}\label{thm: extension not UCF}
Let $A_{\Gamma} = A_{\Gamma_{1}} \times A_{\Gamma_{2}}$ where $A_{\Gamma_{1}}, A_{\Gamma_{2}}$ are directly indecomposable RAAGs such that $\Gamma_{1}$ and $\Gamma_{2}$ contain at least 2 vertices. Then $\mathsf{ConjSL}(A_{\phi})$ of the finite extension is not unambiguous context-free.
\end{thm}

\begin{proof}
First note by assumption that $A_{\Gamma_{1}}, A_{\Gamma_{2}}$ must be acylindrically hyperbolic by \cref{lemma: RAAG acy hyp}. Let $S_{1}, S_{2}$ be finite generating sets for $A_{\Gamma_{1}}, A_{\Gamma_{2}}$ respectively. We claim that
\[ \mathsf{ConjSL}(A_{\Gamma}) \cap S^{\ast}_{1} = \mathsf{ConjSL}(A_{\Gamma_{1}})
\]
The implication $\subseteq$ is clear. For the other direction, we recall that two cyclically reduced words representing elements in a RAAG are conjugate if and only if they are related by a finite sequence of cyclic permutations and commutation relations. In particular, $\mathsf{ConjSL}$ for any RAAG is precisely the set of all cyclically reduced words which are shortlex shortest in their conjugacy class. 
\par 
Suppose there exists a word $x \in \mathsf{ConjSL}(A_{\Gamma_{1}})$ where $x \not \in \mathsf{ConjSL}(A_{\Gamma})$. Then there exists a word $y \in \mathsf{ConjSL}(A_{\Gamma})$ such that $x \sim y$ and $y <_{SL} x$. This means there exists a finite sequence of cyclic permutations and commutation relations from $x$ to $y$. Neither of these moves changes the letters in the word, and so $y$ is precisely a word over $S_{1}$. Hence by the $\subseteq$ implication, $y \in \mathsf{ConjSL}(A_{\Gamma_{1}})$. But since $y <_{SL} x$, this gives a contradiction that $x \not \in \mathsf{ConjSL}(A_{\Gamma_{1}})$. 
\par
Suppose $\mathsf{ConjSL}(A_{\Gamma})$ is unambiguous context-free. Then by \cref{lem:intersection}, this implies $\mathsf{ConjSL}(A_{\Gamma_{1}})$ is also unambiguous context-free. But $A_{\Gamma_{1}}$ is acylindrically hyperbolic and so this language is not unambiguous context-free by \cref{thm:acy hyp languages}, hence a contradiction.\\
Therefore $\mathsf{ConjSL}(A_{\Gamma})$ is not unambiguous context-free. Now by applying \cref{prop:2 cases regular not CF}, by taking $\phi^{l}$ to be trivial, we get that $\mathsf{ConjSL}(A_{\phi})$ is not unambiguous context-free.

\end{proof}
We can now say that for any RAAG, either acylindrically hyperbolic or a direct product, that the language $\mathsf{ConjSL}(A_{\phi})$ is not unambiguous context-free. }
\begin{lemma}
Let $G_{\Gamma}$ be an irreducible graph product, where $V(\Gamma) \geq 2$. Let $G_{\phi}$ be a virtual graph product of the form $G_{\phi} = G_{\Gamma} \rtimes_{\phi} \langle t \rangle$ (as defined in \eqref{eqn: extension}), where $\widehat{X}$ is the standard generating set for $G_{\phi}$, endowed with an order. If $G_{\phi}$ is not virtually abelian, the language $\mathsf{ConjSL}\left(G_{\phi}, \widehat{X}\right)$ is not unambiguous context-free. Otherwise, the spherical conjugacy growth series $\tilde{\sigma}\left(G_{\phi}, \widehat{X}\right)$ is rational. 
\end{lemma}
\vspace{-10pt}
We recall that even if the series $\tilde{\sigma}\left(G_{\phi}, \widehat{X}\right)$ is rational, this does not necessarily imply regularity of the language $\mathsf{ConjSL}\left(G_{\phi}, \widehat{X}\right)$. This can be seen by \cref{language:qi} for virtually abelian groups.
\vspace{-10pt}
\begin{proof}
By \cref{cor:dicotomy}, we need to check two cases. If $G_{\Gamma}$ is virtually cyclic, it is necessarily virtually abelian. Therefore $G_{\phi}$ is virtually abelian, and so $\tilde{\sigma}\left(G_{\phi}, \widehat{X}\right)$ is rational \cite{Evetts2019}. If $G_{\Gamma}$ is acylindrically hyperbolic but not virtually cyclic, then $\mathsf{ConjSL}(G_{\Gamma}, X)$ is not unambiguous context-free by \cref{thm:acy hyp languages}. Hence $\mathsf{ConjSL}\left(G_{\phi}, \widehat{X}\right)$ is not unambiguous context-free by \cref{prop:2 cases regular not CF}, by considering $l = 0$. 
\end{proof}
\vspace{-10pt}
When proving results for graph products, a common technique is to study whether certain properties are closed under taking graph products. For example, it was shown in \cite[Theorem 3.1]{Ciobanu2013}, that regularity of the conjugacy geodesic language is preserved by graph products.
\par 
This method is less clear when we consider the language $\mathsf{ConjSL}(G_{\Gamma}, X)$. For example, consider the graph product $G_{\Gamma} = F_{2} = \Z \ast \Z$. Here $\mathsf{ConjSL}(\Z, X)$ is regular over any generating set $X$, however $\mathsf{ConjSL}(F_{2}, \{x, y\}^{\pm})$ is not context-free by \cite[Proposition 2.2]{Ciobanu2013} (where $\{x,y\}$ is a free basis). We generalise this result further by considering subsets of the vertex groups in a graph product. 

\begin{defn}
Let $\Gamma$ be a graph, and let $S \subseteq V(\Gamma)$. We say $\Lambda \subseteq \Gamma$ is an \emph{induced subgraph} of $\Gamma$ if $\Lambda$ is the graph with vertices $S$, such that each pair of vertices in $\Lambda$ are connected by an edge in $\Lambda$ if and only if they are connected by an edge in $\Gamma$.
\end{defn}
\vspace{-10pt}
Let $G_{\Gamma}$ be a graph product with vertex groups $G_{i} = \langle X_{i} \rangle$. Let $\Lambda \subseteq \Gamma$ be an induced subgraph of $\Gamma$. We define $G_{\Lambda}$ to be the \emph{induced graph product}, with respect to $\Lambda$. Here we let $X_{\Lambda}$ denote the standard generating set for $G_{\Lambda}$, where $X_{\Lambda} = \bigcup_{v \; \in \; V(\Lambda)} X_{v}$.
\begin{defn}
Let $G_{\Gamma} = \langle X \rangle$ be a graph product, with vertex groups $G_{v}$ for every $v \in V(\Gamma)$. Let $g = g_{1}\dots g_{n} \in X^{\ast}$ be a geodesic, where each $g_{i} \in G_{v}$ for some $v \in V(\Gamma)$. We define the \emph{support} of $g$ as
\[ \text{supp}(g) = \{ v \in V(\Gamma) \; | \; \text{there exists} \; i \in \{1, \dots, n \}  \;\text{such that} \; g_{i} \in G_{v} \setminus \{1\} \}.
\]
\end{defn}
\vspace{-10pt}
It was shown in Lemma 3.12 of \cite{Ferov2016} that if two cyclically reduced elements $x,y \in G_{\Gamma}$ are conjugate, then $\text{supp}(x) = \text{supp}(y)$. We use this result to prove the following.
\begin{prop}\label{prop:new thm}
Let $G_{\Gamma}$ be a graph product with standard generating set $X$, endowed with an order. For an induced subgraph $\Lambda \subseteq \Gamma$, let $G_{\Lambda}$ be the induced graph product with respect to $\Lambda$, with standard generating set $X_{\Lambda}$, with an induced ordering from $X$. If $\mathsf{ConjSL}(G_{\Lambda}, X_{\Lambda})$ is not regular, then $\mathsf{ConjSL}(G_{\Gamma}, X)$ is not regular. This result also holds when we replace regular with either unambiguous context-free or context-free.
\end{prop}
\vspace{-10pt}
\begin{proof}
We claim that
\[ \mathsf{ConjSL}(G_{\Gamma}, X) \cap X_{\Lambda}^{\ast} = \mathsf{ConjSL}(G_{\Lambda}, X_{\Lambda}).
\]
The implication $\subseteq$ is clear. For the other direction, suppose there exists a word \\ $x \in \mathsf{ConjSL}(G_{\Lambda}, X_{\Lambda})$ such that $x \not \in \mathsf{ConjSL}(G_{\Gamma}, X)$. Then there exists a word $y \in \mathsf{ConjSL}(G_{\Gamma}, X)$ such that $x \sim y$ and $y <_{SL} x$. By definition, both $x$ and $y$ are cyclically reduced in the graph product, and hence $\text{supp}(x) = \text{supp}(y)$. Since $\text{supp}(x)$ contains only vertices from the vertex groups which generate $G_{\Lambda}$, so must $\text{supp}(y)$. Hence $y$ is precisely a word over $X_{\Lambda}$, and so $y \in \mathsf{ConjSL}(G_{\Lambda}, X_{\Lambda})$ by the $\subseteq$ implication. Now we have a contradiction since $x \in \mathsf{ConjSL}(G_{\Lambda}, X_{\Lambda})$. 
\par 
Now suppose $\mathsf{ConjSL}(G_{\Lambda}, X_{\Lambda})$ is not regular, but $\mathsf{ConjSL}(G_{\Gamma}, X)$ is regular. Then we reach a contradiction by \cref{lem: closure}, which completes the proof. This argument also holds for unambiguous context-free and context-free using \cref{lem:intersection}.
\end{proof}

\begin{cor}\label{cor:graph product extension conjsl}
    Let $G_{\phi}$ be a virtual graph product of the form $G_{\phi} = G_{\Gamma} \rtimes_{\phi} \langle t \rangle$ (as in \cref{eqn: extension}), where $\widehat{X}$ is the standard generating set for $G_{\phi}$, endowed with an order. For an induced subgraph $\Lambda \subseteq \Gamma$, let $G_{\Lambda}$ be the induced graph product with respect to $\Lambda$, with standard generating set $X_{\Lambda}$, with an induced ordering from $X$. If $\mathsf{ConjSL}(G_{\Lambda}, X_{\Lambda})$ is not regular, then $\mathsf{ConjSL}\left(G_{\phi}, \widehat{X}\right)$ is not regular. This result also holds when we replace regular with either unambiguous context-free or context-free.
\end{cor}

\begin{proof}
This result follows from \cref{prop:new thm} and \cref{prop:2 cases regular not CF}, by considering $l=0$. 
\end{proof}

\comm{
\subsection{Virtual RACGs}
We can say a bit more about the conjugacy growth series for RACGs. Behrstock, Hagen and Sisto proved in \myRed{CITE} that any RACG must be either relatively hyperbolic or thick. We can use this, alongside the following result by Gekhtman and Yang \myRed{CITE} to categorise RACGs.
\begin{cor}
The conjugacy growth series of a non-elementary relatively hyperbolic group, with respect to any finite generating set, is transcendental.  
\end{cor}

\begin{thm}
Let $L = \mathsf{ConjSL}(W_{\phi}, X)$. Then either
\begin{itemize}
    \item $f_{L}(z)$ is rational if $W_{\phi}$ is virtually cyclic or virtually abelian, or
    \item $f_{L}(z)$ is transcendental if $W_{\Gamma}$
\end{itemize}
\end{thm}}

\comm{
For RAAGs, we can prove an even stronger result. We begin with some example where twisted $\mathsf{ConjSL}$ languages are not context-free.
\par 
\begin{defn}
A graph is \emph{asymmetric} or \emph{rigid} if there are no non-trivial graph automorphisms. 
\end{defn}

\begin{prop}\label{prop:inversion not CF}
Let $\RAAG$ be defined on a graph $\Gamma$ where
\begin{enumerate}
    \item $|\Gamma| \geq 2$, and 
    \item There exists a non-edge between at least 2 vertices in $\Gamma$, i.e. there exists two generators $a, b \in \RAAG$ such that $[a,b] \neq 1$.
\end{enumerate}
Let $\phi \in \mathsf{Aut}(\RAAG)$ be a composition of inversions. Then $\phi-\mathsf{ConjSL}(\RAAG)$ is not context-free.
\end{prop}
Note that if $\phi$ is trivial and $|\Gamma| = 2$, then $A_{\Gamma} \cong F_{2}$ and the result follows by Proposition 2.2 of \cite{Ciobanu2013}. 

\begin{proof}
Assume $a,b$ are two generators in $\RAAG$ such that $[a,b] \neq 1$. The proof of this Proposition is identical to the proof for $F_{2}$ once we verify that the parameters on the set $I$ are the same. For notation, define a map
\begin{align*}
    \widehat{\phi}: \{1,-1\} &\rightarrow \{1,-1\} \\
    x & \mapsto
    \begin{cases}
    x & \phi(x) = x \\
    -x & \phi(x) = x^{-1}
    \end{cases}
\end{align*}
Consider a word of the form
\[ w = a^{p}b^{l}a^{q}b^{j} \in I
\]
Here we assume $p,l,q,j > 0$. For $w \in \phi-\mathsf{ConjSL}$, we must check that no possible $\phi$-shifts form a shorter word. Since $\phi$ consists of inversions, letters will always be unable to shuffle after shifts. We see that the only possible words which could be shorter than $w$ are
\[ \phi(a^{q})\phi(b^{j})a^{p}b^{j}
\]
and 
\[ a^{q}b^{j}\phi(a^{p})\phi(b^{l})
\]
In the first case, this word is bigger than $w$ if either
\begin{itemize}
    \item $p > \widehat{\phi}(q)$, or
    \item $p=\widehat{\phi}(q)$, $l \leq \widehat{\phi}(j)$
\end{itemize}
Now if $\phi(q) = q^{-1}$, then the first point will always be true since $p$ is positive, and so we can rewrite this as $p > q$. Also the second case can only be true if $\phi$ fixes $q$ and $j$, and so overall these two conditions become:
\begin{itemize}
    \item $p > q$ or
    \item $p = q$, $l \leq j$
\end{itemize}
For the second case, this word is bigger than $w$ if either of the conditions above hold. \\
This matches the conditions required to use the proof for $F_{2}$. 

\end{proof}

\begin{cor}
Let $\Gamma$ be an asymmetric graph with at least 2 vertices, and let $\RAAG$ be the RAAG defined by $\Gamma$. Let $\psi \in \mathsf{Aut}(\RAAG)$ be length preserving. Then $\psi-\mathsf{ConjSL}(\RAAG)$ is not CF, and hence
\[ \mathsf{ConjSL}(A_{\phi}) \cong \RAAG \rtimes_{\phi} \mathbb{Z}_{2}
\] 
is not CF. 
\end{cor}

\begin{proof}
By Proposition \ref{prop:combo len p} and the assumption that $\Gamma$ has no non-trivial graph automorphisms, we can assume $\psi$ is a combination of inversions only (so has order 2). \\
We can also assume there exists at least 2 vertices in $\Gamma$ not connected by an edge, since otherwise $\Gamma \cong K_{n}$ which is not asymmetric. The result then follows by Proposition \ref{prop:inversion not CF}. 
\end{proof}
We can now show a more general result that covers a surprisingly large number of cases.

\begin{thm}(Erdos, Renyi)\\
Almost all graphs are asymmetric.
\end{thm}

\begin{cor}
For almost all RAAGs $\RAAG$, there exists a length preserving automorphism $\phi \in \mathsf{Aut}(\RAAG)$ such that $\phi-\mathsf{ConjSL}(\RAAG)$ is not context-free, and hence the extension $\mathsf{ConjSL}(\G)$ is not CF. 
\end{cor}
This in a sense answers in the affirmative cases that for any RAAG, there exists an automorphism $\phi$ such that $\mathsf{ConjSL}(A_{\phi})$ is not context-free.
\par}
\vspace{-10pt}
For RAAGs, we can prove a stronger result by considering the trivial automorphism. By showing $\mathsf{ConjSL}(\RAAG, X)$ is not context-free for any RAAG, we find that the equivalent language in the extension, $\mathsf{ConjSL}\left(A_{\phi}, \widehat{X}\right)$, is also not context-free (recall \cref{prop:2 cases regular not CF}).

\begin{prop}\label{thm: RAAG not CF}
Let $A_{\Gamma}$ be a RAAG which is not free abelian. Let $X$ be the standard generating set for $\RAAG$, endowed with an order. Then $\mathsf{ConjSL}(A_{\Gamma}, X)$ is not context-free.
\end{prop}

\begin{proof}
By assumption, there exists two generators $x,y \in V(\Gamma)$ such that the set $S = \{x,y\}^{\pm}$ generates a free group of rank two. We claim that
\[ \mathsf{ConjSL}(\RAAG, X) \cap S^{\ast} = \mathsf{ConjSL}(F_{2}, S).
\]
The proof follows similarly to the proof of \cref{prop:new thm}, recalling that in RAAGs, cyclically reduced conjugate elements are related by cyclic permutations and commutation relations.
\par
Suppose $\mathsf{ConjSL}(\RAAG, X)$ is context-free. By \cref{lem:intersection}, this would imply $\mathsf{ConjSL}(F_{2}, S)$ is also context-free. This is a contradiction since $\mathsf{ConjSL}(F_{2}, S)$ is not context-free by \cite[Proposition 2.2]{Ciobanu2013}. Therefore $\mathsf{ConjSL}(\RAAG, X)$ is not context-free. 
\end{proof}
\begin{cor}\label{cor: v RAAGs not CF}
Let $A_{\phi}$ be a virtual RAAG of the form $A_{\phi} = \RAAG \rtimes_{\phi} \langle t \rangle$ (as in \cref{eqn: extension}), where $\widehat{X}$ is the standard generating set for $A_{\phi}$, endowed with an order. If $A_{\Gamma}$ is not abelian, then $\mathsf{ConjSL}\left(A_{\phi},\widehat{X}\right)$ is not context-free. 

\end{cor}

\begin{proof}
The result follows by \cref{thm: RAAG not CF} and \cref{prop:2 cases regular not CF}.
\end{proof}

\comm{
\section{Summary of results}
Previous results: $\tilde{\sigma}$ = conjugacy growth series. $\ast$ = any generating set. PT = piecewise testable.
\begin{center}
    \begin{tabular}{ | c | c | c |c | c |}
    \hline
    \multicolumn{5}{c}{Previous results} \\
    \hline
    Reference & Group & $\mathsf{ConjGeo}$ & $\mathsf{ConjSL}$ & $\tilde{\sigma}$  \\
    \hline
    Thm 4.2 \cite{Ciobanu2016} & Virtually cyclic & Regular^{$\ast$} & Regular^{$\ast$} & Rational \\
    \hline
    Thm 3.1 \cite{Ciobanu2016}, Thm 1.1 \cite{Antolin2016a} & Hyperbolic & Regular^{$\ast$} & & Transcendental (if not v. cyclic) \\
    \hline
   Proposition 3.3 \cite{Ciobanu2016}, \cite{Evetts2019} & Virtually abelian & PT & & Rational \\
   \hline
   \cite{Ciobanu2013}, \cite{Ciobanu2016} & Graph products & Regular & Example: not regular & Example: not rational \\
   \hline 
   \cite{Contract22} & Relatively hyperbolic & & Not UCF* & Transcendental* \\
   \hline
   \cite{Antolin2016a} & Acylindrically hyperbolic & & Not UCF* & \\
   \hline 
\end{tabular}
\end{center}

\begin{center}
    \begin{tabular}{ | c | c | c |c | }
    \hline
    \multicolumn{4}{c}{New results (all w.r.t standard generating set)} \\
    \hline
    Reference & Group & $\mathsf{ConjGeo}$ & $\mathsf{ConjSL}$  \\
    \hline
    \cref{thm: RAAG not CF} & RAAGs & Previous & Not CF \\
    \hline
   \cref{thm:conjslgraph} & Graph products & Previous & Not UCF \\
   \hline
   \cref{thm:ConjGeo Regular} & V. RAAGs $A_{\phi}$ & Regular** & Not CF \\
   & V. RACGs $W_{\phi}$ & Regular ** & Not UCF  \\
   \hline
    \cref{cor:graphprod geo regular} & V. graph products $G_{\phi}$ & Regular (some cases) & Not UCF  \\
    \hline
\end{tabular}
\end{center}
**: all length preserving cases, some non-length preserving.
\par }
\vspace{-10pt}
As an aside, we find examples of RAAGs $\RAAG = \langle X \rangle$ and non-trivial automorphisms $\phi \in \mathrm{Aut}(\RAAG)$, such that the language $\mathsf{ConjSL}_{\phi}(\RAAG,X)$ is not context-free. Recall that changing generating sets can change the behaviour of the language $\mathsf{ConjSL}(G,X)$ (see \cref{language:qi}), and so this example investigates whether changing the ordering of our generating set, or changing the automorphism of the RAAG, can change the behaviour of $\mathsf{ConjSL}_{\phi}(\RAAG,X)$.

\begin{prop}\label{exp: change}
Let $A_{\Gamma} = F_{2} \times F_{2}$, and label the defining graph as follows:
\[\begin{tikzcd}
	a & b \\
	d & c
	\arrow[no head, from=1-1, to=1-2]
	\arrow[no head, from=1-2, to=2-2]
	\arrow[no head, from=2-2, to=2-1]
	\arrow[no head, from=1-1, to=2-1]
\end{tikzcd}\]
Let $\psi \colon a \rightarrow b \rightarrow c \rightarrow d$ be a rotation, and let $X$ be the standard generating set. Consider the following scenarios:
\begin{enumerate}
    \item[(1)] Order X as $a < a^{-1} < b < b^{-1} < c < c^{-1} < d < d^{-1}$, and consider the map $\phi = \psi$ or $\phi = \psi^{2}$. 
    \item[(2)] Order X as $a < a^{-1} < c < c^{-1} < b < b^{-1} < d < d^{-1}$, and consider the map $\phi = \psi$.
\end{enumerate}
In both cases, the language $\mathsf{ConjSL}_{\phi}(A_{\Gamma}, X)$ is not context-free.
\end{prop}
We note that \cref{cor:acyhyp2} does not apply to the following example, since direct products of infinite groups are not acylindrically hyperbolic (Corollary 7.2, \cite{Osin2015}). 
\begin{proof}
(1): First consider $\phi = \psi$. Let $Y = \{a,c\}^{\pm}$. We claim that 
\[ \mathsf{ConjSL}_{\psi}(\RAAG, X) \cap Y = \mathsf{ConjSL}(F_{2}, Y).
\]
For any word $w \in X^{\ast}$, we can write $w = \left(u\left(a^{\pm 1}, c^{\pm 1}\right), v\left(b^{\pm 1}, d^{\pm 1}\right)\right)$, where $u\left(a^{\pm 1}, c^{\pm 1}\right)$ denotes a word over $Y$ and $v\left(b^{\pm 1}, d^{\pm 1}\right)$ denotes a word over $Z = \{b,d\}^{\pm}$. Note that
\[ \psi(w) = \psi\left(u\left(a^{\pm 1}, c^{\pm 1}\right), v\left(b^{\pm 1}, d^{\pm 1}\right)\right) = \left(v\left(c^{\pm},a^{\pm}\right), u\left(b^{\pm},d^{\pm}\right)\right).
\]
Let $(g,1), (h,1) \in \mathsf{ConjSL}_{\psi}(\RAAG, X) \cap Y$. Then 
\begin{align*}
    (g,1) \sim_{\psi} (h,1) 
    &\Leftrightarrow \psi(w)\cdot (g,1) \cdot w^{-1} =_{\RAAG} (h,1) \\
    &\Leftrightarrow v\left(c^{\pm 1},a^{\pm 1}\right)u\left(b^{\pm 1},d^{\pm 1}\right)\cdot (g,1)\cdot v^{-1}\left(b^{\pm 1},d^{\pm 1}\right)u^{-1}\left(a^{\pm 1}, c^{\pm 1}\right) =_{\RAAG} (h,1) \\
    &\Leftrightarrow u\left(b^{\pm 1},d^{\pm 1}\right)v^{-1}\left(b^{\pm 1},d^{\pm 1}\right) =_{\langle Z \rangle} 1 \\
    &\Leftrightarrow u = v \\
    &\Leftrightarrow u\left(c^{\pm 1},a^{\pm 1}\right)\cdot (g,1)\cdot u^{-1}\left(a^{\pm 1}, c^{\pm 1}\right) =_{\RAAG} (h,1)\\
    &\Leftrightarrow g \sim h \; \text{in} \; F_{2} = \langle Y \rangle,
\end{align*}
which proves the claim. Now suppose $\mathsf{ConjSL}_{\psi}(\RAAG, X)$ is context-free. Then by \cref{lem:intersection}, $\mathsf{ConjSL}(F_{2}, Y)$ is context-free, which is a contradiction by \cite[Proposition 2.2]{Ciobanu2013}. Therefore $\mathsf{ConjSL}_{\psi}(\RAAG, X)$ is not context-free.

\comm{
(1): We claim that $\mathsf{ConjSL}_{\psi}(\RAAG,X) = \{a^{n} \mid n \in \Z\}$, which is regular. \\
First note that any word in $\mathsf{ConjSL}_{\psi}(\RAAG,X)$ must start with an $a$ letter - otherwise we could $\psi$-cyclically permute the entire word in each of the cases:
\begin{align*}
    b^{\pm 1}w &\xleftrightarrow{\psi} \psi^{3}\left(b^{\pm 1}\right)\psi^{3}(w) = a^{\pm 1}\psi^{3}(w),  \\
    c^{\pm 1}w &\xleftrightarrow{\psi} \psi^{2}\left(c^{\pm 1}\right)\psi^{2}(w) = a^{\pm 1}\psi^{2}(w), \\
    d^{\pm 1}w &\xleftrightarrow{\psi} \psi\left(d^{\pm 1}\right)\psi(w) = a^{\pm 1}\psi(w).
\end{align*}
We now rule out some cases. No words over $\{a,b\}$ can belong to $\mathsf{ConjSL}_{\psi}(\RAAG,X)$, since these would be of the form $a^{n}b^{m} =_{\RAAG} b^{m}a^{n}$, and we could $\psi$-cyclic shift all the $b$ letters from $b^{m}a^{n}$ to $a^{n+m}$ as follows:
\[ b^{m}a^{n} \xleftrightarrow{\psi} a^{n}\psi^{-1}(b^{m}) = a^{n+m},
\]
which is a smaller word lexicographically. Similarly we can rule out all words over $\{a,d\}$. \par 
Consider words over $\{a,c\}$. We can first rule out words of the form $a^{-1}w_{1}\dots w_{n}c$ or $aw_{1}\dots w_{n}c^{-1}$, where each $w_{i} \in \{a^{\pm1},c^{\pm 1}\}$, since we can apply $\psi$-cyclic shifts to obtain a shorter word. For example:
\begin{align*}
    a^{-1}w_{1}\dots w_{n}c &\xleftrightarrow{\psi} \psi(c)a^{-1}w_{1}\dots w_{n}\\
    &= da^{-1}w_{1}\dots w_{n}\\
    &=_{\RAAG} a^{-1}w_{1}\dots w_{n}d \\
    &\xleftrightarrow{\psi} \psi(d)a^{-1}w_{1}\dots w_{n}\\
    &= aa^{-1}w_{1}\dots w_{n},
\end{align*}
which after cancellation gives a word of shorter length. We also can't have words of the form $aw_{1}\dots w_{n}c$, where each $w_{i} \in \{a^{\pm1},c^{\pm 1}\}$ belong to $\mathsf{ConjSL}_{\psi}(\RAAG,X)$, since we can apply $\psi$-cyclic shifts to get
\begin{align*}
    aw_{1}\dots w_{n}c \xleftrightarrow{\psi} \psi(c)aw_{1}\dots w_{n} &= daw_{1}\dots w_{n} \\
    &=_{\RAAG} aw_{1}\dots w_{n}d \\
    &\xleftrightarrow{\psi} \psi(d)aw_{1}\dots w_{n} = aaw_{1}\dots w_{n}.
\end{align*}
In particular, $w_{1} = a$ must hold for the word to be in $\mathsf{ConjSL}_{\psi}(\RAAG,X)$. By induction, each $w_{i} = a$, but then we can $\psi$-cyclic shift to
\[ a\dots ac \xleftrightarrow{\psi} a^{n+2}.
\]
A similar argument holds for words of the form $a^{-1}w_{1}\dots w_{n}c^{-1}$. We can now rule out all words over $\{a,c\}$. Indeed, any word would have to be of the form $a^{\pm 1}w_{1}\dots w_{n}a^{\pm 1}$ where $w_{i} \in \{a^{\pm1},c^{\pm 1}\}$, which can be $\psi$-cyclic shifted to
\[ a^{\pm 1}w_{1}\dots w_{n}a^{\pm 1} \xmapsto{\psi} \psi\left(a^{\pm 1}\right)a^{\pm 1}w_{1}\dots w_{n} =b^{\pm 1}a^{\pm 1}w_{1}\dots w_{n} =_{\RAAG} a^{\pm 1}b^{\pm 1}w_{1}\dots w_{n}.
\]
Then $w_{1} = a$ must hold, and again by induction, $w_{i} = a$ for all $i$. Hence we can rule out all words $w$ such that $|\text{supp}(w)| = 2$.
\par 
Now consider 3 letters. If $w \in \{a,b,c\}^{\ast}$, then it can be written in the form $w = b^{n}x$ where $x \in \{a,c\}^{\ast}$. Then we can $\psi$-cyclic shift $b^{n}x$ to $xa^{n}$:
\[ b^{n}x \xleftrightarrow{\psi} x\psi^{-1}(b^{n}) = xa^{n},
\]
which can then be $\psi$-cyclic shifted to $a^{k}$. Similar arguments show that no words $w$ such that $|\text{supp}(w)| = 3$ belong to $\mathsf{ConjSL}_{\psi}(\RAAG,X)$. Finally suppose $|\text{supp}(w)| = 4$. Since $\RAAG$ is a product, we can write $w$ in the form $w_{1}w_{2}$, where $w_{1} \in \{a,c\}^{\ast}$, $w_{2} \in \{b,d\}^{\ast}$. This can be $\psi$-cyclic shifted to $\psi(w_{2})w_{1}$, a word over $\{a,c\}$, which can be $\psi$-cyclic shifted to $a^{k}$. 
\par 
We've now ruled out all cases except when a word is over one letter. The smallest lexicographically of these is $a^{n}$, which proves the claim. 
\par 
\comm{
\myRed{Example 2: ignore for now (need to check all cases)}
\par }
(2): Suppose $\psi$-\mathsf{ConjSL} is context-free and consider the intersection $I = \psi-\mathsf{ConjSL} \cap L$ where $L = a^{+}c^{+}a^{+}$. Since $L$ is a regular language, $I$ must be context-free by \cref{lem:intersection}. We first prove that all words in $I$ are of the form
\[ a^{p}c^{q}a^{r}
\]
and $I = I_{1} \sqcup I_{2}$ where 
\begin{align*}
    I_{1} &= \{a^{p}c^{q}a^{r} \mid p>q, p>r \} \\
    I_{2} &= \{a^{p}c^{q}a^{p} \mid q \leq p \}
\end{align*}
Let $ w = a^{p}c^{q}a^{r} \in \phi-\mathsf{ConjSL}$ - recall that by definition, $w$ is shortest, with respect to the ordering of the generators, in its respective $\psi$-conjugacy class. Since $a^{p}c^{q}a^{r}$ is $\psi$-CR and $\psi$ is length-preserving, the $\psi$-conjugacy class consists of all possible $\psi$-shifts of $a^{p}c^{q}a^{r}$. The images of any letters moved must lie in $\{a,d\}$, and so the only possible shifts which could be shorter than $w$ are 
\[ \phi(d^{q})\phi(a^{r})a^{p} = a^{q}d^{r}a^{p}
\]
or
\[ a^{r}\phi^{-1}(a^{p})\phi^{-1}(d^{q}) = a^{r}d^{p}a^{q}
\]
Therefore
\[ w = a^{p}d^{q}a^{r} \leq a^{q}d^{r}a^{p} \Rightarrow 
\begin{cases}
p > q \\
p = q, \; q \leq r
\end{cases}

\]
and 
\[ w = a^{p}d^{q}a^{r} \leq a^{r}d^{p}a^{q} \Rightarrow 
\begin{cases}
p > r \\
p = r, q \leq p
\end{cases}
\]
Comparing possible cases for both options gives the disjoint sets $I_{1}$ and $I_{2}$. \\
Let $k$ be the constant given by the Pumping Lemma for context-free languages applied to the set $I$ (see \cite{Hopcroft}, Theorem 7.18, Page 281) and consider the word
\[ W = a^{n}d^{n}a^{n}
\]
where $n > k$. Note that $W$ is composed of 3 blocks, namely $a^{n}, d^{n}$ and $a^{n}$. By the Pumping Lemma, $W$ can be written as $W = uvwxy$ where $l(vx) \geq 1, l(vwx) \leq k$ and $uv^{i}wx^{i}y \in I$ for all $i \geq 0$. Since $l(vwx) \leq k < n$, $vwx$ cannot be part of more than two consecutive blocks.\\
\textbf{Case 1: $vwx$ is part of one block only.}\\
a) $vwx$ is in first block, i.e.
\[ u\cdot vwx \cdot y = a^{n-s}\cdot a^{s}\cdot d^{n}a^{n}
\]
(for some $s$). For $i = 0$, 
\[ uwy = a^{n-s}a^{t}d^{n}a^{n}
\]
where $p = n-s+t < n$. Since $q = n$, $uwy$ cannot lie in either $I_{1}$ or $I_{2}$, which contradicts our assumption that $uwy \in I$ by the Pumping Lemma. \\
b) $vwx$ is in second block, i.e.
\[  u\cdot vwx \cdot y = a^{n}d^{s_{1}} \cdot d^{s_{2}} \cdot d^{n-s_{1}-s_{2}}a^{n}
\]
For $i \geq 2$, we have that $q > p$, so again $uv^{i}wx^{i} \not \in I$. \\
c) $vwx$ is in third block, i.e.
\[ u\cdot vwx \cdot y = a^{n}d^{n} \cdot a^{s} \codt a^{n-s}
\]
For $i$ = 0, we get than $r<n$ and $p= q = n$, so again $uwy \not \in I$. \\
\textbf{Case 2: $vwx$ is part of more than one block}.\\
Suppose $vwx$ contains both $a$ and $d$ letters. If one of $v$ or $x$ contains both $a$ and $d$, then for $i \geq 2$ the word $uv^{i}wx^{i}y$ contains at least four blocks alternating between powers of $a$ and $d$, and so can't lie in $a^{+}d^{+}a^{+}$. Therefore $v$ must be a power of one letter only, and $x$ must be a power of the other letter. \\
a) $v$ in first block, $x$ in second, i.e.
\[ uvwxy = a^{i_{1}}\cdot a^{i_{2}} \cdot a^{i_{3}}d^{j_{1}} \cdot d^{j_{2}} \cdot d^{j_{3}}a^{n}
\]
where $i_{1} + i_{2} + i_{3} = j_{1} + j_{2} + j_{3} = n$. For $i = 0$,
\[ uwy = a^{i_{1}}a^{i_{3}}d^{j_{1}}d^{j_{3}}a^{n}
\]
Here $p < n, q < n$ and $r = n$, so $uwy \not \in I$. \\
b) $v$ in second block, $x$ in third. Here for $i \geq 2$, the word $uv^{i}wx^{i}y$ is of the form where $p = n, q>n$ and $r>n$, so $uv^{i}wx^{i}y \not \in I$.\\
This covers all cases so we have a contradiction. Hence $\phi-\mathsf{ConjSL}(\RAAG, X)$ is context-free. 
\par }
Now consider $\phi= \psi^{2}$. Suppose $\mathsf{ConjSL}_{\phi}(\RAAG,X)$ is context-free and consider the intersection $I = \mathsf{ConjSL}_{\phi}(\RAAG,X) \cap L$ where $L = a^{+}c^{+}a^{+}$. Since $L$ is a regular language, $I$ must be context-free by \cref{lem:intersection}. We first prove that $I = I_{1} \sqcup I_{2},$ where 
\[ I_{1} = \{a^{p}c^{q}a^{r} \mid p>q, p>r \}, \quad I_{2} = \{a^{p}c^{q}a^{p} \mid p \geq q \}.
\]
Let $w = a^{p}c^{q}a^{r} \in \mathsf{ConjSL}_{\phi}(\RAAG,X)$. Since $a^{p}c^{q}a^{r}$ is $\phi$-CR and $\phi$ is length-preserving, the set of words of minimal length in the $\phi$-conjugacy class of $a^{p}c^{q}a^{r}$ consists of all possible $\phi$-cyclic permutations of $a^{p}c^{q}a^{r}$. The images of any letters by $\phi$ must lie in $\{a,c\}$, which do not commute in $\RAAG$, and so the only possible $\phi$-cyclic permutations which could be lexicographically less than $w$ are of the form
\[ \phi(c^{q})\phi(a^{r})a^{p} = a^{q}c^{r}a^{p} \quad \text{or} \quad a^{r}\phi^{-1}(a^{p})\phi^{-1}(c^{q}) = a^{r}c^{p}a^{q}.
\]
Therefore
\[ w = a^{p}c^{q}a^{r} \leq a^{q}c^{r}a^{p} \Rightarrow 
\begin{cases}
p > q, \\
p = q, q \leq r,
\end{cases}
\]
and 
\[ w = a^{p}c^{q}a^{r} \leq a^{r}c^{p}a^{q} \Rightarrow 
\begin{cases}
p > r, \\
p = r, p \geq q.
\end{cases}
\]
Comparing possible cases for both options gives the disjoint sets $I_{1}$ and $I_{2}$. 
\par
Now let $k$ be the constant given by the Pumping Lemma for context-free languages applied to the set $I$ (see \cite{Hopcroft}, Theorem 7.18, Page 281) and consider the word $W = a^{n}c^{n}a^{n}$, where $n > k$. Note that $W$ is composed of 3 blocks, namely $a^{n}, c^{n}$ and $a^{n}$. By the Pumping Lemma, $W$ can be written as $W = uvwxy$, where $l(vx) \geq 1, l(vwx) \leq k$ and $uv^{i}wx^{i}y \in I$ for all $i \geq 0$. Since $l(vwx) \leq k < n$, $vwx$ cannot be part of more than two consecutive blocks.
\par 
\textbf{Case 1: $vwx$ is part of one block only.}\\
a) $vwx$ lies in first block, i.e.
\[ u\cdot vwx \cdot y = a^{n-s-t}\cdot a^{s}\cdot a^{t}c^{n}a^{n}.
\]
For $i = 0$, 
\[ uwy = a^{n-s-t}\cdot a^{r}\cdot a^{t}c^{n}a^{n},
\]
where $r < s$ (cannot have $r=s$ since $l(vx) \geq 1$). If $uwy \in I$, then $p = n+r-s$ and $q=n$. Hence $p < q$, which is a contradiction. 

b) $vwx$ is in second block, i.e.
\[  u\cdot vwx \cdot y = a^{n}c^{s_{1}} \cdot c^{s_{2}} \cdot c^{n-s_{1}-s_{2}}a^{n}.
\]
For $i \geq 2$, we have that $p < q$, so again $uv^{i}wx^{i}y \not \in I$. \\
c) $vwx$ is in third block, i.e.
\[ u \cdot vwx \cdot y = a^{n}c^{n} \cdot a^{s} \cdot a^{n-s}.
\]
For $i$ = 0, we have that $r<n$ and $p= q = n$, so again $uwy \not \in I$. \par 
\textbf{Case 2: $vwx$ is part of more than one block}.\\
Suppose $vwx$ contains both $a$ and $c$ letters. If one of $v$ or $x$ contains both $a$ and $c$, then for $i \geq 2$ the word $uv^{i}wx^{i}y$ contains at least four blocks alternating between powers of $a$ and $c$, and so can't lie in $a^{+}c^{+}a^{+}$. Therefore $v$ must be a power of one letter only, and $x$ must be a power of the other letter. 

a) $v$ in first block, $x$ in second, i.e.
\[ uvwxy = a^{i_{1}}\cdot a^{i_{2}} \cdot a^{i_{3}}c^{j_{1}} \cdot c^{j_{2}} \cdot c^{j_{3}}a^{n},
\]
where $i_{1} + i_{2} + i_{3} = j_{1} + j_{2} + j_{3} = n$. For $i = 0$,
\[ uwy = a^{i_{1}}a^{i_{3}}c^{j_{1}}c^{j_{3}}a^{n}.
\]
Here $p < n, q < n$ and $r = n$, so $uwy \not \in I$. 

b) $v$ in second block, $x$ in third. Here for $i \geq 2$, the word $uv^{i}wx^{i}y$ is of the form where $p = n, q>n$ and $r>n$, so $uv^{i}wx^{i}y \not \in I$. This covers all cases and so $I$ is not context-free. Hence $\mathsf{ConjSL}_{\phi}(\RAAG, X)$ is not context-free. 
\par 
The proof for Case (2) is identical to the second part of Case (1), with some additional justification needed for the set of words $I$.
\end{proof}
\vspace{-10pt}
It remains unclear whether all twisted conjugacy languages $\mathsf{ConjSL}_{\phi}(\RAAG, X)$ are not context-free for RAAGs. Whilst we know that for any finite extension $A_{\phi}$, the language $\mathsf{ConjSL}\left(A_{\phi}, \widehat{X}\right)$ is not context-free by \cref{cor: v RAAGs not CF}, there could be examples of RAAGs and automorphisms where the language $\mathsf{ConjSL}_{\phi}(\RAAG, X)$ is regular.

\begin{question}
    Does there exist a RAAG $\RAAG = \langle X \rangle$ and non-trivial automorphism $\phi \in \mathrm{Aut}(\RAAG)$ such that $\mathsf{ConjSL}_{\phi}(\RAAG, X)$ is regular?
\end{question}

\section{Open questions}
We finish with some unanswered problems arising from this paper.
\par 
For conjugacy geodesics, when considering virtual RAAGs via length-preserving automorphisms, it remains unclear as to whether the language of conjugacy geodesics is regular or not.
\vspace{-10pt}
\begin{question}\label{question 1}
Let $A_{\phi}$ be a virtual RAAG (as defined in \eqref{eqn: extension}) such that $\phi \in \mathrm{Aut}(\RAAG)$ is a graph automorphism. Is $\mathsf{ConjGeo}\left(A_{\phi}, \widehat{X}\right)$ regular?
\end{question}
\vspace{-10pt}
We know in general that for graph automorphisms, $\mathsf{ConjGeo}_{\psi} \subsetneq \mathsf{CycGeo}_{\psi}$ (see \cref{exmp:clash languages}), and while $\mathsf{CycGeo}_{\psi}$ is regular, a subset of a regular language may not be regular. We can however give a partial answer to \cref{question 1}, when considering direct products of free groups. 
\par 
For $\psi \in \mathrm{Aut}(\RAAG)$, where $\RAAG = A_{\Gamma_{1}} \times A_{\Gamma_{2}}$, we say $\psi$ \emph{fixes elements within vertex groups} if for all $v \in V(\Gamma_{i})$, $\psi(v) \in A_{\Gamma_{i}}$. In this scenario, we can immediately deduce that
\[ \mathsf{ConjGeo}_{\psi}(\RAAG, X) = \mathsf{ConjGeo}_{\psi}(A_{\Gamma_{1}}, X_{1}) \cdot \mathsf{ConjGeo}_{\psi}(A_{\Gamma_{2}}, X_{2}),
\]
where $X_{i} = V(\Gamma_{i})^{\pm}$ for $i \in \{1,2\}$. By \cref{lem: closure}, $\mathsf{ConjGeo}_{\psi}(\RAAG, X)$ is regular if both $\mathsf{ConjGeo}_{\psi}(A_{\Gamma_{1}}, X_{1})$ and $\mathsf{ConjGeo}_{\psi}(A_{\Gamma_{2}}, X_{2})$ are regular. Using this result, we can find an example of a graph automorphism $\phi$ such that $\mathsf{ConjGeo}\left(A_{\phi}, \widehat{X}\right)$ is regular.
\begin{exmp}\label{exmp:regular conj}
Let $\RAAG = F_{2} \times F_{2}$ and label the defining graph as follows:
\[\begin{tikzcd}
	a & b \\
	d & c
	\arrow[no head, from=1-1, to=1-2]
	\arrow[no head, from=1-2, to=2-2]
	\arrow[no head, from=2-2, to=2-1]
	\arrow[no head, from=1-1, to=2-1]
\end{tikzcd}\]
Let $\phi \colon a \leftrightarrow c, \; b \leftrightarrow d$ be a reflection, and let $X$ be the standard generating set. To determine whether $\mathsf{ConjGeo}\left(A_{\phi}, \widehat{X}\right)$ is regular,  we apply a similar method as \cref{thm:conjgeo regular extension group}, and check regularity of two twisted conjugacy classes, namely $\mathsf{ConjGeo}(\RAAG, X)$ and $\mathsf{ConjGeo}_{\phi}(\RAAG, X)$. The first language is regular by \cite[Corollary 2.5]{Ciobanu2016}. Since $\phi$ fixes elements within the vertex groups, showing $\mathsf{ConjGeo}_{\phi}(\RAAG, X)$ is regular reduces to showing that $\mathsf{ConjGeo}_{\phi}\left(F_{2}, \{a,c\}^{\pm}\right)$ is regular. Note if $u,v \in \mathsf{ConjGeo}_{\phi}\left(F_{2}, \{a,c\}^{\pm}\right)$, then $u \sim_{\phi} v$ if and only if $u$ and $v$ are related via a sequence of $\phi$-cyclic permutations only (see \cref{cor:sequence}). It is then straightforward to show that the language $\mathsf{ConjGeo}_{\phi}\left(F_{2}, \{a,c\}^{\pm}\right)$ consists of all words $w \in \mathsf{Geo}\left(F_{2}, \{a,c\}^{\pm}\right)$, such that $w \not \in \{avc^{-1}, a^{-1}vc, cva^{-1}, c^{-1}va\}$ for all $v \in \mathsf{Geo}\left(F_{2}, \{a,c\}^{\pm}\right)$. This can be written as a regular expression, and so the language $\mathsf{ConjGeo}_{\phi}\left(F_{2}, \{a,c\}^{\pm}\right)$ is regular. Since both twisted classes are regular, $\mathsf{ConjGeo}_{\phi}\left(A_{\phi}, \widehat{X}\right)$ must also be regular.
\end{exmp}
\vspace{-10pt}
We note that if we consider the same RAAG with the rotation $\phi \colon a \rightarrow b \rightarrow c \rightarrow d$, then this question becomes more difficult. Here we would need to check four twisted classes. Two of these, namely $\mathsf{ConjGeo}$ and $\mathsf{ConjGeo}_{\phi^{2}}$, are regular by \cref{exmp:regular conj}, but it is still unclear whether $\mathsf{ConjGeo}_{\phi}$ or $\mathsf{ConjGeo}_{\phi^{3}}$ are regular or not. 
\par 
For the spherical conjugacy language, it would be interesting to see if we can strengthen our results in the case of extensions of RACGs.

\begin{question}
Does there exist a virtual RACG $W_{\phi} = W_{\Gamma} \rtimes_{\phi} \langle t \rangle$, such that $\mathsf{ConjSL}\left(W_{\phi}, \widehat{X}\right)$ is not context-free?
\end{question}
\vspace{-10pt}
We cannot take the same approach as \cref{thm: RAAG not CF}, since a similar intersection would give
\[ \mathsf{ConjSL}(W_{\Gamma}, X) \cap \{x, y\}^{\ast} = \mathsf{ConjSL}(D_{\infty}, \{x,y\}).
\]
This is inconclusive since $\mathsf{ConjSL}(D_{\infty}, \{x,y\})$ is regular \cite[Theorem 2.4]{Ciobanu2013}, so a new approach will be needed.
\par 
For more general graph products, it may be possible to extend \cref{cor:acyhyp2}, if the following two questions can be answered.

\begin{question}
When are graph products $\mathcal{AH}$-accessible?
\end{question}
\vspace{-10pt}
Some example of groups which are $\mathcal{AH}$-accessible can be found in \cite[Theorem 2.18]{Abbott2019}, which includes RAAGs.
\begin{question}\label{q:HHGs}
Are all graph products hierarchically hyperbolic groups?
\end{question}
\vspace{-10pt}
We mention a recent result, Theorem C of \cite{Berlyne2022}, which states that if each of the vertex groups of a graph product are hierarchically hyperbolic, then the graph product itself is a hierarchically hyperbolic group. In general, \cref{q:HHGs} is false, since there exists graph products with unsolvable word problem.

\section*{Acknowledgments}
I would like to thank my supervisor Laura Ciobanu for her mathematical guidance and support while writing my first paper. I would also like to thank Andrew Duncan, Michal Ferov, Alex Levine and Alessandro Sisto for helpful comments, as well as the anonymous referee for detailed comments and suggestions.

\appendix

\comm{
\section{$\mathsf{ConjGeo}$ for RAAGs}
Length preserving cases: in general 
\[ \psi-\mathsf{ConjGeo} \subset \psi-\mathsf{CycGeo}
\]
Still unclear in general whether $\psi-\mathsf{ConjGeo}$ is regular. Some subcases:
\subsection{Free groups}
Want all twisted classes to be regular since hyperbolic group. 
\begin{prop}
If $\RAAG \cong F_{n}$, then $\psi-\mathsf{ConjGeo} = \psi-\mathsf{CycGeo}$.
\end{prop}

\begin{proof}
Let $w \in \psi-\mathsf{CycGeo}$ and suppose $w$ is not $\psi$-cyclically reduced. Then there exists a sequence of $\psi$-shifts
\[ w = w_{0} \leftrightarrow w_{1} \leftrightarrow \dots \leftrightarrow w_{s}
\]
such that $w_{s}$ is not geodesic. By definition, $w_{s}$ is a $\psi$-cyclic permutation of $w$, and so $w_{s}$ must be geodesic since $w \in \psi-\mathsf{CycGeo}$. This gives a contradiction.

\end{proof}
Key point: problem with proving equality comes from shuffles. 
\myRed{Incorrect result}
\begin{prop}
Let $\psi$ be length preserving of order 2. Then $v \in \RAAG$ is $\psi$-CR if and only if $v$ cannot be written in the form
\[ v =_{\RAAG} \psi(u)^{-1}wu
\]
where $|w| < |v|$.
\end{prop}

\begin{proof}
The forward direction was proven in \cref{prop:cycle red}. For the reverse direction, suppose $v$ is not $\psi$-CR. By definition there exists an alternating sequence
\[ v = v_{0} \leftrightarrow v_{1} \leftrightarrow \dots \leftrightarrow v_{s}
\] 
of $\psi$-shifts and shuffles such that $|\mathrm{Geod}(v_{s})| < |v|$. We first prove the following:\\
\textbf{Claim}: let $x \leftrightarrow y$ be a $\psi$-shift. Then $x = \psi(z)^{-1}yz$ for some $z \in \RAAG$.\\
\textbf{Proof of Claim:} Let $x = x_{1}\dots x_{n}$. By definition,
\[ y = \psi(x_{i+1})\dots \psi(x_{n})x_{1}\dots x_{i}
\]
for some $i$. Hence we can write
\[ x = \psi(x_{i+1}\dots x_{n})^{-1}\cdot y \cdot x_{i+1}\dots x_{n}
\]
as required.\\
Suppose the sequence above is of the form
\[ v = v_{0} \xleftrightarrow{\psi} v_{1} \xleftrightarrow{E} v_{2} \xleftrightarrow{\psi} \dots \xleftrightarrow{\psi} v_{s}
\] 
By the Claim, $v_{0} = \psi(x_{0})^{-1}v_{1}x_{0} =_{\RAAG} \psi(x_{0})^{-1}v_{2}x_{0}$ for some $x_{0}$. Repeating the Claim we can say
\[ v_{2} = \psi(x_{2})^{-1}v_{3}x_{2} =_{\RAAG} \psi(x_{2})^{-1}v_{4}x_{2}
\]
and so 
\[ v_{0} =_{\RAAG} \psi(x_{0})^{-1}v_{2}x_{0} =_{\RAAG} \psi(x_{2}x_{0})^{-1}v_{4}x_{2}x_{0}
\]
Continuing in this way we get 
\[ v_{0} =_{\RAAG} \psi(w)^{-1}v_{s}w
\]
Since $|v_{s}| < |v_{0}|$, we reach a contradiction as required. The proof is analogous if the sequence starts with shuffling. 

\end{proof}

\begin{rmk}
This proof won't work for orders above 2. This is because the powers of the automorphism could differ when we take $\psi$-shifts. \\
Equivalence is up to cancellation here as well as shuffling.
\end{rmk}
Example: line on 4 points, auto = reflection. Then equality of sets doesn't hold (see Python).}

\begin{appendices}
\section{Conjugacy geodesics in virtual graph products}
In this section we prove \cref{cor:graphprod geo regular} - a twisted adaptation of Theorem 2.4 in \cite{Ciobanu2016}, with some restrictions on the vertex groups. Recall $G_{\Gamma} = \langle X \rangle$ denotes a graph product with standard generating set $X = \bigcup^{k}_{i=1} X_{i}$, where each vertex group $G_{i}$ is generated by $X_{i}$.

Before proving \cref{cor:graphprod geo regular}, we first establish some notation and related results. 

\begin{defn}
For $\psi \in \mathrm{Aut}(G_{\Gamma})$, we say $\psi$ \emph{fixes the vertex groups} if for all $x_{i} \in X_{i}$, $\psi(x_{i}) \in G_{i}$ for every vertex group $G_{i} = \langle X_{i} \rangle$. 
\end{defn}

\begin{defn}
    For all $1 \leq i \leq k$, define $\mathsf{Geo}_{i} := \mathsf{Geo}(G_{i}, X_{i})$ to be the set of geodesic words over each vertex group. Define the map $\rho_{i} \colon X^{*} \rightarrow (X_{i} \cup \$)^{*}$ where, for each $a \in X_{j}$, we set $\rho_{i}(a) := a$ if $i=j$, $\rho_{i}(a):= \epsilon$ if the vertices $i$ and $j$ are adjacent in $\Gamma$, and $\rho_{i}(a) := \$ $ otherwise (here $\epsilon$ denotes the empty word).

    For each $i$, we define the extension of $\psi$:
\begin{align*}
    \widehat{\psi} \colon (X_{i} \cup \$)^{*} &\rightarrow (X_{i} \cup \$)^{*} \\
    a &\mapsto \psi(a), \quad a \in X_{i}, \\
    \$ &\mapsto \$.
\end{align*}
In particular, $\rho_{i}(\psi(x)) = \widehat{\psi}(\rho_{i}(x))$ for all $x \in X$.
\end{defn}

\begin{lemma}\label{lemma: twisted cycles hat}
    Let $w \in X^{\ast}$. If $w' \in X^{\ast}$ is a $\psi$-cyclic permutation of $w$, then $\rho_{i}\left(w'\right)$ is a $\widehat{\psi}$-cyclic permutation of $\rho_{i}(w)$ (for all $1 \leq i \leq k$).
\end{lemma}

\begin{proof}
    Let $w = w_{1}\dots w_{n} \in X$, and suppose 
    \[ w' = \psi^{k}(w_{j+1})\dots \psi^{k}(w_{n})\psi^{k-1}(w_{1})\dots \psi^{k-1}(w_{j}),
\]
for some $1 \leq j \leq n$, $0 \leq k \leq m-1$. Then for all $i$,
\begin{align*}
    \rho_{i}\left(w'\right) &= \rho_{i}\left(\psi^{k}\left(w_{j+1}\right)\right)\dots \rho_{i}\left(\psi^{k}\left(w_{n}\right)\right)\rho_{i}\left(\psi^{k-1}\left(w_{1}\right)\right)\dots \rho_{i}\left(\psi^{k-1}\left(w_{j}\right)\right) \\
    &= \widehat{\psi^{k}}\left(\rho_{i}\left(w_{j+1}\right)\right)\dots \widehat{\psi^{k}}\left(\rho_{i}\left(w_{n}\right)\right)\widehat{\psi^{k-1}}\left(\rho_{i}\left(w_{1}\right)\right)\dots \widehat{\psi^{k-1}}\left(\rho_{i}\left(x_{j}\right)\right),
\end{align*}
which is precisely a $\widehat{\psi}$-cyclic permutation of $\rho_{i}(w)$. 
\end{proof}

The following result was shown in \cite{Ciobanu2013}.

\begin{lemma}\label{lemma: geo to funny geo}\cite[Proposition 3.3]{Ciobanu2013}
    A word $w \in X^{\ast}$ lies in $\mathsf{Geo}(G_{\Gamma},X)$ if and only if $\rho_{i}(w) \in \mathsf{Geo}_{i}(\$\mathsf{Geo}_{i})^{*}$ for all $1 \leq i \leq k$. 
\end{lemma}
 
\begin{prop}\label{lemma:conjgeo options}
Let $w \in X^{\ast}$. Let $\psi \in \mathrm{Aut}(G_{\Gamma})$ be length-preserving and fix the vertex groups of $G_{\Gamma}$. Then
\[ w \in \mathsf{ConjGeo}_{\psi}(G_{\Gamma},X) \Leftrightarrow \; \text{for all} \; 1 \leq i \leq k , 
\]
\[
\rho_{i}(w) \in \mathsf{ConjGeo}_{\psi}(G_{i}, X_{i}) \cup \{ u_{0}\$ u_{1} \dots \$ u_{n} \mid n \geq 1 \; \text{and} \; \psi(u_{n})u_{0}, u_{1} \dots, u_{n-1} \in \mathsf{Geo}_{i} \}.
\]
\end{prop}

\begin{proof}
($\Rightarrow$): Suppose $w \in \mathsf{ConjGeo}_{\psi}(G_{\Gamma},X)$. Then $w \in \mathsf{Geo}(G_{\Gamma},X)$ and so $\rho_{i}(w) \in \mathsf{Geo}_{i}(\$\mathsf{Geo}_{i})^{*}$ by \cref{lemma: geo to funny geo}. In particular, $\rho_{i}(w) = u_{0}\$ u_{1} \dots \$ u_{n}$, where each $u_{j} \in \mathsf{Geo}_{i}$ ($0 \leq j \leq n)$.
\par 
\textbf{Case 1: $n=0$}. Here $\rho_{i}(w) = u_{0} \in \mathsf{Geo}_{i}$. In particular, we can write $w =_{G_{\Gamma}} u_{0}v$, where all letters from $v$ lie in vertex groups $X_{j}$, such that the vertices $i,j \in V(\Gamma)$ are connected by an edge in $\Gamma$. Suppose $u_{0} \not \in \mathsf{ConjGeo}_{\psi}(G_{i}, X_{i})$. Then there exists $g \in \mathsf{Geo}_{i}$ such that
\[ \psi(g)^{-1}u_{0}g =_{G_{\Gamma}} z \in G_{i},
\]
where $l(z) < l(u_{0})$. This implies that
\[ \psi(g)^{-1}wg =_{G_{\Gamma}} \psi(g)^{-1}u_{0}vg =_{G_{\Gamma}} zv.
\]
Here $l(zv)<l(w)$, and so $w \not \in \mathsf{ConjGeo}_{\psi}(G_{\Gamma},X)$, which is a contradiction. Hence $\rho_{i}(w) \in \mathsf{ConjGeo}_{\psi}(G_{i}, X_{i})$.
\par 
\textbf{Case 2: $n>0$}. By \cref{lemma: twisted cycles hat}, there exists $w' \in X^{\ast}$ such that
\[ \rho_{i}\left(w'\right) = \psi(u_{n})u_{0}\$ u_{1} \dots \$ u_{n-1}\$.
\]
Since $w'$ is geodesic, $\rho_{i}\left(w'\right) \in \mathsf{Geo}_{i}(\$\mathsf{Geo}_{i})^{*}$ by \cref{lemma: geo to funny geo}, and so $\psi(u_{n})u_{0} \in \mathsf{Geo}_{i}$ as required.
\par 
($\Leftarrow$): From the assumption it follows that $\rho_{i}(w) \in \mathsf{Geo}_{i}(\$\mathsf{Geo}_{i})^{*}$, and so $w \in \mathsf{Geo}(G_{\Gamma}, X)$ by \cref{lemma: geo to funny geo}. Therefore we can write $\rho_{i}(w) = w_{i}u_{i}w'_{i}$, where $w_{i}, w'_{i} \in \mathsf{Geo}_{i}$, and either $u_{i} = w'_{i} = \epsilon$ or $u_{i} \in \$ (\mathsf{Geo}_{i}\$)^{*}$. Suppose $w \not \in \mathsf{ConjGeo}_{\psi}(G_{\Gamma}, X)$. Then there exists $y \in \mathsf{Geo}(G_{\Gamma}, X)$ such that $\psi(y)wy^{-1} =_{G_{\Gamma}} x$ and $l(x) < l(w)$. Here we choose $x$ and $y$ to be of minimal length, and note $l(y) = l(\psi(y))$. Since $y \in \mathsf{Geo}(G_{\Gamma}, X)$, we can write $\rho_{i}(y) = v_{i}y_{i}$, where $v_{i} \in (\mathsf{Geo}_{i}\$)^{*}$ and $y_{i} \in \mathsf{Geo}_{i}$ by \cref{lemma: geo to funny geo}. We consider the two scenarios for $\rho_{i}(w)$.

\textbf{Case 1:} $\rho_{i}(w) = w_{i} \in \mathsf{Geo}_{i}$. This implies $\rho_{i}(\psi(y)^{-1}xy) \in X_{i}$, so we can assume $v_{i} = \epsilon$, i.e
\[ \rho_{i}(y) = y_{i}, \quad \rho_{i}(\psi(y)) = \psi(y_{i}) \quad \Rightarrow \rho_{i}(w) =_{G_{\Gamma}} \psi(y_{i})^{-1}\rho_{i}(x)y_{i}.
\]
Since $\rho_{i}(w) \in \mathsf{ConjGeo}_{\psi}(G_{i}, X_{i})$, then $l(\rho_{i}(w))\leq l(\rho_{i}(x))$. 

\comm{
Since $\psi$ fixes elements within vertex groups, we can assume that $\rho_{i}(\psi(y)) = \widehat{\psi}(v_{i})\psi(y_{i})$. Let $G^{\ast}_{i} = G_{i} \ast \langle \$ \rangle$. Then 
\[ \rho_{i}(x) =_{G^{\ast}_{i}} \rho_{i}\left(\psi(y)wy^{-1}\right) =_{G^{\ast}_{i}} \widehat{\psi}(v_{i})\psi(y_{i})\cdot w_{i}u_{i}w'_{i} \cdot y^{-1}_{i}v^{-1}_{i}.
\]}
Let $l_{i}(u)$ be the number of letters in a word $u \in (X_{i} \cup \$)^{*}$ which lie in $X_{i}$. Then for any $z \in X^{\ast}$,
\[ l(z) = \sum_{i=1}^{k} l_{i}(\rho_{i}(z)).
\]
Note if $x \in X_{i}$, then $l_{i}(\rho_{i}(x)) = l(\rho_{i}(x))$. We now have
\[
    l(x) = \sum^{k}_{i=1} l_{i}(\rho_{i}(x))
    = \sum^{k}_{i=1} l(\rho_{i}(x)) 
    \geq \sum^{k}_{i=1} l(\rho_{i}(w)) 
    = \sum^{k}_{i=1} l_{i}(\rho_{i}(w)) = l(w),
\]
which contradicts our assumption that $l(x) < l(w)$.\\
\comm{
In particular
\begin{equation}\label{eqn:2}
    l\left(\psi(y)wy^{-1}\right) = \sum_{i=1}^{k} l_{i}\left(\widehat{\psi}(v_{i})\psi(y_{i})\cdot w_{i}u_{i}w'_{i} \cdot y^{-1}_{i}v^{-1}_{i}\right) \geq \sum_{i=1}^{k}l(\psi(y_{i})w_{i}) + l_{i}(u_{i}) + l\left(w'_{i}y^{-1}_{i}\right).
\end{equation}
We need the following claim:\\
\textbf{Claim:}
\[ l\left(w_{i}\right) + l\left(w'_{i}\right) \leq l\left(\psi(y_{i})w_{i}\right) + l\left(w'_{i}y^{-1}_{i}\right).
\]
\textbf{Proof of Claim:} \\}
\comm{
Then
\begin{align*}
    l(\psi(y_{i})w_{i}) + l\left(w'_{i}y^{-1}_{i}\right) &= l(\rho_{i}(x)y_{i}) + l\left(y^{-1}_{i}\right) \\
    &= l(\rho_{i}(x)) + l(y_{i}) + l\left(y^{-1}_{i}\right) \\
    &= l(\rho_{i}(x)) + l(\psi(y_{i})) + l\left(y^{-1}_{i}\right) \\
    &= l(w_{i}).
\end{align*}}
\textbf{Case 2}: $\rho_{i}(w) = w_{i}u_{i}w'_{i}$. Let $w = w_{i}\tilde{w}w'_{i}$, and suppose $y = \tilde{y}y_{i}$, where $\psi(y_{i}) = w^{-1}_{i}$. By assumption $\psi\left(w'_{i}\right)w_{i} \in \mathsf{Geo}_{i}$, and so 
\begin{align*}
    \psi\left(w'_{i}\right)w_{i} \in \mathsf{Geo}_{i} &\Rightarrow \psi\left(w'_{i}\right)\psi\left(y^{-1}_{i}\right) \in \mathsf{Geo}_{i} \\
    &\Rightarrow \psi\left(w'_{i}y^{-1}_{i}\right) \in \mathsf{Geo}_{i} \\
    &\Rightarrow w'_{i}y^{-1}_{i} \in \mathsf{Geo}_{i}.
\end{align*}
\comm{
Then
\begin{align*}
    l(\psi(y_{i})w_{i}) + l\left(w'_{i}y^{-1}_{i}\right) &= l(\psi(y_{i})w_{i}) + l\left(\psi\left(w'_{i}\right)\psi\left(y^{-1}_{i}\right)\right) \\
    &\geq \left(\psi\left(w'_{i}\right)\psi\left(y^{-1}_{i}\right)\psi\left(y_{i}\right)w_{i}\right) \\
    &= l\left(\psi\left(w'_{i}\right)w_{i}\right) \\
    &= l\left(\psi\left(w'_{i}\right)\right) + l(w_{i}) \\
    &= l\left(w'_{i}\right) + l(w_{i}).
\end{align*}}
Then
\begin{equation}\label{eqn:contradiction graph product}
    \psi(y)wy^{-1} = \psi(\tilde{y})\psi(y_{i})w_{i}\tilde{w}w'_{i}y^{-1}_{i} \tilde{y}^{-1} =_{G_{\Gamma}} \psi(\tilde{y})\cdot \tilde{w}w'_{i}y^{-1}_{i}\cdot \tilde{y}^{-1}.
\end{equation}
Since $\tilde{w}w'_{i}$ and $w'_{i}y^{-1}_{i}$ are both geodesic in $X_{i}$, $\tilde{w}w'_{i}y^{-1}_{i} \in \mathsf{Geo}_{i}$. Moreover, $l(\tilde{w}w'_{i}y^{-1}_{i}) = l(w)$ and $l(\tilde{y}) < l(y)$, and so \cref{eqn:contradiction graph product} contradicts our assumption of minimality in our choice of $y$. 
\comm{

, and so the left hand side of the inequality cannot equal zero. Therefore by \cref{eqn:2} and the Claim above, we have that
\begin{align*}
    l\left(\psi(y)wy^{-1}\right) &\geq \sum_{i=1}^{k}\left(l\left(w_{i}\right) + l_{i}\left(u_{i}\right) + l\left(w'_{i}\right)\right)  \\
    &= \sum_{i=1}^{k}l_{i}(\rho_{i}(w)) \\
    &= l(w).
\end{align*}
But this contradicts our assumption that $l(x) < l(w) < l\left(\psi(y)wy^{-1}\right)$, and so \\$w \in \mathsf{ConjGeo}_{\psi}(G_{\Gamma}, X)$, which completes the proof.}
\end{proof}

\begin{thm}\label{thm:same twisted conj geo cyc geo}
Let $\psi \in \mathrm{Aut}(G_{\Gamma})$ be length-preserving and fix the vertex groups of $G_{\Gamma}$. Suppose $\mathsf{CycGeo}_{\psi}(G_{i}, X_{i}) = \mathsf{ConjGeo}_{\psi}(G_{i}, X_{i})$ for each $i$. Then 
\[ \mathsf{CycGeo}_{\psi}(G_{\Gamma}, X) = \mathsf{ConjGeo}_{\psi}(G_{\Gamma}, X).
\]
\end{thm}

\begin{proof}
First note $\mathsf{ConjGeo}_{\psi}(G_{\Gamma},X) \subseteq \mathsf{CycGeo}_{\psi}(G_{\Gamma}, X)$ by \cref{lemma:subset twisted conj languages}. Now suppose $w \in \mathsf{CycGeo}_{\psi}(G_{\Gamma}, X)$. Since $w$ is geodesic, then by \cref{lemma: geo to funny geo} we have $\rho_{i}(w) = u_{0}\$u_{1}\$\dots \$ u_{n}$ for each $i$, where each $u_{j} \in \mathsf{Geo}_{i}$ ($0 \leq j \leq n$). By \cref{lemma: twisted cycles hat}, there exists a $\psi$-cyclic permutation $w' \in X^{\ast}$ of $w$ such that $\rho_{i}\left(w'\right) = \psi(u_{n})u_{0}\$u_{1}\dots u_{n-1}\$ $. Since $w' \in \mathsf{Geo}(G_{\Gamma},X)$ by assumption, we must have that $\psi(u_{n})u_{0}, u_{1}, \dots, u_{n-1} \in \mathsf{Geo}_{i}$ by \cref{lemma: geo to funny geo}. 
\par 
Now suppose $n=0$, i.e. $\rho_{i}(w) = u_{0} \in \mathsf{Geo}_{i}$. For every $\psi$-cyclic permutation $u'$ of $u_{0}$, there exists a $\psi$-cyclic permutation $w'$ of $w$ such that $\rho_{i}\left(w'\right) = u' \in \mathsf{Geo}_{i}$ (by \cref{lemma: twisted cycles hat}). Therefore $u_{0} \in \mathsf{CycGeo}(G_{i}, X_{i})$. By assumption this means $u_{0} \in \mathsf{ConjGeo}(G_{i}, X_{i})$. Hence for all $n \geq 0$, we can apply \cref{lemma:conjgeo options} to conclude $w \in \mathsf{ConjGeo}_{\psi}(G_{\Gamma},X)$ as required. 
\end{proof}

\begin{proof}[Proof of \cref{cor:graphprod geo regular}]
By assumption, $\mathsf{Geo}(G_{\Gamma}, X)$ is regular \cite{LOEFFLER2002}, and so the result follows by \cref{thm:same twisted conj geo cyc geo} and \cref{thm:conjgeo regular extension group}. 
\comm{
Since $\mathsf{Geo}(G_{i}, X_{i})$ is regular, $\mathsf{CycGeo}_{\phi^{l}}(G_{i}, X_{i})$ is regular for each vertex group, by a similar argument as Proposition \ref{prop:cycgeo}. By assumption, this implies $\mathsf{ConjGeo}_{\phi^{l}}(G_{i}, X_{i})$ is regular for each vertex group, and therefore $\mathsf{ConjGeo}_{\phi^{l}}(G_{\Gamma}, X)$ is regular. The result then follows by a similar proof as \Cref{thm:ConjGeo Regular}. }
\end{proof}

\end{appendices}
\comm{
\section{Notation (draft only)}
$G_{\Gamma}$ := graph product \\
$G_{k}$ := vertex groups of graph product \\
$G_{\phi}$ := finite index group extension \\
$f_{L}(z)$ := strict growth series \\
$l(w)$ := word length \\
$|g|$ := length of a group element \\
$|g|_{c}$ := length up to conjugacy \\
$\text{Rd}(V)$ := reduced word from $V$ \\
$\|V\|$ := length of reduction of $V$ \\
$|g|_{\psi}$ := length up to twisted conjugacy
}

\comm{
\subsection{Further examples}
\begin{itemize}
    \item Paragraph: F2 case, summary of proof
\end{itemize}
\begin{defn}
A graph is \emph{asymmetric} or \emph{rigid} if there are no non-trivial graph automorphisms. 
\end{defn}

\begin{prop}\label{prop:inversion not CF}
Let $\RAAG$ be defined on a graph $\Gamma$ where
\begin{enumerate}
    \item $|\Gamma| \geq 2$, and 
    \item There exists a non-edge between at least 2 vertices in $\Gamma$, i.e. there exists two generators $a, b \in \RAAG$ such that $[a,b] \neq 1$.
\end{enumerate}
Let $\phi \in \mathrm{Aut}(\RAAG)$ be a composition of inversions. Then $\phi-\mathsf{ConjSL}(\RAAG)$ is not context-free.
\end{prop}
Note that if $\phi$ is trivial and $|\Gamma| = 2$, then $A_{\Gamma} \cong F_{2}$ and the result follows by \textbf{CITE F2 CASE}. 

\begin{proof}
Assume $a,b$ are two generators in $\RAAG$ such that $[a,b] \neq 1$. The proof of this Proposition is identical to the proof for $F_{2}$ once we verify that the parameters on the set $I$ are the same. For notation, define a map
\begin{align*}
    \widehat{\phi}: \{1,-1\} &\rightarrow \{1,-1\} \\
    x & \mapsto
    \begin{cases}
    x & \phi(x) = x \\
    -x & \phi(x) = x^{-1}
    \end{cases}
\end{align*}
Consider a word of the form
\[ w = a^{p}b^{l}a^{q}b^{j} \in I
\]
Here we assume $p,l,q,j > 0$. For $w \in \phi-\mathsf{ConjSL}$, we must check that no possible $\phi$-shifts form a shorter word. Since $\phi$ consists of inversions, letters will always be unable to shuffle after shifts. We see that the only possible words which could be shorter than $w$ are
\[ \phi(a^{q})\phi(b^{j})a^{p}b^{j}
\]
and 
\[ a^{q}b^{j}\phi(a^{p})\phi(b^{l})
\]
In the first case, this word is bigger than $w$ if either
\begin{itemize}
    \item $p > \widehat{\phi}(q)$, or
    \item $p=\widehat{\phi}(q)$, $l \leq \widehat{\phi}(j)$
\end{itemize}
Now if $\phi(q) = q^{-1}$, then the first point will always be true since $p$ is positive, and so we can rewrite this as $p > q$. Also the second case can only be true if $\phi$ fixes $q$ and $j$, and so overall these two conditions become:
\begin{itemize}
    \item $p > q$ or
    \item $p = q$, $l \leq j$
\end{itemize}
For the second case, this word is bigger than $w$ if either of the conditions above hold. \\
This matches the conditions required to use the proof for $F_{2}$. 

\end{proof}

\begin{cor}
Let $\Gamma$ be an asymmetric graph with at least 2 vertices, and let $\RAAG$ be the RAAG defined by $\Gamma$. Let $\psi \in \mathrm{Aut}(\RAAG)$ be length preserving. Then $\psi-\mathsf{ConjSL}(\RAAG)$ is not context-free, and hence
\[ \mathsf{ConjSL}(A_{\phi}) \cong \RAAG \rtimes_{\phi} \mathbb{Z}_{2}
\] 
is not context-free. 
\end{cor}

\begin{proof}
By Proposition \ref{prop:combo len p} and the assumption that $\Gamma$ has no non-trivial graph automorphisms, we can assume $\psi$ is a combination of inversions only (so has order 2). \\
We can also assume there exists at least 2 vertices in $\Gamma$ not connected by an edge, since otherwise $\Gamma \cong K_{n}$ which is not asymmetric. The result then follows by Proposition \ref{prop:inversion not CF}. 
\end{proof}
We can now show a more general result that covers a surprisingly large number of cases.

\begin{thm}(Erdos, Renyi)\\
Almost all graphs are asymmetric.
\end{thm}

\begin{cor}
For almost all RAAGs $\RAAG$, there exists a length preserving automorphism $\phi \in \mathrm{Aut}(\RAAG)$ such that $\phi-\mathsf{ConjSL}(\RAAG)$ is not context-free, and hence the extension $\mathsf{ConjSL}(\G)$ is not context-free. 
\end{cor}
This in a sense answers in the affirmative cases that for any RAAG, there exists an automorphisms $\phi$ such that $\mathsf{ConjSL}(A_{\phi})$ is not context-free. We now turn our attention to the bigger question of whether this result is true for \textbf{any} automorphism. \\

\begin{prop}\label{prop: order 2 cases}
Let $\phi \in \mathrm{Aut}(\RAAG)$ be an order 2 graph automorphism. Suppose $\Gamma$ is labelled in such a way that one of the following holds: 
\begin{enumerate}
    \item There exists two vertices $x, y \in \Gamma$ such that $\phi: x \mapsto y, \; y \mapsto x$, and $x$ and $y$ are not connected by an edge.
    \item There exists two vertices $x,y \in \Gamma$ such that $\phi$ fixes both $x$ and $y$, and $x$ and $y$ are not connected by an edge.
    \item There exists vertices $x,y \in \Gamma$ such that $x<y$, $x$ and $y$ are not connected by an edge, and $\phi(x) > x$ and $\phi(y) > y$ (with respect to shortlex). 
    \item There exists vertices $x,y,z \in \Gamma$ such that 
    \[ \phi: x \mapsto x, \; y \mapsto z, \; z \mapsto y
    \]
    and $[x,y] \neq 1, [x,z] \neq 1$ in $\RAAG$, where ordering is $x<y<z$. 
\end{enumerate}
Then $\phi-\mathsf{ConjSL}(\RAAG)$ is not context-free. 
\end{prop}

\begin{proof} 
1) Take $I = \phi-\mathsf{ConjSL} \cap L$ where $L = x^{+}y^{+}x^{+}$. The proof then follows similar to Proposition \ref{prop:conjSL not CF}.\\
2) Since $x,y$ are fixed and do not commute, we can simply take $L = x^{+}y^{+}x^{+}y^{+}$ and apply the $F_{2}$ case, since $x,y$ and their $\phi$-images generate an $F_{2}$ subgroup. \\
3) Let $L$ be words of the form 
\[ x^{p}y^{l}x^{q}y^{j}
\]
One $\phi$-shift which could make the word shorter is
\[ x^{q}y^{j}\phi^{-1}(x^{p})\phi^{-1}(y^{l})
\]
and so we always require that $p \geq q, l \leq j$. This is precisely the conditions on $I$ for the $F_{2}$ case. We now want to check that there are no other conditions required. \\
The other possible words which could be shorter are words which start with a $\phi$-image, for example
\[ \phi(y^{j})x^{p}y^{l}x^{q} \quad \text{or} \quad \phi(x^{q})\phi(y^{j})x^{p}y^{l}
\]
There could be shuffling, however because of the assumptions that $\phi(x) > x$ and $\phi(y) > y$, these words will always be bigger. The rest of the proof then follows the $F_{2}$ case.\\
4) Let $L$ be words of the form
\[ x^{p}y^{l}x^{q}y^{j}
\]
Note by assumption that any words which are a $\phi$-shift also do not contain any shuffles. Therefore the only shifts which could possibly be smaller are of the form
\[ x^{q}y^{j}\phi(x^{p})\phi(y^{l}) = x^{q}y^{j}x^{p}z^{l}
\]
and 
\[ \phi(x^{q}\phi(y^{j})x^{p}y^{l} = x^{q}z^{j}x^{p}y^{l}
\]
For both of these words to be bigger, we need $p \geq q$ and $l \geq j$. This is precisely the conditions we need to apply the $F_{2}$ case.

\end{proof}

\begin{prop}
Let
\[ A_{\Gamma} = \mathbb{Z} \times F_{3}
\]
Then $\phi-\mathsf{ConjSL}(A_{\Gamma})$ is not context-free for all length preserving $\phi$.
\end{prop}

\begin{proof}
Label graph as follows:
\[\begin{tikzcd}
	& a \\
	& d \\
	c && b
	\arrow[no head, from=2-2, to=1-2]
	\arrow[no head, from=2-2, to=3-1]
	\arrow[no head, from=2-2, to=3-3]
\end{tikzcd}\]
Inversions: see earlier work.\\
Graph automorphism: any $\phi$ must fix the central vertex $d \in \mathbb{Z}$. Hence
\[ Aut(A_{\Gamma}) \cong D_{3} \cong S_{3}
\]
Each of the order 2 elements of $S_{3}$ fall into Case 1 of \cref{prop: order 2 cases}. So we only need to consider the order 3 elements.\\
$(a,b,c)$: Consider a word of the form 
\[ a^{p}b^{q}a^{r}
\]
in $\phi-\mathsf{ConjSL}$. A possible twisted cycle is $a^{r}c^{p}a^{q}$, so $r \leq p$. One can check that the only other possible twisted cycle which could be smaller is $a^{q}c^{r}b^{p}$, and so $q \leq p$. This is precisely the same set up as in \cref{prop:conjSL not CF}, and proof follows similarly. \\
The proof is almost identical for $(a c b)$. This covers all possible length preserving automorphisms. 
\end{proof}

We can also find stronger result when we alter the ordering:
\begin{prop}
For 
\[ A_{\phi} = \RAAG \rtimes_{\phi} \mathbb{Z}_{2}
\]
where $\phi$ is a graph automorphism, $\mathsf{ConjSL}(A_{\phi})$ is not context-free for \textbf{some} ordering.
\end{prop}

\begin{proof}
Choose any two vertices $x, y \in \Gamma$ such that $[x,y] \neq 1$ (at least one non-edge exists, otherwise $\Gamma = K_{n}$). Define an ordering which contains the following:
\[ x < \phi(x) < y < \phi(y)
\]
We now lie in Case 3 of \cref{prop: order 2 cases}. 
\end{proof}
}

\bibliography{conj_lang}
\bibliographystyle{plain}
\uppercase{\footnotesize{Department of Mathematics, University of Manchester M13 9PL, UK and the Heilbronn Institute for Mathematical Research, Bristol, UK}}
\par 
\textit{Email address:} \texttt{gemma.crowe@manchester.ac.uk}

\end{document}